\documentclass[12pt,a4paper]{article}

\usepackage{amsfonts,amsmath,amsthm,amssymb,amsopn}
\usepackage{graphicx}
\usepackage{epstopdf}
\usepackage{algorithmic}
\usepackage{enumitem}

\usepackage{bm,color}
\usepackage{stmaryrd}
\usepackage[T1]{fontenc}
\usepackage{float, caption, subcaption}
\usepackage{booktabs}
\usepackage{diagbox}
\usepackage{multirow}
\usepackage{siunitx}
\usepackage{pifont}
\usepackage{hyperref}
\usepackage{cleveref}
\newcommand{\cmark}{\ding{51}}%
\newcommand{\xmark}{\ding{55}}%

\newcommand\bbR{\mathbb{R}}

\newcommand\bF{\bm{F}}

\newcommand\bv{\bm{v}}

\newcommand\bU{\bm{U}}

\newcommand\bR{\bm{R}}

\newcommand\bJ{\bm{J}}
\newcommand\bD{\bm{D}}

\newcommand\dd{\mathrm{d}}
\newcommand\pd[2]{\frac{\partial {#1}}{\partial {#2}}}

\newcommand\abs[1]{\lvert #1 \rvert}
\newcommand\norm[1]{\left\lVert #1 \right\rVert}

\newcommand\xr{i+\frac12}
\newcommand\xl{i-\frac12}
\newcommand\yr{j+\frac12}
\newcommand\yl{j-\frac12}

\theoremstyle{plain}
\newtheorem{lemma}{Lemma}[section]

\theoremstyle{definition}

\newtheorem{example}{Example}[section]
\newtheorem{proposition}{Proposition}[section]

\theoremstyle{remark}
\newtheorem{remark}{Remark}[section]

\crefname{equation}{}{}
\crefname{figure}{Figure}{Figures}
\crefname{table}{Table}{Tables}
\crefname{example}{Example}{Examples}
\crefname{section}{Section}{Sections}
\renewcommand{\title}{Active flux methods for hyperbolic conservation laws -- flux vector splitting and bound-preservation: Two-dimensional case}

\newcommand{\authorOne}{Junming Duan\footnote{Corresponding author. Institute of Mathematics, University of W\"urzburg, Emil-Fischer-Stra\ss e 40, 97074 W\"urzburg, Germany, junming.duan@uni-wuerzburg.de}}
\newcommand{\authorTwo}{Wasilij Barsukow\footnote{Institut de Math\'ematiques de Bordeaux (IMB), CNRS UMR 5251, University of Bordeaux, 33405 Talence, France, wasilij.barsukow@math.u-bordeaux.fr}}
\newcommand{\authorThree}{Christian Klingenberg\footnote{Institute of Mathematics, University of W\"urzburg, Emil-Fischer-Stra\ss e 40, 97074 W\"urzburg, Germany, christian.klingenberg@uni-wuerzburg.de}}

\begin{document}

\begin{center} \Large
\title

\vspace{1cm}

\date{}
\normalsize

\authorOne, \authorTwo, \authorThree
\end{center}

\begin{abstract}

{\color{blue}A more elaborate version, based on this preprint and arXiv:2405.02447 can be found as arXiv:2411.00065.}

This paper studies the active flux (AF) methods for two-dimensional hyperbolic conservation laws, focusing on the flux vector splitting (FVS) for the point value update and bound-preserving (BP) limitings, which is an extension of our previous work [J.M. Duan, W. Barsukow, C. Klingenberg, arXiv:2405.02447].
The FVS-based point value update is shown to address the mesh alignment issue that appeared in a quasi-2D Riemann problem along one axis direction on Cartesian meshes.
Consequently, the AF methods based on the FVS outperform those using Jacobian splitting, which are prone to transonic and mesh alignment issues.
A shock sensor-based limiting is proposed to enhance the convex limiting for the cell average, which can reduce oscillations well.
Some benchmark problems are tested to verify the accuracy, BP property, and shock-capturing ability of our BP AF method.
Moreover, for the double Mach reflection and forward-facing step problems, the present AF method can capture comparable or better small-scale features compared to the third-order discontinuous Galerkin method with the TVB limiter on the same mesh resolution, while using fewer degrees of freedom, demonstrating the efficiency and potential of our BP AF method for high Mach number flows.



\end{abstract}

\section{Introduction}\label{sec:introduction}
This paper focuses on the active flux (AF) method for hyperbolic conservation laws, which is a new type of finite volume method \cite{Eymann_2011_Active_InCollection, Eymann_2011_Active_InProceedings,Eymann_2013_Multidimensional_InCollection,Roe_2017_Is_JoSC},
inspired by \cite{VanLeer_1977_Towards_JoCP}.
The AF method evolves the cell averages and additional degrees of freedom located at the cell interfaces, known as point values.
The original AF method is third-order accurate with a piecewise quadratic compact reconstruction, leading to a global continuous representation of the numerical solution.
Unlike usual Godunov methods, the AF method is free from Riemann solvers, since the numerical flux for the cell average update can be obtained directly, thanks to the continuity of the numerical solution across the cell interface.

Introducing the point value with independent evolution allows more flexibility in the AF methods.
The original one does not require time integration methods by construction, employing exact or approximate evolution operators and Simpson's rule for flux quadrature in time.
Exact evolution operators have been derived for linear equations in \cite{Barsukow_2019_Active_JoSC,Fan_2015_Investigations_InCollection,Eymann_2013_Multidimensional_InCollection, VanLeer_1977_Towards_JoCP}.
Approximate evolution operators have been studied for Burgers' equation \cite{Eymann_2011_Active_InCollection,Eymann_2011_Active_InProceedings,Roe_2017_Is_JoSC,Barsukow_2021_active_JoSC},
the compressible Euler equations in one spatial dimension \cite{Eymann_2011_Active_InCollection,Helzel_2019_New_JoSC,Barsukow_2021_active_JoSC},
and hyperbolic balance laws \cite{Barsukow_2021_Active_SJoSC, Barsukow_2023_Well_CoAMaC}, etc.
The AF method stands out from standard finite volume methods because of its structure-preserving property.
For example, it preserves the vorticity and stationary states for multi-dimensional acoustic equations \cite{Barsukow_2019_Active_JoSC}, and is well-balanced for acoustics with gravity \cite{Barsukow_2021_Active_SJoSC}.

For nonlinear systems, especially in multiple spatial dimensions, it might not be easy to seek exact or approximate evolution operators, leading to the development of the so-called generalized AF method, proposed in \cite{Abgrall_2023_Combination_CoAMaC, Abgrall_2023_Extensions_EMMaNA, Abgrall_2023_active}.
Using the method of lines, the evolution of the cell average and point value is first discretized in space, and then integrated in time using Runge-Kutta methods, for instance.
Two kinds of point value updates have been considered.
The first was proposed in \cite{Abgrall_2023_Combination_CoAMaC, Abgrall_2023_Extensions_EMMaNA}, named Jacobian splitting (JS), which splits the Jacobian matrix based on the sign of the eigenvalues, and then employs upwind-biased stencils to approximate the spatial derivatives.
There are some deficiencies of using the JS for the AF methods, e.g., the transonic issue \cite{Helzel_2019_New_JoSC} for nonlinear problems, leading to spikes in the cell average; and the mesh alignment issue to be shown in \Cref{ex:2d_sod}, where large errors appear in the numerical solution at the point values.
We proposed to employ the flux vector splitting (FVS) for the point value update in \cite{Duan_2024_Active_SJoSC}, which is originally used to identify the upwind directions for solving hyperbolic systems \cite{Toro_2009_Riemann}.
The FVS addresses the transonic issue by borrowing information from the neighbors naturally and uniformly, and we will also show that it is more robust than the JS for the mesh alignment issue.

This paper also develops a bound-preserving (BP) AF method in 2D case, i.e., one that guarantees that the numerical solutions at a later time will stay in a so-called \emph{admissible state set} $\mathcal{G}$, if the initial numerical solutions belong to $\mathcal{G}$.
Consider systems of hyperbolic conservation laws
\begin{equation}\label{eq:2d_hcl}
	\pd{\bU(x, t)}{t} + \pd{\bF_1(\bU)}{x} + \pd{\bF_2(\bU)}{y} = 0,\quad
	\bU(x,y,0) = \bU_0(x,y),
\end{equation}
where $\bU\in \bbR^{m}$ is the vector of $m$ conservative variables,
$\bF_1, \bF_2\in \bbR^{m}$ are the $x$- and $y$-directional physical fluxes,
and $\bU_0(x)$ is assumed to be bounded-variation initial data.
This paper deals with two cases.
The first is a scalar conservation law ($m=1$)
\begin{equation}\label{eq:2d_scalar}
	\pd{u}{t} + \pd{f_1(u)}{x} + \pd{f_2(u)}{y} = 0.
\end{equation}
The solutions to initial value problems of \eqref{eq:2d_scalar} satisfy a strict maximum principle (MP) \cite{Dafermos_2000_Hyperbolic_book},
\begin{equation}\label{eq:2d_scalar_g}
	\mathcal{G} = \left\{ u ~|~ m_0 \leqslant u \leqslant M_0 \right\},
	\quad m_0 = \min_{x, y} u_0(x, y), ~M_0 = \max_{x, y} u_0(x, y).
\end{equation}
The second case is that of compressible Euler equations with $\bU=(\rho, \rho\bv, E)^\top$
and $\bF_1=(\rho v_1, \rho v_1^2 + p, \rho v_1v_2, (E+p)v_1)^\top$, $\bF_2=(\rho v_2, \rho v_1v_2, \rho v_2^2 + p, (E+p)v_2)^\top$, i.e.,
\begin{equation}\label{eq:2d_euler}
	\begin{aligned}
		\dfrac{\partial}{\partial t}\begin{pmatrix}
			\rho\\ \rho v_1\\ \rho v_2\\ E \\
		\end{pmatrix}
		+ \dfrac{\partial}{\partial x}\begin{pmatrix}
			\rho v_1\\ \rho v_1^2 + p\\ \rho v_1v_2 \\ (E + p)v_1 \\
		\end{pmatrix}
		+ \dfrac{\partial}{\partial y}\begin{pmatrix}
			\rho v_2\\ \rho v_1v_2\\ \rho v_2^2 + p \\ (E + p)v_2 \\
		\end{pmatrix}
		 = \bm{0}.
	\end{aligned}
\end{equation}
Here $\rho$ denotes the density, $\bv = (v_1, v_2)$ the velocity,
$p$ the pressure, and $E=\frac12\rho \abs{\bv}^2 + \rho e$ the total energy with $e$ the specific internal energy.
The perfect gas equation of state (EOS) $p = (\gamma-1)\rho e$ is used to close the system \eqref{eq:2d_euler}, with the adiabatic index $\gamma > 1$.
Physically, both the density and pressure in the solutions to \eqref{eq:2d_euler} should stay positive, i.e.,
\begin{equation}\label{eq:2d_euler_g}
	\mathcal{G} = \left\{\bU = \left(\rho, \rho\bv, E\right) ~\Big|~ \rho > 0,~ p = (\gamma-1)\left(E - \frac{\norm{\rho\bv}_2^2}{2\rho}\right) > 0 \right\}.
\end{equation}
Throughout this paper, it is assumed that $\mathcal{G}$ is a \emph{convex} set,
which is obvious for the scalar case \eqref{eq:2d_scalar_g}
and can be verified for the Euler equations \eqref{eq:2d_euler_g}, see e.g. \cite{Zhang_2011_Positivity_JoCP}.

The BP property of numerical methods plays a significant role in both theoretical analysis and numerical stability.
For instance, negative density or pressure for the compressible Euler equations causes loss of hyperbolicity and nonphysical solutions, which may lead to a crash of the simulation.
In the past few decades, different kinds of BP methods have been developed, e.g., a series of works by Shu and collaborators \cite{Zhang_2011_Maximum_PotRSAMPaES, Hu_2013_Positivity_JoCP, Xu_2014_Parametrized_MoC},
a recent general framework on the analysis of BP methods \cite{Wu_2023_Geometric_SR},
and the convex limiting approach \cite{Guermond_2018_Second_SJoSC, Hajduk_2021_Monolithic_C&MwA, Kuzmin_2020_Monolithic_CMiAMaE},
which may be traced back to the flux-corrected transport (FCT) schemes for scalar conservation laws \cite{Cotter_2016_Embedded_JoCP, Guermond_2017_Invariant_SJoNA, Lohmann_2017_Flux_JoCP, Kuzmin_2012_Flux_book}.
For the AF methods, some efforts have been made on the limiting for the point value update, see e.g. \cite{Barsukow_2021_active_JoSC, Chudzik_2021_Cartesian_AMaC}, however, those limitings cannot guarantee the BP property, and they are not enough for high Mach number flows or problems with strong discontinuities.
In a very recent paper, a stabilization using the MOOD \cite{Clain_2011_high_JoCP} approach was adopted to achieve the BP property \cite{Abgrall_2023_Activea} in an a posteriori fashion,
and our recent BP AF methods \cite{Duan_2024_Active_SJoSC} based on the convex limiting and scaling limiter have been shown to work well for challenging 1D problems.

This paper extends our previous work \cite{Duan_2024_Active_SJoSC} to the 2D case, and the main contributions of this work are summarized as follows.
\begin{enumerate}[label=\roman{enumi})., wide=0pt, nosep]
	\item We show that the FVS-based point value update not only provides a remedy to the transonic issue \cite{Duan_2024_Active_SJoSC} but also addresses the mesh alignment issue, which appears in a quasi-2D Riemann problem along one axis direction on Cartesian meshes.
        For the numerical tests involving high Mach number flows and strong discontinuities, we observe that the FVS-based AF methods are robust and show good resolution for small-scale features.
	\item We extend our BP limitings for both the update of the cell average and point value to the 2D case, where the convex limiting and scaling limiter are applied, respectively.
	The high-order AF methods are blended with first-order LLF methods in a convex combination fashion, and the blending coefficients are computed by enforcing certain bounds.
	We show that under a suitable time step size and using BP limitings, the numerical solutions of the BP AF methods satisfy the MP for scalar conservation laws, and maintain positive density and pressure for the compressible Euler equations.
	\item We design a shock sensor-based limiting for the cell average, which helps to reduce oscillations by detecting shock strength, thus improving the shock-capturing ability essentially, illustrated by some benchmark tests.
	Moreover, for the double Mach reflection and forward-facing step problems, our BP AF method captures comparable or better small-scale features compared to the third-order DG method with the TVB limiter on the same mesh resolution, while using fewer degrees of freedom, demonstrating its efficiency and potential for high Mach number flows.
\end{enumerate}

The remainder of this paper is organized as follows.
\Cref{sec:2d_af_schemes} derives the AF methods based on the FVS for the point value update.
To seek BP methods, \Cref{sec:2d_limiting_average} presents our convex limiting approach for the cell average, including shock sensor-based limiting for the compressible Euler equations,
and \Cref{sec:2d_limiting_point} discusses the limiting for the point value.
Some numerical tests are conducted in  \Cref{sec:results} to experimentally demonstrate the accuracy, BP properties, and shock-capturing ability of the methods.
\Cref{sec:conclusion} concludes the paper with final remarks.

\section{2D active flux methods on Cartesian meshes}\label{sec:2d_af_schemes}
This section presents the 2D generalized active flux (AF) methods for the hyperbolic conservation laws \eqref{eq:2d_hcl},
based on the flux vector splitting (FVS) for the point value update.
The fully-discrete methods are obtained using the Runge-Kutta methods.

Assume that a 2D computational domain is divided into $N_1\times N_2$ cells,
$I_{i,j} = [x_{\xl}, x_{\xr}]\times[y_{\yl}, y_{\yr}]$ with the cell centers $(x_i, y_j) = \left(\frac{x_{\xl} + x_{\xr}}{2}, \frac{y_{\yl} + y_{\yr}}{2}\right)$ and cell sizes $\Delta x_i = x_{\xr}-x_{\xl}, \Delta y_j = y_{\yr}-y_{\yl}$, $i=1,\cdots,N_1,~ j=1,\cdots,N_2$.
In the AF methods, the degrees of freedom (DoFs) are the approximations of cell averages of the conservative variable as well as point values at the cell interfaces, where the former is used to guarantee the conservation.
The point values can be chosen as conservative variables, primitive variables, entropy variables, etc., which illustrates the flexibility of the AF methods.
This paper considers using the conservative variables for the point values, and the DoFs consist of the cell average
\begin{equation*}
    \overline{\bU}_{i,j}(t) = \dfrac{1}{\Delta x_i\Delta y_j}\int_{I_{i,j}} \bU_h(x,y,t) ~\dd x\dd y,
\end{equation*}
the face-centered values
\begin{equation*}
    \bU_{i+\frac12,j}(t) = \bU_h(x_{i+\frac12}, y_j, t),~
\bU_{i,j+\frac12}(t) = \bU_h(x_{i}, y_{j+\frac12}, t),
\end{equation*}
and the nodal value
\begin{equation*}
    \bU_{i+\frac12,j+\frac12}(t) = \bU_h(x_{i+\frac12}, y_{j+\frac12}, t),
\end{equation*}
where $\bU_h(x,y,t)$ is the numerical solution.
A sketch of the degrees of freedom (DoFs) for the third-order AF method (for the scalar case) is given in \Cref{fig:2d_af_dofs}.

\begin{figure}[hptb!]
	\centering
	\includegraphics[width=0.4\linewidth]{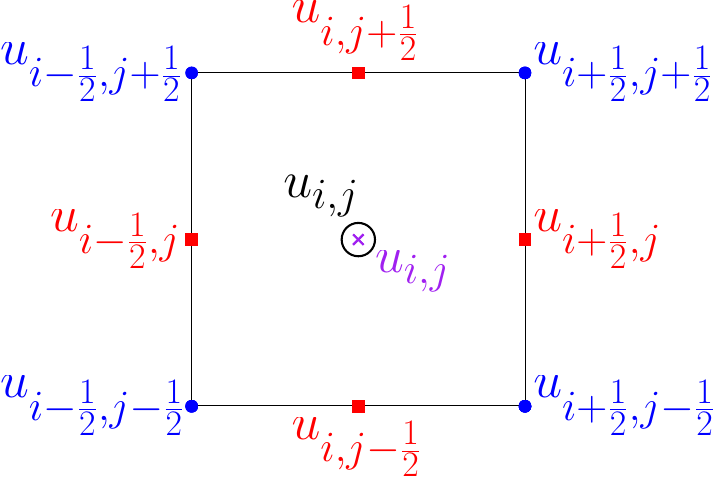}
	\caption{The DoFs for the third-order AF method: cell average (circle), face-centered value (squares), nodal value (dots). Note that the cell-centered point value $u_{i,j}$ (cross) is used in constructing the schemes, but does not belong to the DoFs.}
	\label{fig:2d_af_dofs}
\end{figure}

The update of the cell average follows the usual finite volume method by integrating \eqref{eq:2d_hcl} over $I_{i,j}$ and introducing numerical fluxes
\begin{equation}\label{eq:semi_av_2d}
	\frac{\dd \overline{\bU}_{i,j}}{\dd t} =
	-\frac{1}{\Delta x_i}\left(\widehat{\bF}_{\xr,j} - \widehat{\bF}_{\xl,j}\right)
	-\frac{1}{\Delta y_j}\left(\widehat{\bF}_{i,\yr} - \widehat{\bF}_{i,\yl}\right),
\end{equation}
where $\widehat{\bF}_{\xr,j}$ and $\widehat{\bF}_{i,\yr}$ are the numerical fluxes approximating the integral at the cell interfaces
\begin{equation}\label{eq:2d_num_fluxes}
	\widehat{\bF}_{\xr,j} = \frac{1}{\Delta y_j}\int_{y_{\yl}}^{y_{\yr}} \bF_1(\bU_h(x_{\xr},y))~\dd y,~
	\widehat{\bF}_{i,\yr} = \frac{1}{\Delta x_i}\int_{x_{\xl}}^{x_{\xr}} \bF_2(\bU_h(x,y_{\yr}))~\dd x.
\end{equation}
The accuracy of \eqref{eq:semi_av_2d} is determined by the approximation accuracy of the integral and the point values.
For third-order accuracy, the numerical fluxes can be discretized by Simpson's rule
\begin{equation*}
	\widehat{\bF}_{\xr,j} = \frac16\left(\bF_1(\bU_{\xr,\yl}) + 4\bF_1(\bU_{\xr,j}) + \bF_1(\bU_{\xr,\yr})\right).
\end{equation*}

For the evolution of the point values, this paper considers the following general form
\begin{equation}\label{eq:semi_pnt_2d}
	\frac{\dd \bU_{\sigma}}{\dd t} = - \bm{\mathcal{R}}\left(\overline{\bU}_c(t), \bU_{\sigma'}(t)\right),~c\in\mathcal{C}(\sigma), \sigma'\in\Sigma(\sigma),
\end{equation}
where $\bm{\mathcal{R}}$ is a consistent approximation of $\partial\bF_1/\partial x + \partial\bF_2/\partial y$ at the point $\sigma$, $\mathcal{C}(\sigma)$ and $\Sigma(\sigma)$ are the spatial stencils containing the cell averages and point values, respectively.
The discretization for the point value update is essential to achieve stability, and it is natural to incorporate the upwind idea for hyperbolic systems.
The point value update based on the Jacobian splitting (JS) has been studied in \cite{Abgrall_2023_Extensions_EMMaNA,Abgrall_2023_active}, however, it suffers from the transonic issue for nonlinear problems \cite{Helzel_2019_New_JoSC, Barsukow_2021_active_JoSC}, and also a mesh alignment issue to be shown in \Cref{ex:2d_sod}.
In \cite{Duan_2024_Active_SJoSC}, we proposed using the FVS for the point value update, which borrows the information from the neighbors naturally and maintains the original continuous reconstruction.
The FVS-based point value update has been shown to cure the transonic issue in \cite{Duan_2024_Active_SJoSC}, and it also mitigates the mesh alignment issue.
The discretizations using the JS can be found in \cite{Abgrall_2023_active} in detail, so this paper only presents the FVS-based point value update in 2D.

\subsection{Point value update using flux vector splitting}\label{sec:2d_fvs}
For the point value at the node $(x_{\xr}, y_{\yr})$, the FVS-based update reads
\begin{equation}\label{eq:2d_semi_fvs_node}
	\dfrac{\dd \bU_{\xr,\yr}}{\dd t} 
	= - \sum_{\ell=1}^2 \left[\bD_\ell^+\bF_\ell^{+}(\bU)+ \bD_\ell^-\bF_\ell^{-}(\bU) \right]_{\xr,\yr}
\end{equation}
where the fluxes are split as
\begin{equation}\label{eq:FVS_condition}
	 \bF_\ell= \bF_\ell^{+} + \bF_\ell^{-},\quad
	\lambda\left(\frac{\partial\bF_\ell^{+}}{\partial\bU}\right) \geqslant 0, \quad \lambda\left(\frac{\partial\bF_\ell^{-}}{\partial\bU}\right) \leqslant 0,
\end{equation}
i.e., the eigenvalues of ${\partial\bF_\ell^{+}}/{\partial\bU}$ and ${\partial\bF_\ell^{-}}/{\partial\bU}$ are non-negative and non-positive, respectively.
Different FVS can be employed if they satisfy the constraint \eqref{eq:FVS_condition}, which will be discussed later.
The finite difference operators $\bD_1^{\pm}$ and $\bD_2^{\pm}$ to approximate the flux derivatives are defined in a dimension-by-dimension fashion.
For third-order accuracy, they are similar to those in 1D \cite{Duan_2024_Active_SJoSC}. To be specific,
\begin{subequations}
	\begin{align*}
		\left(\bD_1^{+}\bF_1^+\right)_{\xr,\yr} &= \dfrac{1}{\Delta x_i}\left((\bF_1)_{\xl,\yr}^{+} - 4 (\bF_1)_{i,\yr}^{+} + 3 (\bF_1)_{\xr,\yr}^{+} \right), \\
		\left(\bD_1^{-}\bF_1^-\right)_{\xr,\yr} &= \dfrac{1}{\Delta x_{i+1}}\left(- 3 (\bF_1)_{\xr,\yr}^{-} + 4 (\bF_1)_{i+1,\yr}^{-} - (\bF_1)_{i+\frac32,\yr}^{-} \right), \\
		\left(\bD_2^{+}\bF_2^+\right)_{\xr,\yr} &= \dfrac{1}{\Delta y_j}\left((\bF_2)_{\xr,\yl}^{+} - 4 (\bF_2)_{\xr,j}^{+} + 3 (\bF_2)_{\xr,\yr}^{+} \right), \\
		\left(\bD_2^{-}\bF_2^-\right)_{\xr,\yr} &= \dfrac{1}{\Delta y_{j+1}}\left(- 3 (\bF_2)_{\xr,\yr}^{-} + 4 (\bF_2)_{\xr,j+1}^{-} - (\bF_2)_{\xr,j+\frac32}^{-} \right),
	\end{align*}
\end{subequations}
where the flux point value is obtained by directly evaluating the flux using the corresponding point value.

For the face-centered point value at $(x_{\xr}, y_{j})$, the FVS-based update reads
\begin{equation}\label{eq:2d_semi_fvs_facex}
	\dfrac{\dd \bU_{\xr,j}}{\dd t} 
	= - \left[\bD_1^+\bF_1^{+}(\bU)+ \bD_1^-\bF_1^{-}(\bU) \right]_{\xr,j}
	- \left(\bD_2\bF_2(\bU)\right)_{\xr,j},
\end{equation}
with
\begin{subequations}
	\begin{align*}
		\left(\bD_1^{+}\bF_1^+\right)_{\xr,j} &= \dfrac{1}{\Delta x_i}\left((\bF_1)_{\xl,j}^{+} - 4 (\bF_1)_{i,j}^{+} + 3 (\bF_1)_{\xr,j}^{+} \right), \\
		\left(\bD_1^{-}\bF_1^-\right)_{\xr,j} &= \dfrac{1}{\Delta x_{i+1}}\left(- 3 (\bF_1)_{\xr,j}^{-} + 4 (\bF_1)_{i+1,j}^{-} - (\bF_1)_{i+\frac32,j}^{-} \right), \\
		\left(\bD_2\bF_2\right)_{\xr,j} &= \dfrac{1}{\Delta y_j}\left((\bF_2)_{\xr,\yr} - (\bF_2)_{\xr,\yl} \right),
	\end{align*}
\end{subequations}
and the point value at the cell center is computed from the parabolic reconstruction as
\begin{align*}
	\bU_{i,j} = \frac{1}{16}\Big[
		36\overline{\bU}_{i,j} &- 4\left(\bU_{\xl,j}+\bU_{\xr,j}+\bU_{i,\yl}+\bU_{i,\yr}\right)\\
		&- \left(\bU_{\xl,\yl} + \bU_{\xr,\yl} + \bU_{\xl,\yr} + \bU_{\xr,\yr}\right)
	\Big].
\end{align*}
Note that the solution is continuous at the cell interface $x=x_{\xr}, y_{\yl}\leqslant y \leqslant y_{\yr}$ and for the compactness of the approximation, the central finite difference is used for the $y$-directional approximation.

The update for the face-centered point value at $(x_{i}, y_{\yr})$ is omitted here to save space, which is similar to \eqref{eq:2d_semi_fvs_facex}.
The remaining task is to define the FVS.

\subsubsection{Local Lax-Friedrichs flux vector splitting}
Consider the LLF FVS
\begin{equation*}
	\bF_\ell^\pm = \frac12(\bF_\ell(\bU) \pm \alpha_\ell \bU),
\end{equation*}
where the choice of $\alpha_\ell$ should satisfy the condition \eqref{eq:FVS_condition} across the spatial stencil.
In our implementation, it is determined by
\begin{align}\label{eq:2d_Rusanov_alpha}
	(\alpha_1)_{\xr,q} &= \max_{s} \left\{\left|\varrho_{1}(\bU_{s,q})\right|\right\},
	~s \in \left\{ i-\frac12, i, i+\frac12, i+1, i+\frac32\right\}, q=j,j+\frac12, \\
	(\alpha_2)_{q,\yr} &= \max_{s} \left\{\left|\varrho_{2}(\bU_{q,s})\right|\right\},
	~s \in \left\{ j-\frac12, j, j+\frac12, j+1, j+\frac32\right\}, q=i,i+\frac12, \nonumber
\end{align}
where $\varrho_\ell$ is the spectral radius of the Jacobian matrix ${\partial{\bF_\ell}}/{\partial\bU}$.

\subsubsection{Upwind flux vector splitting}
The flux can also be split based on each characteristic field as follows
\begin{equation}\label{eq:upwind_fvs}
	\bF_\ell^\pm = \frac12(\bF_\ell(\bU)\pm \abs{\bJ_\ell} \bU),\quad
	\abs{\bJ_\ell} = \bR_\ell(\bm{\Lambda}_\ell^{+} - \bm{\Lambda}_\ell^{-})\bR_\ell^{-1},
\end{equation}
where $\bJ_\ell={\partial\bF_\ell}/{\partial\bU} = \bR_\ell\bm{\Lambda}_\ell\bR_\ell^{-1}$ the eigen-decomposition of the Jacobian matrix.
Note that for linear systems, the FVS \eqref{eq:upwind_fvs} reduces to the JS \cite{Duan_2024_Active_SJoSC}.

Such an FVS is also called the Steger-Warming (SW) FVS \cite{Steger_1981_Flux_JoCP} for the Euler equations \eqref{eq:2d_euler},
and the explicit expressions of the SW FVS in the $x$-direction are
\begin{align*}
	\bF_1^{\pm} &= \begin{bmatrix}
		\frac{\rho}{2\gamma}\alpha^\pm \\
		\frac{\rho}{2\gamma}\left(\alpha^\pm v_1 + a (\lambda_2^\pm - \lambda_3^\pm)\right) \\
		\frac{\rho}{2\gamma}\alpha^\pm v_2 \\
		\frac{\rho}{2\gamma}\left(\frac12\alpha^\pm \norm{\bv}_2^2 + a v_1 (\lambda_2^\pm - \lambda_3^\pm) + \frac{a^2}{\gamma-1}(\lambda_2^\pm + \lambda_3^\pm)\right) \\
	\end{bmatrix},
\end{align*}
where
$\lambda_1 = v_\ell, ~\lambda_2 = v_\ell + a, ~\lambda_3 = v_\ell - a,
~
\alpha^\pm = 2(\gamma-1)\lambda_1^\pm + \lambda_2^\pm + \lambda_3^\pm$,
and $a = \sqrt{\gamma p / \rho}$ is the sound speed.
The expressions in the $y$-direction can be obtained using the rotational invariance.

\subsubsection{Van Leer-H\"anel flux vector splitting for the Euler equations}
Another well-known FVS for the Euler equations was developed by \cite{Leer_1982_Flux_InProceedings},
and improved in \cite{Haenel_1987_accuracy_InCollection}.
The flux is split according to the Mach number $M=v_1/a$ as
\begin{equation*}
	\bF_1 = \begin{bmatrix}
		\rho a M \\
		\rho a^2(M^2 + \frac{1}{\gamma}) \\
		\rho a M v_2 \\
		\rho a^3M(\frac{1}{2}M^2 + \frac{1}{\gamma-1}) + \frac{\rho a M v_2^2}{2} \\
	\end{bmatrix} = \bF_1^+ + \bF_1^-,~
	\bF_1^{\pm} = \begin{bmatrix}
		\pm\frac{1}{4}\rho a (M\pm1)^2 \\
		\pm\frac{1}{4}\rho a (M\pm1)^2 v_1 + p^{\pm} \\
		\pm\frac{1}{4}\rho a (M\pm1)^2 v_2 \\
		\pm\frac{1}{4}\rho a (M\pm1)^2 H \\
	\end{bmatrix}
\end{equation*}
with the enthalpy $H=(E+p)/\rho$,
and the pressure-splitting
$p^{\pm} = \frac{1}{2}(1\pm\gamma M)p$.
The VH FVS is quadratic differentiable with respect to the Mach number.

\subsection{Time discretization}
The fully-discrete scheme is obtained by using the SSP-RK3 method \cite{Gottlieb_2001_Strong_SRa}
\begin{equation}\label{eq:ssp_rk3}
	\begin{aligned}
		\bU^{*} &= \bU^{n}+\Delta t^n \bm{L}\left(\bU^{n}\right),\\
		\bU^{**} &= \frac{3}{4}\bU^{n}+\frac{1}{4}\left(\bU^{*}+\Delta t^n \bm{L}\left(\bU^{*}\right)\right),\\
		\bU^{n+1} &= \frac{1}{3}\bU^{n}+\frac{2}{3}\left(\bU^{**}+\Delta t^n \bm{L}\left(\bU^{**}\right)\right),
	\end{aligned}
\end{equation}
where $\bm{L}$ is the right-hand side of the semi-discrete schemes \eqref{eq:semi_av_2d} or  \eqref{eq:semi_pnt_2d}.
The time step size is determined by the usual CFL condition
\begin{equation*}
	\Delta t^n = \dfrac{C_{\texttt{CFL}}}{\max\limits_{i,j}\{\max\{\varrho_1(\overline{\bU}_{i,j})/\Delta x_i,~  \varrho_2(\overline{\bU}_{i,j})/\Delta y_j\}\}}.
\end{equation*}
\section{Convex limiting for cell average}\label{sec:2d_limiting_average}

This section extends our BP limitings \cite{Duan_2024_Active_SJoSC} for the cell average based on the convex limiting approach \cite{Guermond_2018_Second_SJoSC, Hajduk_2021_Monolithic_C&MwA, Kuzmin_2020_Monolithic_CMiAMaE} to the 2D case.
By blending high-order numerical fluxes in the AF methods with first-order LLF fluxes, the limited cell average update can be shown to preserve certain bounds, since it can be written as a convex combination of the numerical solutions at the previous time step and some intermediate states staying in the convex admissible state set $\mathcal{G}$.

Consider the following first-order LLF scheme
\begin{align*}
	&\overline{\bU}^{\texttt{L}}_{i,j} = \overline{\bU}^{n}_{i,j}
	- \mu_{1,i}\left(\widehat{\bF}^{\texttt{L}}_{\xr,j} - \widehat{\bF}^{\texttt{L}}_{\xl,j}\right)
	- \mu_{2,j}\left(\widehat{\bF}^{\texttt{L}}_{i,\yr} - \widehat{\bF}^{\texttt{L}}_{i,\yl}\right), \\
	&\widehat{\bF}^{\texttt{L}}_{\xr,j} = \frac12\left(\bF_1(\overline{\bU}^{n}_{i,j}) + \bF_1(\overline{\bU}^{n}_{i+1,j})\right) - \frac{(\alpha_1)_{\xr,j}}{2}\left(\overline{\bU}^{n}_{i+1,j} - \overline{\bU}^{n}_{i,j}\right), \\
	&\widehat{\bF}^{\texttt{L}}_{i,\yr} = \frac12\left(\bF_2(\overline{\bU}^{n}_{i,j}) + \bF_2(\overline{\bU}^{n}_{i,j+1})\right) - \frac{(\alpha_2)_{i,\yr}}{2}\left(\overline{\bU}^{n}_{i,j+1} - \overline{\bU}^{n}_{i,j}\right),\\
	&(\alpha_1)_{\xr,j} = \max\{\varrho_1(\overline{\bU}^{n}_{i,j}), ~\varrho_1(\overline{\bU}^{n}_{i+1,j})\},~(\alpha_2)_{i,\yr} = \max\{\alpha_2(\overline{\bU}^{n}_{i,j}), ~\alpha_2(\overline{\bU}^{n}_{i,j+1})\}, \\
	&\mu_{1,i} = \Delta t^n / \Delta x_i, ~ \mu_{2,j} = \Delta t^n / \Delta y_j.
\end{align*}
Note that here $(\alpha_1)_{\xr,j}$ is not the same as the one in the LLF FVS \eqref{eq:2d_Rusanov_alpha}.
Following \cite{Guermond_2016_Invariant_SJoNA}, rewrite the first-order LLF scheme as
\begin{align}\label{eq:2d_lo_decomp}
	\overline{\bU}^{\texttt{L}}_{i,j} =& \left[1-\mu_{1,i}\left((\alpha_1)_{\xl,j}+(\alpha_1)_{\xr,j}\right)
	-\mu_{2,j}\left((\alpha_2)_{i,\yl}+(\alpha_2)_{i,\yr}\right)\right]
	\overline{\bU}^{n}_{i,j} \nonumber\\
	&+ \mu_{1,i}(\alpha_1)_{\xl,j}\widetilde{\bU}_{\xl,j} + \mu_{1,i}(\alpha_1)_{\xr,j}\widetilde{\bU}_{\xr,j} \nonumber\\
	&+ \mu_{2,j}(\alpha_2)_{i,\yl}\widetilde{\bU}_{i,\yl} + \mu_{2,j}(\alpha_2)_{i,\yr}\widetilde{\bU}_{i,\yr},
\end{align}
with the intermediate states
\begin{equation}\label{eq:2d_llf_inter_states}
	\begin{aligned}
	&\widetilde{\bU}_{\xl,j} = \frac12\left(\overline{\bU}^{n}_{i-1,j} + \overline{\bU}^{n}_{i,j}\right)
	+ \frac{1}{2(\alpha_1)_{\xl,j}}\left[\bF_1(\overline{\bU}^{n}_{i-1,j}) - \bF_1(\overline{\bU}^{n}_{i,j})\right], \\
	&\widetilde{\bU}_{\xr,j} = \frac12\left(\overline{\bU}^{n}_{i,j} + \overline{\bU}^{n}_{i+1,j}\right)
	+ \frac{1}{2(\alpha_1)_{\xr,j}}\left[\bF_1(\overline{\bU}^{n}_{i,j}) - \bF_1(\overline{\bU}^{n}_{i+1,j})\right], \\
	&\widetilde{\bU}_{i,\yl} = \frac12\left(\overline{\bU}^{n}_{i,j-1} + \overline{\bU}^{n}_{i,j}\right)
	+ \frac{1}{2(\alpha_2)_{i,\yl}}\left[\bF_2(\overline{\bU}^{n}_{i,j-1}) - \bF_2(\overline{\bU}^{n}_{i,j})\right], \\
	&\widetilde{\bU}_{i,\yr} = \frac12\left(\overline{\bU}^{n}_{i,j} + \overline{\bU}^{n}_{i,j+1}\right)
	+ \frac{1}{2(\alpha_2)_{i,\yr}}\left[\bF_2(\overline{\bU}^{n}_{i,j}) - \bF_2(\overline{\bU}^{n}_{i,j+1})\right].
	\end{aligned}
\end{equation}

\begin{lemma}\label{lem:lf_scalar}
    For the scalar conservation laws \eqref{eq:2d_scalar}, the intermediate state $\widetilde{u}=\frac12(u_L + u_R) + \frac{1}{2\alpha}(f_\ell(u_L) - f_\ell(u_R))$ stays in $\mathcal{G}$ \eqref{eq:2d_scalar_g} if $\alpha\geqslant\max\{\varrho_\ell(u_L), \varrho_\ell(u_R)\}$.
\end{lemma}

\begin{proof}
    The partial derivatives of the intermediate state satisfy
    \begin{equation*}
        \frac{\partial\widetilde{u}(u_L, u_R)}{\partial u_L} = \frac12\left(1 + \frac{f_\ell'(u_L)}{\alpha}\right)\geqslant 0, \quad
        \frac{\partial\widetilde{u}(u_L, u_R)}{\partial u_R} = \frac12\left(1 - \frac{f_\ell'(u_R)}{\alpha}\right)\geqslant 0.
    \end{equation*}
    As $\widetilde{u}(m_0, m_0) = m_0$, $\widetilde{u}(M_0, M_0) = M_0$, it holds $m_0 \leqslant \widetilde{u} \leqslant M_0$.
\end{proof}

\begin{lemma}\label{lem:lf_euler}
    For the Euler equations \eqref{eq:2d_euler}, the intermediate state $\widetilde{\bU}=\frac12(\bU_L + \bU_R) + \frac{1}{2\alpha}(\bF_\ell(\bU_L) - \bF_\ell(\bU_R))$ stays in $\mathcal{G}$ \eqref{eq:2d_euler_g} if $\alpha\geqslant\max\{\varrho_\ell(\bU_L), \varrho_\ell(\bU_R)\}$.
\end{lemma}

\begin{proof}
    For the Euler equations \eqref{eq:2d_euler}, as the intermediate state is a convex combination of ${\bU}_L - \frac{1}{\alpha}\bF_\ell({\bU}_L)$ and ${\bU}_{R} + \frac{1}{\alpha}\bF_\ell({\bU}_R)$, we only need to show that the ${\bU} \pm \frac{1}{\alpha}\bF_\ell({\bU})$ belongs to $\mathcal{G}$.
    The density component $\left(\rho \pm (\rho v_\ell)/{\alpha} \right)$ is positive since $\alpha>\abs{v_\ell}$.
    The recovered internal energy is
    \begin{align*}
        \rho e\left({\bU} \pm \frac{1}{\alpha}\bF_\ell({\bU})\right)
        &= E\left({\bU} \pm \frac{1}{\alpha}\bF_\ell({\bU})\right) - \frac{\norm{\rho \bv\left({\bU} \pm \frac{1}{\alpha}\bF_\ell({\bU})\right)}_2^2}{2\rho\left({\bU} \pm \frac{1}{\alpha}\bF_\ell({\bU})\right)} \\
        &= \left(1 - \frac{p^2}{2(\alpha\pm v_\ell)^2\rho^2e}\right)\left(1\pm\frac{v_\ell}{\alpha}\right)\rho e,
    \end{align*}
    so that one has $\rho e\left({\bU} \pm \frac{1}{\alpha}\bF_\ell({\bU})\right)> 0 \iff \frac{p^2}{2\rho^2e} < (\alpha\pm v_\ell)^2 \iff \frac{\gamma-1}{2\gamma}a^2 < (\alpha\pm v_\ell)^2$ for the perfect gas EOS,
    which holds as $\alpha\geqslant \abs{v_\ell}+a$. 
\end{proof}

\begin{remark}
    Here $\alpha$ is not chosen to be larger than the maximal wave speed, as it is sufficient for the MP for the scalar case and the positivity-preserving property for the Euler equations.
\end{remark}


\begin{lemma}\label{lem:llf_g}
	If the time step size $\Delta t^n$ satisfies
	\begin{equation}\label{eq:2d_convex_combination_dt}
		\Delta t^n \leqslant \frac12\min\left\{\dfrac{\Delta x_i}{(\alpha_1)_{\xl,j}+(\alpha_1)_{\xr,j}}, \dfrac{\Delta y_j}{(\alpha_2)_{i,\yl}+(\alpha_2)_{i,\yr}}\right\},
	\end{equation}
	then \eqref{eq:2d_lo_decomp} is a convex combination,
	and	the first-order LLF scheme \eqref{eq:2d_lo_decomp} is BP.
\end{lemma}

\begin{proof}
    The intermediate states \eqref{eq:2d_llf_inter_states} belongs to $\mathcal{G}$ due to \Cref{lem:lf_scalar} and \ref{lem:lf_euler}.
    Under the constraint \eqref{eq:2d_convex_combination_dt}, the LLF scheme \eqref{eq:2d_lo_decomp} is a convex combination, and it is BP thanks to the convexity of $\mathcal{G}$.
\end{proof}

Denote the anti-diffusion flux by $\Delta \widehat{\bF}_{\xr,j} = \widehat{\bF}^{\texttt{H}}_{\xr,j} - \widehat{\bF}^{\texttt{L}}_{\xr,j}$,
$\Delta \widehat{\bF}_{i,\yr} = \widehat{\bF}^{\texttt{H}}_{i,\yr} - \widehat{\bF}^{\texttt{L}}_{i,\yr}$, with $\widehat{\bF}^{\texttt{H}}_{i\pm\frac12,j}, \widehat{\bF}^{\texttt{H}}_{i, j\pm\frac12}$ given in \eqref{eq:2d_num_fluxes}.
Applying a forward-Euler step to the semi-discrete high-order AF scheme \eqref{eq:semi_av_2d} yields
\begin{align}
	\overline{\bU}^{\texttt{H}}_{i,j} =& \left[1-\mu_{1,i}\left((\alpha_1)_{\xl,j}+(\alpha_1)_{\xr,j}\right)
	-\mu_{2,j}\left((\alpha_2)_{i,\yl}+(\alpha_2)_{i,\yr}\right)\right]
	\overline{\bU}^{n}_{i,j} \nonumber\\
	&+ \mu_{1,i}(\alpha_1)_{\xl,j}\widetilde{\bU}_{\xl,j}^{\texttt{H},+}
	+ \mu_{1,i}(\alpha_1)_{\xr,j}\widetilde{\bU}_{\xr,j}^{\texttt{H},-} \nonumber\\
	&+ \mu_{2,j}(\alpha_2)_{i,\yl}\widetilde{\bU}_{i,\yl}^{\texttt{H},+}
	+ \mu_{2,j}(\alpha_2)_{i,\yr}\widetilde{\bU}_{i,\yr}^{\texttt{H},-} \label{eq:2d_ho_decomp},
\end{align}
with the high-order intermediate states
\begin{align*}
	&\widetilde{\bU}_{\xl,j}^{\texttt{H},+} := \widetilde{\bU}_{\xl,j} + \frac{\Delta\widehat{\bF}_{\xl,j}}{(\alpha_1)_{\xl,j}}, \quad
	\widetilde{\bU}_{\xr,j}^{\texttt{H},-} := \widetilde{\bU}_{\xr,j} - \frac{\Delta\widehat{\bF}_{\xr,j}}{(\alpha_1)_{\xr,j}}, \\
	&\widetilde{\bU}_{i,\yl}^{\texttt{H},+} := \widetilde{\bU}_{i,\yl} + \frac{\Delta\widehat{\bF}_{i,\yl}}{(\alpha_2)_{i,\yl}}, \quad
	\widetilde{\bU}_{i,\yr}^{\texttt{H},-} := \widetilde{\bU}_{i,\yr} - \frac{\Delta\widehat{\bF}_{i,\yr}}{(\alpha_2)_{i,\yr}}.
\end{align*}

Observe that the first-order scheme \eqref{eq:2d_lo_decomp} and high-order scheme \eqref{eq:2d_ho_decomp} share the same abstract form so that one can define a limited scheme as
\begin{align}
	\overline{\bU}^{\texttt{Lim}}_{i,j} =& \left[1-\mu_{1,i}\left((\alpha_1)_{\xl,j}+(\alpha_1)_{\xr,j}\right)
	-\mu_{2,j}\left((\alpha_2)_{i,\yl}+(\alpha_2)_{i,\yr}\right)\right]
	\overline{\bU}^{n}_{i,j} \nonumber\\
	&+ \mu_{1,i}(\alpha_1)_{\xl,j}\widetilde{\bU}_{\xl,j}^{\texttt{Lim},+}
	+ \mu_{1,i}(\alpha_1)_{\xr,j}\widetilde{\bU}_{\xr,j}^{\texttt{Lim},-} \nonumber\\
	&+ \mu_{2,j}(\alpha_2)_{i,\yl}\widetilde{\bU}_{i,\yl}^{\texttt{Lim},+}
	+ \mu_{2,j}(\alpha_2)_{i,\yr}\widetilde{\bU}_{i,\yr}^{\texttt{Lim},-} \label{eq:2d_limited_decomp},
\end{align}
with the following limited intermediate states
\begin{align*}
	&\widetilde{\bU}_{\xl,j}^{\texttt{Lim},+} = \widetilde{\bU}_{\xl,j} + \frac{\Delta\widehat{\bF}^{\texttt{Lim}}_{\xl,j}}{(\alpha_1)_{\xl,j}}
	:= \widetilde{\bU}_{\xl,j} + \frac{\theta_{\xl,j}\Delta\widehat{\bF}_{\xl,j}}{(\alpha_1)_{\xl,j}}, \\
	&\widetilde{\bU}_{\xr,j}^{\texttt{Lim},-} = \widetilde{\bU}_{\xr,j} - \frac{\Delta\widehat{\bF}^{\texttt{Lim}}_{\xr,j}}{(\alpha_1)_{\xr,j}}
	:= \widetilde{\bU}_{\xr,j} - \frac{\theta_{\xr,j}\Delta\widehat{\bF}_{\xr,j}}{(\alpha_1)_{\xr,j}}, \\
	&\widetilde{\bU}_{i,\yl}^{\texttt{Lim},+} = \widetilde{\bU}_{i,\yl} + \frac{\Delta\widehat{\bF}^{\texttt{Lim}}_{i,\yl}}{(\alpha_2)_{i,\yl}}
	:= \widetilde{\bU}_{i,\yl} + \frac{\theta_{i,\yl}\Delta\widehat{\bF}_{i,\yl}}{(\alpha_2)_{i,\yl}}, \\
	&\widetilde{\bU}_{i,\yr}^{\texttt{Lim},-} = \widetilde{\bU}_{i,\yr} - \frac{\Delta\widehat{\bF}^{\texttt{Lim}}_{i,\yr}}{(\alpha_2)_{i,\yr}}
	:= \widetilde{\bU}_{i,\yr} - \frac{\theta_{i,\yr}\Delta\widehat{\bF}_{i,\yr}}{(\alpha_2)_{i,\yr}},
\end{align*}
where the coefficient $\theta_{i\pm \frac12,j}$ and $\theta_{i,j\pm \frac12}$ stay in $[0,1]$.

\begin{proposition}
	If the cell average at the last time step $\overline{\bU}_{i,j}^n$ and the limited intermediate states $\widetilde{\bU}_{i\pm\frac12,j}^{\texttt{Lim},\mp}$, $\widetilde{\bU}_{i,j\pm\frac12}^{\texttt{Lim},\mp}$ belong to the admissible state set $\mathcal{G}$,
	then the limited average update \eqref{eq:2d_limited_decomp} is BP, i.e., $\overline{\bU}^{\texttt{Lim}}_{i,j} \in \mathcal{G}$, under the CFL condition \eqref{eq:2d_convex_combination_dt}.
	If the SSP-RK3 \eqref{eq:ssp_rk3} is used for the time integration,
	the high-order scheme is also BP.	
\end{proposition}

\begin{proof}
	Under the constraint \eqref{eq:2d_convex_combination_dt}, the limited cell average update $\overline{\bU}^{\texttt{Lim}}_{i,j}$ is a convex combination of $\overline{\bU}_{i,j}^n$ and $\widetilde{\bU}_{i\pm\frac12,j}^{\texttt{Lim},\mp}, \widetilde{\bU}_{i,j\pm\frac12}^{\texttt{Lim},\mp}$,
	thus it belongs to $\mathcal{G}$ due to the convexity of $\mathcal{G}$.
	Because the SSP-RK3 is a convex combination of forward-Euler stages,
	the high-order scheme equipped with the SSP-RK3 is also BP according to the convexity.
\end{proof}

\begin{remark}
	The scheme \eqref{eq:2d_limited_decomp} is conservative because it can be rewritten as a conservative scheme with limited numerical fluxes,
	e.g. for the $x$-direction,
	$\widehat{\bF}^{\texttt{L}}_{\xr,j} + \theta_{\xr,j}\Delta \widehat{\bF}_{\xr,j} = \theta_{\xr,j} \widehat{\bF}^{\texttt{H}}_{\xr,j} + (1 - \theta_{\xr,j}) \widehat{\bF}^{\texttt{L}}_{\xr,j}$,
	which is a convex combination of the high-order and low-order fluxes.
\end{remark}

The remaining goal is to determine the largest coefficients $\theta_{i\pm\frac12,j}, \theta_{i,j\pm\frac12}\in[0,1]$ such that $\widetilde{\bU}_{i\pm\frac12,j}^{\texttt{Lim},\mp}, \widetilde{\bU}_{i,j\pm\frac12}^{\texttt{Lim},\mp}\in\mathcal{G}$.
To save space, only the limitings in the $x$-direction are detailed below.

\subsection{Application to scalar conservation laws}
This section is devoted to applying the convex limiting approach to scalar conservation laws \eqref{eq:2d_scalar},
such that the numerical solutions satisfy the global or local MP.
The blending coefficient $\theta_{i+\frac12,j}\in[0,1]$ is chosen such that $u^{\min}_{i,j} \leqslant \tilde{u}_{i+\frac12,j}^{\texttt{Lim},-} \leqslant u^{\max}_{i,j}, u^{\min}_{i+1,j} \leqslant \tilde{u}_{i+\frac12,j}^{\texttt{Lim},+} \leqslant u^{\max}_{i+1,j}$.
To this end, the explicit expressions of the limited anti-diffusive flux are
\begin{align*}
	&\Delta\hat{f}^{\texttt{Lim}}_{\xr,j} =
	\begin{cases}
		\min\big\{\Delta\hat{f}_{\xr,j}, \Delta\hat{f}^{+}_{\xr,j} \big\},
		&\text{if}~ \Delta \hat{f}_{\xr,j} \geqslant 0, \\
		\max\big\{\Delta\hat{f}_{\xr,j}, \Delta\hat{f}^{-}_{\xr,j} \big\}, &\text{otherwise}, \\
	\end{cases}\\
	&\Delta\hat{f}^{+}_{\xr,j} = (\alpha_1)_{\xr,j}\min\big\{ \tilde{u}_{\xr,j}-u^{\min}_{i,j},
	u^{\max}_{i+1,j}-\tilde{u}_{\xr,j} \big\}, \\
	&\Delta\hat{f}^{-}_{\xr,j} =
	(\alpha_1)_{\xr,j}\max\big\{ u^{\min}_{i+1,j}-\tilde{u}_{\xr,j},
	\tilde{u}_{\xr,j}-u^{\max}_{i,j} \big\}.
\end{align*}
For the global MP, one can set $u^{\min}_{i,j} = m_0, u^{\max}_{i,j} = M_0, \forall i,j$ with $m_0, M_0$ defined in \eqref{eq:2d_scalar_g},
and for the local MP, the local bounds can be chosen based on the intermediate states
\begin{equation*}
	u^{\min}_{i,j} = \min\left\{\bar{u}_{i,j}^n, ~\tilde{u}_{\xl,j}, ~\tilde{u}_{\xr,j}\right\},~
	u^{\max}_{i,j} = \max\left\{\bar{u}_{i,j}^n, ~\tilde{u}_{\xl,j}, ~\tilde{u}_{\xr,j}\right\}.
\end{equation*}
Enforcing the local MP is useful to suppress spurious oscillations and to improve shock-capturing ability \cite{Guermond_2018_Second_SJoSC, Kuzmin_2012_Flux_book, Guermond_2016_Invariant_SJoNA}.
The final limited numerical flux is
\begin{equation}\label{eq:2d_flux_limited_scalar}
	\hat{f}^{\texttt{Lim}}_{\xr,j} = \hat{f}^{\texttt{L}}_{\xr,j} + \Delta\hat{f}^{\texttt{Lim}}_{\xr,j}.
\end{equation}

\subsection{Application to the compressible Euler equations}
This section focuses on enforcing the strict positivity of density and pressure, i.e., $\rho, p > \varepsilon$ (chosen as $10^{-13}$ in our numerical tests), and further limiting the solution based on shock sensors.
The limiting consists of three sequential steps.

\subsubsection{Positivity of density}
The first step is to impose the positivity of the density $\widetilde{\bU}_{\xr,j}^{\texttt{Lim},\pm,\rho} > \varepsilon$,
where $\bU^{*,\rho}$ denotes the density component of $\bU^{*}$.
The density component of the anti-diffusive flux is limited as
\begin{equation*}
	\Delta\widehat{\bF}^{\texttt{Lim},\rho}_{\xr,j} = \begin{cases}
		\min\left\{\Delta\widehat{\bF}^{\rho}_{\xr,j}, ~(\alpha_1)_{\xr,j}\left(\widetilde{\bU}^{\rho}_{\xr,j}-\varepsilon\right) \right\}, ~&\text{if}~ \Delta \widehat{\bF}^{\rho}_{\xr,j} \geqslant 0, \\
		\max\left\{\Delta\widehat{\bF}^{\rho}_{\xr,j}, ~(\alpha_1)_{\xr,j}\left(\varepsilon-\widetilde{\bU}^{\rho}_{\xr,j}\right) \right\}, ~&\text{otherwise}. \\
	\end{cases}
\end{equation*}
After this step, the density component of the numerical flux is modified as $\widehat{\bF}_{\xr,j}^{\texttt{Lim}, *, \rho} = \widehat{\bF}_{\xr,j}^{\texttt{L}, \rho} + \Delta\widehat{\bF}_{\xr,j}^{\texttt{Lim}, \rho}$,
while the momentum and energy components remain unchanged as the high-order flux $\widehat{\bF}_{\xr,j}^{\texttt{H}}$.

\subsubsection{Positivity of pressure}
The second step is to enforce the positivity of the pressure $p(\widetilde{\bU}_{\xr,j}^{\texttt{Lim},\pm}) > \varepsilon$,
where $p(\bU^{*})$ denotes the pressure recovered from $\bU^{*}$.
Let
\begin{equation*}
	\widetilde{\bU}_{\xr,j}^{\texttt{Lim},\pm} = \widetilde{\bU}_{\xr,j} \pm \dfrac{\theta_{\xr,j}\Delta \widehat{\bF}_{\xr,j}^{\texttt{Lim}, *}}{(\alpha_1)_{\xr,j}},\quad
	\Delta \widehat{\bF}_{\xr,j}^{\texttt{Lim}, *} = \widehat{\bF}_{\xr,j}^{\texttt{Lim}, *} - \widehat{\bF}_{\xr,j}^{\texttt{L}},
\end{equation*}
then the constraints yield two inequalities
\begin{equation}\label{eq:pressure_inequality}
	A_{\xr,j}\theta^2_{\xr,j} \pm B_{\xr,j}\theta_{\xr,j} < C_{\xr,j},
\end{equation}
with the coefficients (the subscript $(\cdot)_{\xr,j}$ is omitted on the right-hand side)
\begin{align*}
	A_{\xr,j} &= \dfrac{1}{2} \norm{ \Delta \widehat\bF^{\texttt{Lim}, *, \rho \bv} }_2^2
	- \Delta \widehat\bF^{\texttt{Lim}, *, \rho} \Delta \widehat\bF^{\texttt{Lim}, *, E}, \\
	B_{\xr,j} &= \alpha_1\left(\Delta \widehat\bF^{\texttt{Lim}, *, \rho} \widetilde\bU^{E}
	+ \widetilde\bU^{\rho} \Delta \widehat\bF^{\texttt{Lim}, *, E}
	- \Delta \widehat\bF^{\texttt{Lim}, *, \rho \bv}\cdot \widetilde\bU^{\rho \bv}
	- \varepsilon \Delta \widehat\bF^{\texttt{Lim}, *, \rho}\right), \\
	C_{\xr,j} &= \alpha_1^2\left(\widetilde\bU^{\rho} \widetilde\bU^{E}
	- \dfrac{1}{2} \norm{ \widetilde\bU^{\rho \bv} }_2^2
	- \varepsilon \widetilde\bU^{\rho}\right).
\end{align*}
Following \cite{Kuzmin_2020_Monolithic_CMiAMaE} and using $\theta_{\xr,j}^2 \leqslant \theta_{\xr,j}$, a linear sufficient condition for the constraints \eqref{eq:pressure_inequality} is
\begin{equation*}
	\left(\max\{0, A_{\xr,j}\} + \abs{B_{\xr,j}}\right)\theta_{\xr,j} < C_{\xr,j}.
\end{equation*}
Then the parameter is
\begin{equation*}
	\theta_{\xr,j} = \min\left\{1,~ \dfrac{C_{\xr,j}}{\max\{0, A_{\xr,j}\} + \abs{B_{\xr,j}}}\right\},
\end{equation*}
and the limited numerical flux is
\begin{equation}\label{eq:2d_flux_limited_Euler}
	\widehat{\bF}_{\xr,j}^{\texttt{Lim},**} = \widehat{\bF}_{\xr,j}^{\texttt{L}} + \theta_{\xr,j}\Delta\widehat{\bF}_{\xr,j}^{\texttt{Lim}, *}.
\end{equation}

\subsubsection{Shock sensor-based limiting}
Spurious oscillations are observed in the numerical solutions, especially near strong shock waves, if only the BP limitings are employed, see \Cref{sec:results}.
To reduce oscillations, we propose to further limit the numerical fluxes using another parameter $\theta_{\xr,j}^{s}$ based on shock sensors.
Consider the Jameson's shock sensor in \cite{Jameson_1981_Solutions_AJ},
\begin{equation*}
	(\varphi_1)_{i,j} = \dfrac{\abs{p_{i+1,j} - 2p_{i,j} + p_{i-1,j}}}{\abs{p_{i+1,j} + 2p_{i,j} + p_{i-1,j}}},
\end{equation*}
which was later improved in \cite{Ducros_1999_Large_JoCP} as a multiplication of two terms $(\varphi_1)_{i,j}(\widetilde{\varphi}_2)_{i,j}$,
with
\begin{equation*}
	(\widetilde{\varphi}_2)_{i,j} = \dfrac{(\nabla\cdot\bv)^2}{(\nabla\cdot\bv)^2 + (\nabla\times\bv)^2 + 10^{-40}},
\end{equation*}
known as the Ducros' shock sensor.
This paper only considers the sign of the velocity divergence rather than its magnitude, such that the shock waves can be located better, i.e.,
\begin{equation*}
	(\varphi_2)_{i,j} = \max\left\{\dfrac{-\nabla\cdot\bv}{\sqrt{(\nabla\cdot\bv)^2 + (\nabla\times\bv)^2 + 10^{-40}}}, ~0\right\}.
\end{equation*}
Finally, the blending coefficient is designed as
\begin{align*}
	&\theta_{\xr,j}^{s} = \exp(-\kappa (\varphi_1)_{\xr,j} (\varphi_2)_{\xr,j})\in (0, 1],\\
	&(\varphi_s)_{\xr,j} = \max\left\{(\varphi_s)_{i,j}, (\varphi_s)_{i+1,j}\right\}, ~s=1,2.
\end{align*}
The problem-dependent parameter $\kappa$ adjusts the strength of the limiting, and its optimal choice will be explored in the future.
The final limited numerical flux is
\begin{equation}\label{eq:2d_flux_limited_Euler_ss}
	\widehat{\bF}_{\xr,j}^{\texttt{Lim}} = \widehat{\bF}_{\xr,j}^{\texttt{L}} + \theta_{\xr,j}^{s}\Delta\widehat{\bF}_{\xr,j}^{\texttt{Lim}, **},
\end{equation}
with $\Delta\widehat{\bF}_{\xr,j}^{\texttt{Lim},**} = \widehat{\bF}_{\xr,j}^{\texttt{Lim},**} - \widehat{\bF}_{\xr,j}^{\texttt{L}}$.

\section{Scaling limiter for point value}\label{sec:2d_limiting_point}
We must also introduce BP limitings for the point value, achieved by using a simple scaling limiter \cite{Liu_1996_Nonoscillatory_SJoNA} directly on the high-order point values as no conservation is required for the point value update.

A first-order LLF scheme for the point value update at nodes can be chosen as
\begin{align}\label{eq:2d_llf_node}
	\bU_{\xr,\yr}^{\texttt{L}} = \bU_{\xr,\yr}^{n}
	&- \dfrac{2\Delta t^n}{\Delta x_i+\Delta x_{i+1}}
	\left(\widehat{\bF}^{\texttt{L}}_{i+1,\yr}
	- \widehat{\bF}^{\texttt{L}}_{i,\yr}\right) \nonumber\\
	&- \dfrac{2\Delta t^n}{\Delta y_j+\Delta y_{j+1}}
	\left(\widehat{\bF}^{\texttt{L}}_{\xr,j+1}
	- \widehat{\bF}^{\texttt{L}}_{\xr,j}\right),
\end{align}
with the LLF numerical fluxes
\begin{align*}
	&\widehat{\bF}^{\texttt{L}}_{i+1,\yr}:=\widehat{\bF}_1^{\texttt{LLF}}(\bU_{\xr,\yr}^n, \bU_{i+\frac32,\yr}^n),~
	\widehat{\bF}^{\texttt{L}}_{i,\yr}:=\widehat{\bF}_1^{\texttt{LLF}}(\bU_{i-\frac12,\yr}^n, \bU_{i+\frac12,\yr}^n), \\
	&\widehat{\bF}^{\texttt{L}}_{\xr,j+1}:=\widehat{\bF}_2^{\texttt{LLF}}(\bU_{\xr,\yr}^n, \bU_{\xr,j+\frac32}^n),~
	\widehat{\bF}^{\texttt{L}}_{\xr,j}:=\widehat{\bF}_2^{\texttt{LLF}}(\bU_{\xr,j-\frac12}^n, \bU_{\xr,j+\frac12}^n).
\end{align*}
For the vertical face-centered point value, this paper chooses the first-order LLF scheme as
\begin{equation}\label{eq:2d_llf_facex}
	\bU_{\xr,j}^{\texttt{L}} =\  \bU_{\xr,j}^{n}
	- \dfrac{2\Delta t^n}{\Delta x_i+\Delta x_{i+1}}
	\left(\widehat{\bF}^{\texttt{L}}_{i+1,j}
	- \widehat{\bF}^{\texttt{L}}_{i,j}\right) 
	- \dfrac{\Delta t^n}{\Delta y_j}
	\left(\widehat{\bF}^{\texttt{L}}_{\xr,j+\frac12}
	- \widehat{\bF}^{\texttt{L}}_{\xr,j-\frac12}\right),
\end{equation}
with the LLF numerical fluxes
\begin{align*}
	&\widehat{\bF}^{\texttt{L}}_{i+1,j}:=\widehat{\bF}_1^{\texttt{LLF}}(\bU_{\xr,j}^n, \bU_{i+\frac32,j}^n),~
	\widehat{\bF}^{\texttt{L}}_{i,j}:=\widehat{\bF}_1^{\texttt{LLF}}(\bU_{i-\frac12,j}^n, \bU_{i+\frac12,j}^n), \\
	&\widehat{\bF}^{\texttt{L}}_{\xr,j+\frac12}:=\widehat{\bF}_2^{\texttt{LLF}}(\bU_{\xr,j}^n, \bU_{\xr,\yr}^n),~
	\widehat{\bF}^{\texttt{L}}_{\xr,j-\frac12}:=\widehat{\bF}_2^{\texttt{LLF}}(\bU_{\xr,j-\frac12}^n, \bU_{\xr,j}^n).
\end{align*}
Note that $\bU_{i+\frac12,j+\frac12}^n$ and $\bU_{i+\frac12,j-\frac12}^n$ are used to compute the fluxes in the $y$-direction, while the corresponding spatial mesh size used in the scheme is $\Delta y_j$ rather than $\Delta y_j/2$.
This is because we restrict the spatial stencil for the first-order scheme within that for the higher-order AF schemes, and let the time step sizes required for the BP property of the first-order schemes \eqref{eq:2d_llf_facex} and \eqref{eq:2d_llf_node} be consistent in the $y$-direction with uniform meshes.
There may be other choices of the first-order schemes for the point value update, which needs further exploration. 
The LLF scheme for the face-centered value on the horizontal face is similar,
\begin{equation}\label{eq:2d_llf_facey}
	\bU_{i,\yr}^{\texttt{L}} = \bU_{i,\yr}^{n}
	- \dfrac{\Delta t^n}{\Delta x_i}
	\left(\widehat{\bF}^{\texttt{L}}_{i+\frac12,\yr}
	- \widehat{\bF}^{\texttt{L}}_{i-\frac12,\yr}\right)
	- \dfrac{2\Delta t^n}{\Delta y_j+\Delta y_{j+1}}
	\left(\widehat{\bF}^{\texttt{L}}_{i,j+1}
	- \widehat{\bF}^{\texttt{L}}_{i,j}\right),
\end{equation}
with similarly defined LLF numerical fluxes as for the vertical face.

\begin{lemma}
	The LLF schemes for the point value updates \eqref{eq:2d_llf_node}-\eqref{eq:2d_llf_facey} are BP under the following constraint for the time step size
	\begin{align}
		\Delta t^n \leqslant \frac12\min\Bigg\{
		&\dfrac{\Delta x_i + \Delta x_{i+1}}{2\left((\alpha_1)_{i,j+\frac12}+ (\alpha_1)_{i+1,j+\frac12}\right)},
		\dfrac{\Delta y_j + \Delta y_{j+1}}{2\left((\alpha_2)_{i+\frac12,j}+ (\alpha_2)_{i+\frac12,j+1}\right)}, \nonumber\\
		&\dfrac{\Delta x_i + \Delta x_{i+1}}{2\left((\alpha_1)_{i,j}+ (\alpha_1)_{i+1,j}\right)},
		\dfrac{\Delta y_j}{\left((\alpha_2)_{i+\frac12,j+\frac12}+ (\alpha_2)_{i+\frac12,j-\frac12}\right)}, \nonumber\\
		&\dfrac{\Delta x_i}{\left((\alpha_1)_{i+\frac12,j+\frac12}+ (\alpha_1)_{i-\frac12,j+\frac12}\right)},
		\dfrac{\Delta y_j + \Delta y_{j+1}}{2\left((\alpha_2)_{i,j} + (\alpha_2)_{i,j+1}\right)} 
		 \Bigg\}. \label{eq:2d_pnt_llf_dt}
	\end{align}
\end{lemma}
The proof is similar to that for \Cref{lem:llf_g}.

\begin{remark}
	For uniform meshes, and if taking the maximal wave speeds in the domain, the following condition
	\begin{equation}\label{eq:uniform_mesh_dt}
		\Delta t^n \leqslant \frac14\min\left\{\dfrac{\Delta x}{\norm{\varrho_1}_\infty},
		\dfrac{\Delta y}{\norm{\varrho_2}_\infty}
		\right\}
	\end{equation}
	suffices to satisfy the condition \eqref{eq:2d_convex_combination_dt} and \eqref{eq:2d_pnt_llf_dt}.
\end{remark}


The limited state is obtained by blending the high-order AF scheme \eqref{eq:semi_pnt_2d} with the forward Euler step and the LLF schemes \eqref{eq:2d_llf_node}-\eqref{eq:2d_llf_facey} as $\bU_{\sigma}^{\texttt{Lim}} = \theta_{\sigma} \bU_{\sigma}^{\texttt{H}} + (1-\theta_{\sigma}) \bU_{\sigma}^{\texttt{L}}$,
such that $\bU_{\sigma}^{\texttt{Lim}}\in\mathcal{G}$.

\subsection{Application to scalar conservation laws}
This section enforces the global MP $m_0 \leqslant u_{\sigma}^{\texttt{Lim}} \leqslant M_0$ for the point value update in scalar case,
and the coefficient is determined by
\begin{equation*}
	\theta_{\sigma} = \min\left\{ \left|\dfrac{u_{\sigma}^{\texttt{L}}-m_0}{u_{\sigma}^{\texttt{L}}-u_{\sigma}^{\texttt{H}}} \right|, ~\left| \dfrac{M_0-u_{\sigma}^{\texttt{L}}}{u_{\sigma}^{\texttt{H}}-u_{\sigma}^{\texttt{L}}} \right|, ~1
	\right\}.
\end{equation*}
The final limited state is
\begin{equation}\label{eq:2d_pnt_limited_state_scalar}
	u_{\sigma}^{\texttt{Lim}} = \theta_{\sigma} u_{\sigma}^{\texttt{H}} + \left(1-\theta_{\sigma}\right) u_{\sigma}^{\texttt{L}}.
\end{equation}

\subsection{Application to the compressible Euler equations}
The limiting consists of two steps.
First, the high-order state $\bU_{\sigma}^{\texttt{H}}$ is modified as $\bU_{\sigma}^{\texttt{Lim}, *}$,
such that its density component satisfies $\bU_{\sigma}^{\texttt{Lim}, *, \rho} > \varepsilon$.
Solving this inequality gives the coefficient
\begin{equation*}
	\theta_{\sigma}^{*} = \min\left\{ \dfrac{\bU_{\sigma}^{\texttt{L}, \rho}-\varepsilon}{\bU_{\sigma}^{\texttt{L}, \rho}-\bU_{\sigma}^{\texttt{H}, \rho}}, ~1
	\right\}.
\end{equation*}
The density component of the limited state is $\bU_{\sigma}^{\texttt{Lim}, *, \rho} = \theta^{*}_{\sigma} \bU_{\sigma}^{\texttt{H}, \rho} + (1-\theta^{*}_{\sigma}) \bU_{\sigma}^{\texttt{L}, \rho}$,
with the momentum and energy components remaining the same as the high-order state $\bU_{\sigma}^{\texttt{H}}$.

Then the limited state $\bU_{\sigma}^{\texttt{Lim}, *}$ is modified as $\bU_{\sigma}^{\texttt{Lim}}$,
such that it recovers positive pressure, i.e., $p\left(\bU_{\sigma}^{\texttt{Lim}}\right) > \varepsilon$.
Let $\bU_{\sigma}^{\texttt{Lim}} = \theta^{**}_{\sigma} \bU_{\sigma}^{\texttt{Lim}, *} + (1-\theta^{**}_{\sigma}) \bU_{\sigma}^{\texttt{L}}$.
The pressure is a concave function (see e.g. \cite{Zhang_2011_Maximum_PotRSAMPaES}) of the conservative variables, so that $p\left(\bU_{\sigma}^{\texttt{Lim}}\right) > \theta^{**}_{\sigma}p\left(\bU_{\sigma}^{\texttt{Lim}, *}\right) + \left(1-\theta^{**}_{\sigma}\right)p\left(\bU_{\sigma}^{\texttt{L}}\right)$
based on Jensen's inequality and $\bU_{\sigma}^{\texttt{Lim}, *, \rho} > 0$, $\bU_{\sigma}^{\texttt{L}, \rho} > 0$, $\theta_{\sigma}^{**} \in [0,1]$.
Thus a sufficient condition is
\begin{equation*}
	\theta_{\sigma}^{**} = \min\left\{
	\dfrac{p\left(\bU_{\sigma}^{\texttt{L}}\right) - \varepsilon}{p\left(\bU_{\sigma}^{\texttt{L}}\right) - p\left(\bU_{\sigma}^{\texttt{Lim}, *}\right)}, ~1
	\right\},
\end{equation*}
and the final limited state is
\begin{equation}\label{eq:2d_pnt_limited_state_Euler}
	\bU_{\sigma}^{\texttt{Lim}} = \theta^{**}_{\sigma} \bU_{\sigma}^{\texttt{Lim}, *} + \left(1-\theta^{**}_{\sigma}\right) \bU_{\sigma}^{\texttt{L}}.
\end{equation}

Let us summarize the main results of the BP AF methods in this paper.
\begin{proposition}
	If the initial numerical solution $\overline{\bU}_{i,j}^0, \bU_{\sigma}^0\in\mathcal{G}$ for all $i,j,\sigma$,
	and the time step size satisfies \eqref{eq:2d_convex_combination_dt} and \eqref{eq:2d_pnt_llf_dt},
	then the AF methods \eqref{eq:semi_av_2d}-\eqref{eq:semi_pnt_2d} equipped with the SSP-RK3 \eqref{eq:ssp_rk3} and the BP limitings
	\begin{itemize}[leftmargin=*]
		\item \eqref{eq:2d_flux_limited_scalar} and \eqref{eq:2d_pnt_limited_state_scalar} preserve the maximum principle for scalar case;
		\item \eqref{eq:2d_flux_limited_Euler} and \eqref{eq:2d_pnt_limited_state_Euler} preserve the density and pressure positivity for the Euler equations.
	\end{itemize}
\end{proposition}


\section{Numerical results}\label{sec:results}
This section presents some numerical results to verify the proposed AF methods' accuracy, BP property, and shock-capturing ability.
Unless otherwise stated, the CFL number is chosen as $0.25$ and the adiabatic index is $\gamma=1.4$ for the Euler equations.
In the 2D plots, the numerical solutions are visualized on a refined mesh with the half mesh size, where the values at the grid points are the cell averages or point values on the original mesh.

\begin{example}[Advection equation]\label{ex:2d_advection_discontinuity}\rm
	This test solves the 2D advection equation $u_t + u_x + u_y = 0$,
	on the periodic domain $[0,1]\times[0,1]$ with the following initial data
	\begin{equation*}
		u_0(x,y) = \begin{cases}
			1-\abs{5r}, &\text{if}\quad r=\sqrt{(x-0.25)^2+(y-0.25)^2}<0.2, \\
			1, &\text{if}\quad \max\{\abs{x-0.75}, ~\abs{y-0.75}\} < 0.2, \\
			0, &\text{otherwise}. \\
		\end{cases}
	\end{equation*}
	It is solved for two periods, i.e., until $T= 2$, and used to verify that the BP limitings can preserve the MP for both the cell average and point value. 

For the advection equation, the JS and FVS are equivalent.
The results on the uniform $100\times 100$ mesh obtained without and with BP limitings imposing the global MP for the cell average and point value are presented in \Cref{fig:2d_advection_discontinuity_surface}, as well as the cut-line along $y=x$.
The BP limitings suppress the oscillations well near the discontinuities and do not reduce the resolution.
The ranges of the numerical solutions with different limitings are listed in \Cref{tab:2d_advection_discontinuity_bounds}, showing the minimum and maximum of both the cell averages and point values.
The activation of either BP limiting for the cell average or point value fails to preserve the bound $[0,1]$.
Only when both the BP limitings are performed on the cell average and point value, the MP is achieved,
demonstrating that it is necessary to use the two BP limitings simultaneously.

\begin{figure}[htbp!]
	\centering
	\begin{subfigure}[b]{0.31\textwidth}
		\centering
		\includegraphics[width=1.0\linewidth]{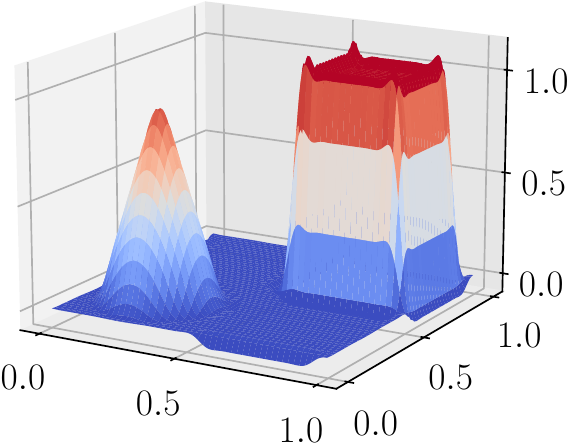}
	\end{subfigure}
	\begin{subfigure}[b]{0.31\textwidth}
		\centering
		\includegraphics[width=1.0\linewidth]{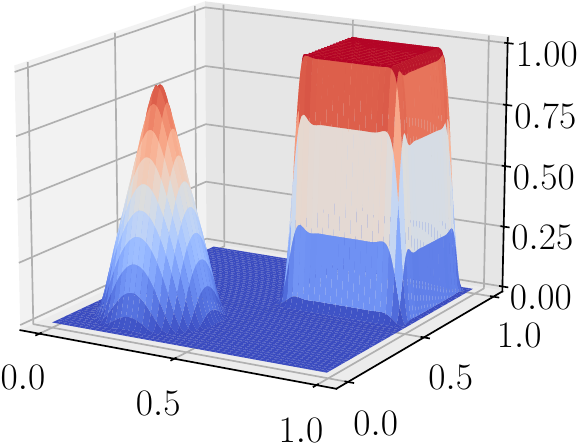}
	\end{subfigure}
        \begin{subfigure}[b]{0.35\textwidth}
		\centering
		\includegraphics[width=1.0\linewidth]{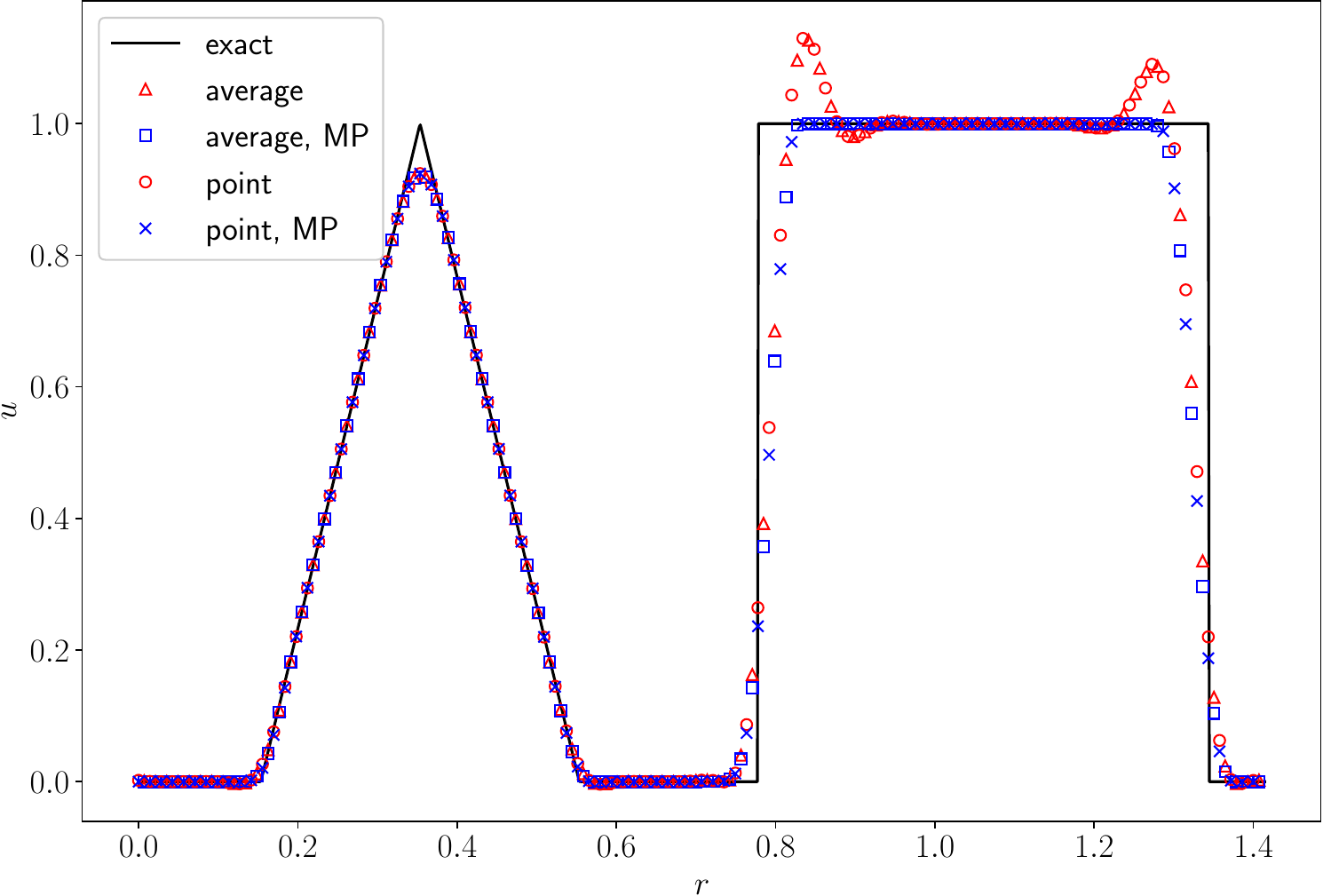}
	\end{subfigure}
	\caption{\Cref{ex:2d_advection_discontinuity}, 2D advection equation.
        The numerical solution is obtained based on the JS on the uniform $100\times100$ mesh.
        From left to right: without any limiting, with BP limitings imposing the global MP, cut-line along $y=x$.}
	\label{fig:2d_advection_discontinuity_surface}
\end{figure}

\begin{table}[htbp!]
	\centering
	{\scriptsize 
	\begin{tabular}{l|r|r}
		\hline\hline
	\diagbox{cell average}{point value}  & no limiting & global MP \\ \hline
	no limiting   &  $[-\num{6.6e-02}, 1+\num{1.3e-01}]$\quad{\color{red}\xmark}   &      $[-\num{5.3e-02}, 1+\num{5.7e-04}]$\quad{\color{red}\xmark}        \\ \hline
	global MP   &  $[-\num{1.4e-03}, 1+\num{5.0e-03}]$\quad{\color{red}\xmark}   &      $[0, 1]$\quad{\color{blue}\cmark}        \\ \hline
	local MP   &  $[-\num{3.1e-04}, 1+\num{4.3e-08}]$\quad{\color{red}\xmark}   &      $[0, 1-\num{8.6e-12}]$\quad{\color{blue}\cmark}        \\
		\hline\hline
	\end{tabular}
}
	\caption{\Cref{ex:2d_advection_discontinuity}, 2D advection equation.
	The ranges of the numerical solutions after two periods with different limitings.}
	\label{tab:2d_advection_discontinuity_bounds}
\end{table}
\end{example}

\begin{example}[Burgers' equation]\label{ex:2d_burgers}\rm
	Consider 2D Burgers' equation $u_t + \left(\frac12 u^2\right)_x + \left(\frac12 u^2\right)_y = 0$
	on the periodic domain $[0,1]\times[0,1]$.
	The initial condition is a sine wave $u_0(x, y) = 0.5 + \sin(2\pi(x+y))$.
	This test is solved until $T=0.3$, when the shock waves have emerged.

\Cref{fig:2d_burgers_shock_surface} plots the solutions using the LLF FVS on the uniform $100\times100$ mesh without and with BP limitings imposing local MP for the cell average and global MP for the point value, as well as the cut-line along $y=x$.
The oscillations near the shock waves are suppressed well when the limitings are activated, and the numerical solutions agree well with the reference solution.
The blending coefficients $\theta_{\xr,j}, \theta_{i,\yr}$ for the cell average and $\theta_{\sigma}$ for the point value are also presented in \Cref{fig:2d_burgers_shock_limiting}, verifying that the limitings are only locally activated near the shock waves.
 
\begin{figure}[htbp!]
	\centering
	\begin{subfigure}[b]{0.31\textwidth}
		\centering
		\includegraphics[width=\linewidth]{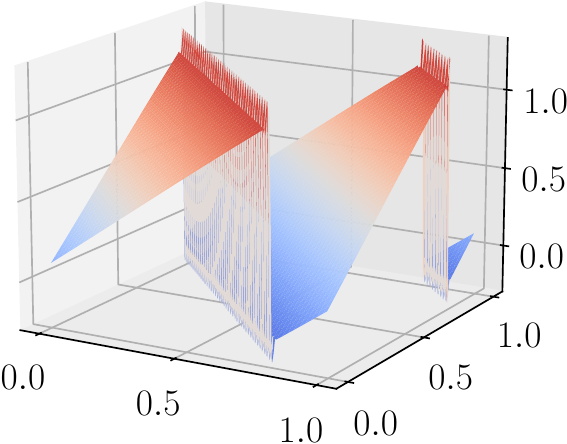}
	\end{subfigure}
	\begin{subfigure}[b]{0.31\textwidth}
		\centering
		\includegraphics[width=\linewidth]{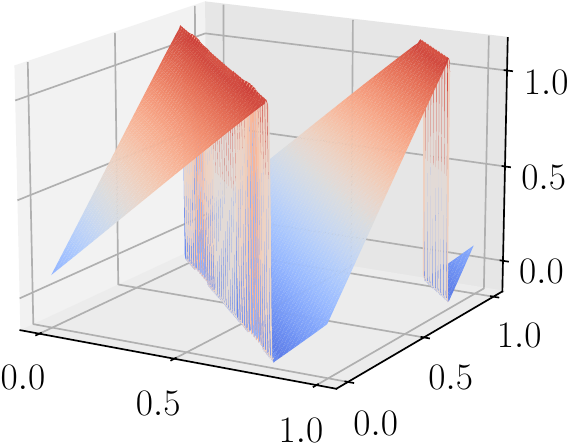}
	\end{subfigure}
        \begin{subfigure}[b]{0.32\textwidth}
		\centering
		\includegraphics[width=\linewidth]{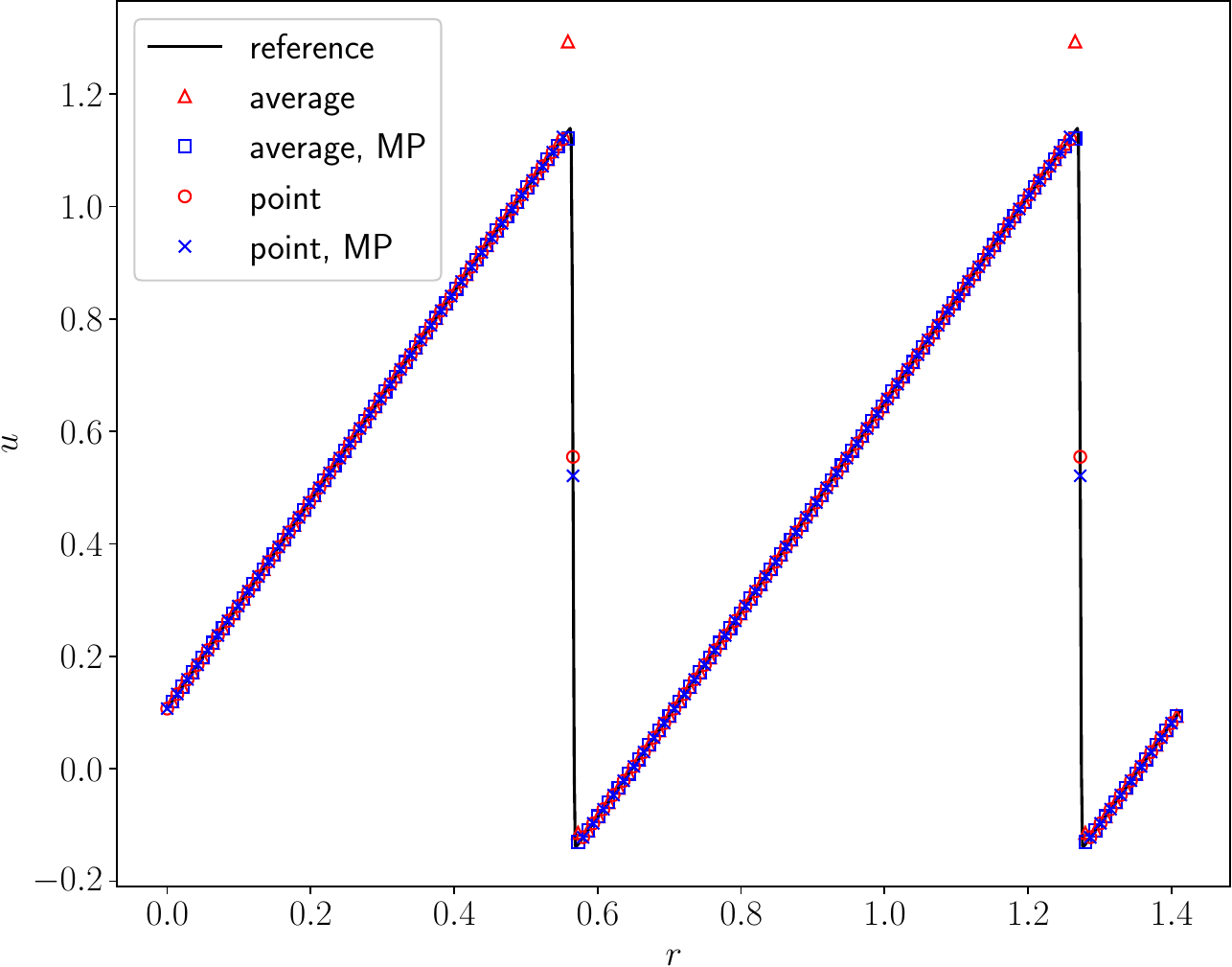}
	\end{subfigure}
	\caption{\Cref{ex:2d_burgers}, 2D Burgers' equation.
	The numerical solution is obtained using the LLF FVS on the uniform $100\times100$ mesh.
 From left to right: without limiting, with BP limiting, cut-line along $y=x$.}
	\label{fig:2d_burgers_shock_surface}
\end{figure}

\begin{figure}[htbp!]
	\centering
	\begin{subfigure}[b]{0.32\textwidth}
		\centering
		\includegraphics[width=\linewidth]{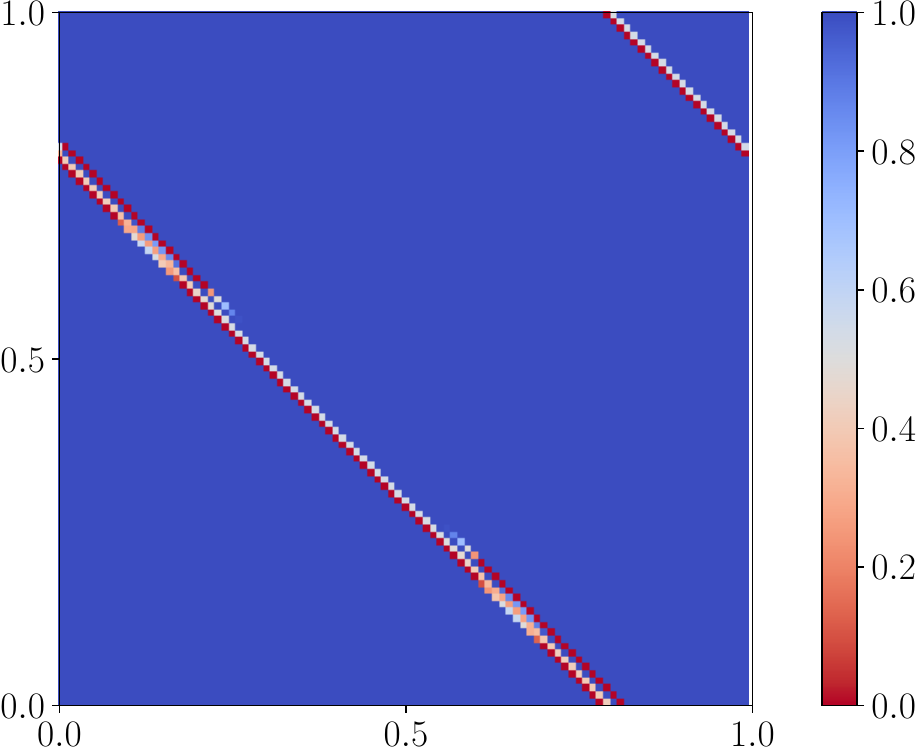}
	\end{subfigure}
	\begin{subfigure}[b]{0.32\textwidth}
		\centering
		\includegraphics[width=\linewidth]{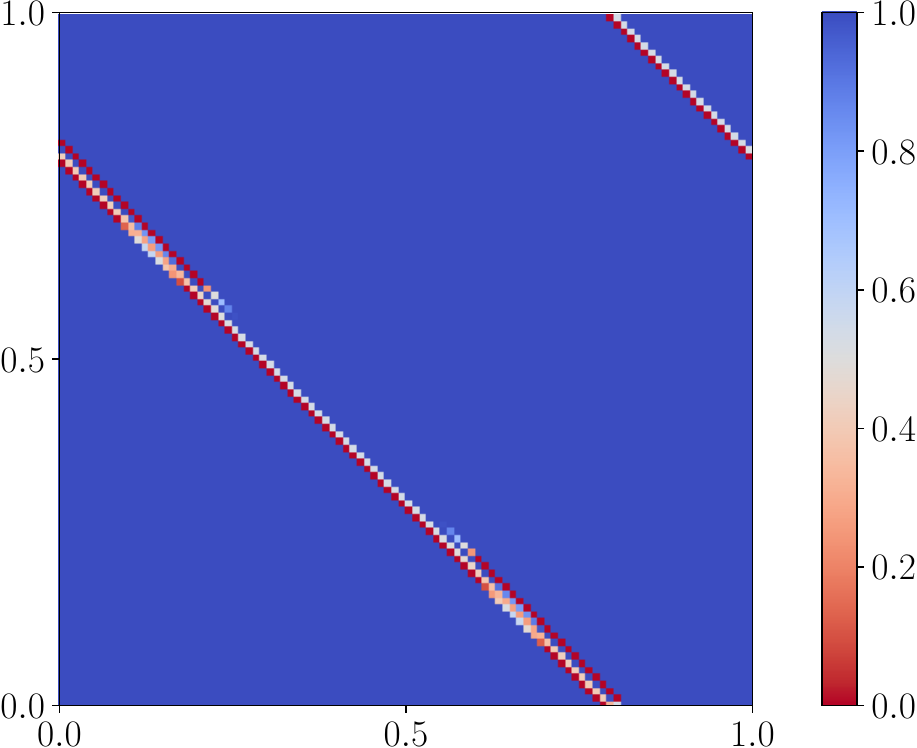}
	\end{subfigure}
	\begin{subfigure}[b]{0.32\textwidth}
		\centering
		\includegraphics[width=\linewidth]{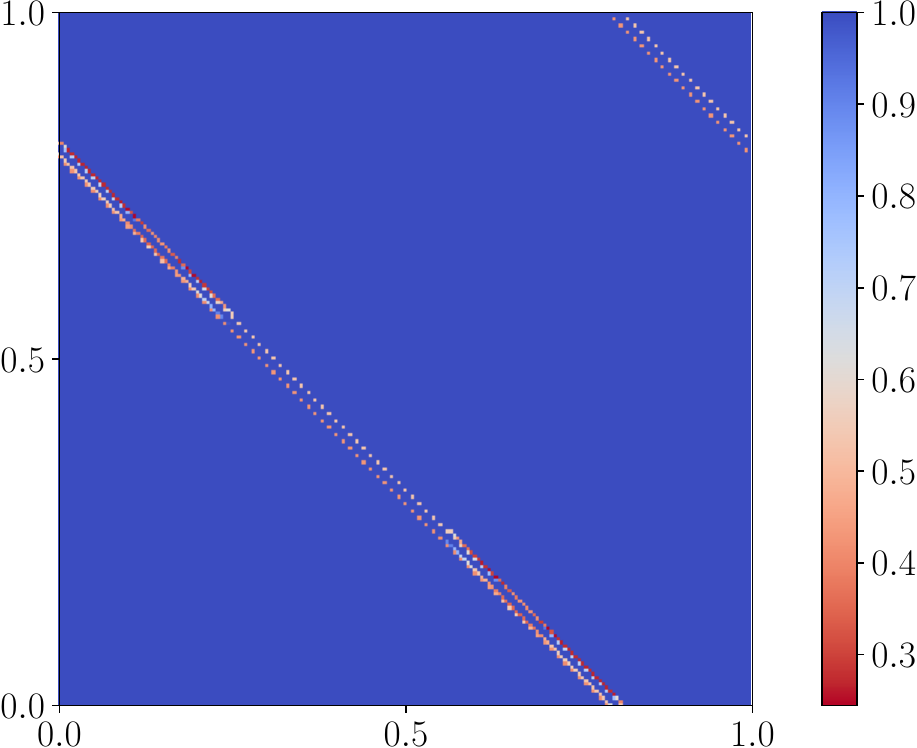}
	\end{subfigure}
	\caption{\Cref{ex:2d_burgers}, 2D Burgers' equation.
	The blending coefficients used in the limitings.
	From left to right: $\theta_{\xr,j}$ and $\theta_{i,\yr}$ for the cell average, $\theta_{\sigma}$ for the point value.}
	\label{fig:2d_burgers_shock_limiting}
\end{figure}
\end{example}

\begin{example}[2D isentropic vortex]\label{ex:2d_vortex}\rm
	This test describes an isentropic vortex propagating periodically at a constant speed $(1,1)$ in the 2D domain $[-5,5]\times[-5,5]$.
	The initial condition is
	\begin{equation*}
		\rho = T_0^{\frac{1}{\gamma}},~
		(v_1, v_2) = (1, 1) + k_0(y, -x),~
		p = T_0\rho,~
		k_0=\dfrac{\epsilon}{2\pi}e^{0.5(1-r^2)},~
		T_0 = 1 - \dfrac{\gamma-1}{2\gamma}k_0^2,
	\end{equation*}
	where $r^2 = x^2+y^2$, and $\epsilon=5$ is the vortex strength.
	The problem is solved for one period until $T = 10$.

Figure \ref{fig:2d_vortex_accuracy} shows the errors and corresponding convergence rates of the conservative variables in the $\ell^1$ norm.
Similar to the 1D case \cite{Duan_2024_Active_SJoSC}, the AF methods based on the JS and all the FVS except for the SW FVS achieve the designed third-order accuracy.
The convergence rate of the SW FVS reduces to around $2.5$,
which happens when the eigenvalue is close to zero, similar to the ``entropy glitch'' in the literature.
The common entropy fix by Harten and Hyman \cite{Harten_1983_Self_JoCP} cannot recover the third-order accuracy when it is directly applied to modify the eigenvalues here.
Further investigation is needed in this direction.

\begin{figure}[htbp!]
	\centering
	\begin{subfigure}[b]{0.5\textwidth}
		\centering
		\includegraphics[width=1.0\linewidth]{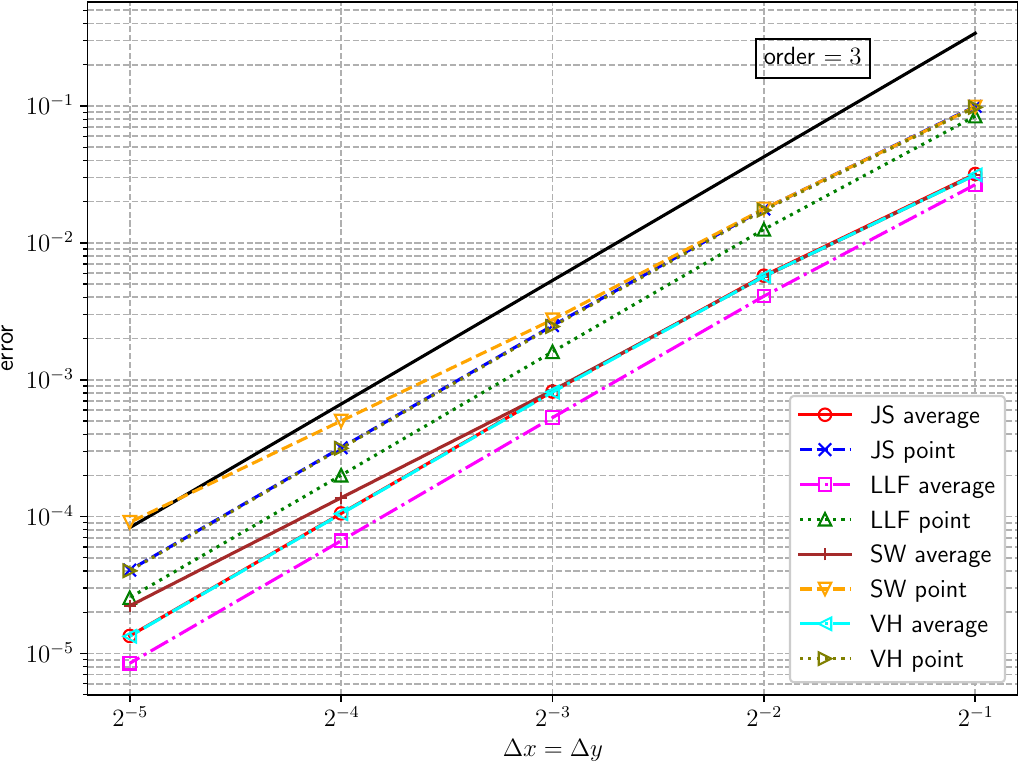}
	\end{subfigure}
	\caption{\Cref{ex:2d_vortex}, 2D isentropic vortex problem.
		The errors of the conservative variables in the $\ell^1$ norm.
	}
	\label{fig:2d_vortex_accuracy}
\end{figure}
\end{example}

\begin{example}[Quasi-2D Sod shock tube]\label{ex:2d_sod}\rm
	In this test, the Sod shock tube problem along the $x$-direction is solved on the domain $[0,1]\times[0,1]$ with a $100\times 2$ uniform mesh until $T=0.2$.
	The initial condition is
	\begin{equation*}
		(\rho, v_1, v_2, p) = \begin{cases}
			(1, ~0, ~0, ~1), &\text{if}~ x < 0.5,\\
			(0.125, ~0, ~0, ~0.1), &\text{otherwise}.\\
		\end{cases}
	\end{equation*}
	
	The density plots obtained by using different ways for the point value update without ($\kappa=0$) and with the shock sensor ($\kappa=1$) are shown in \Cref{fig:2d_sod_density}.
	The JS-based AF method gives solutions with large errors between the contact discontinuity and shock wave, which are not reduced by the limiting based on the shock sensor, known as the mesh alignment issue.
	The results of all the FVS-based methods agree well with the exact solution when the limiting is activated, thus the FVS-based AF methods are more advantageous in simulations since they can cure both the transonic issue \cite{Duan_2024_Active_SJoSC} and the mesh alignment issue.
	To save space, the following tests are all computed based on the LLF FVS, which gives slightly better results among all the three FVS.
		
	\begin{figure}[htbp!]
		\centering		
		\begin{subfigure}[b]{0.24\textwidth}
			\includegraphics[width=\linewidth]{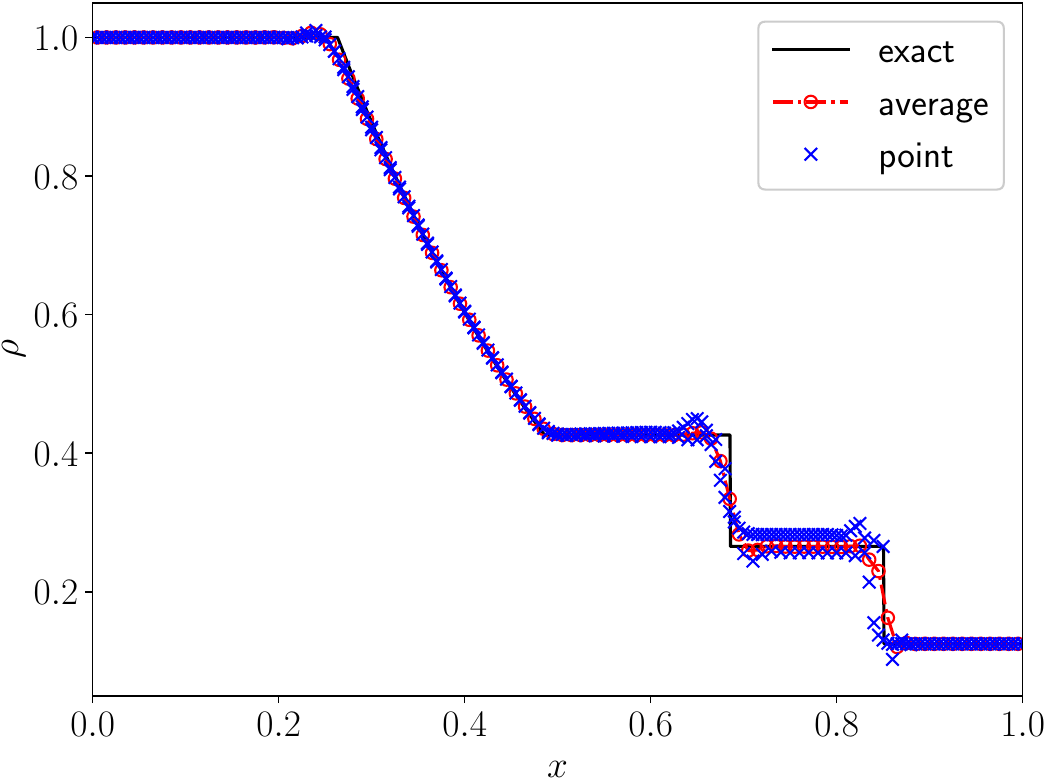}
		\end{subfigure}
		\begin{subfigure}[b]{0.24\textwidth}
			\centering
			\includegraphics[width=\linewidth]{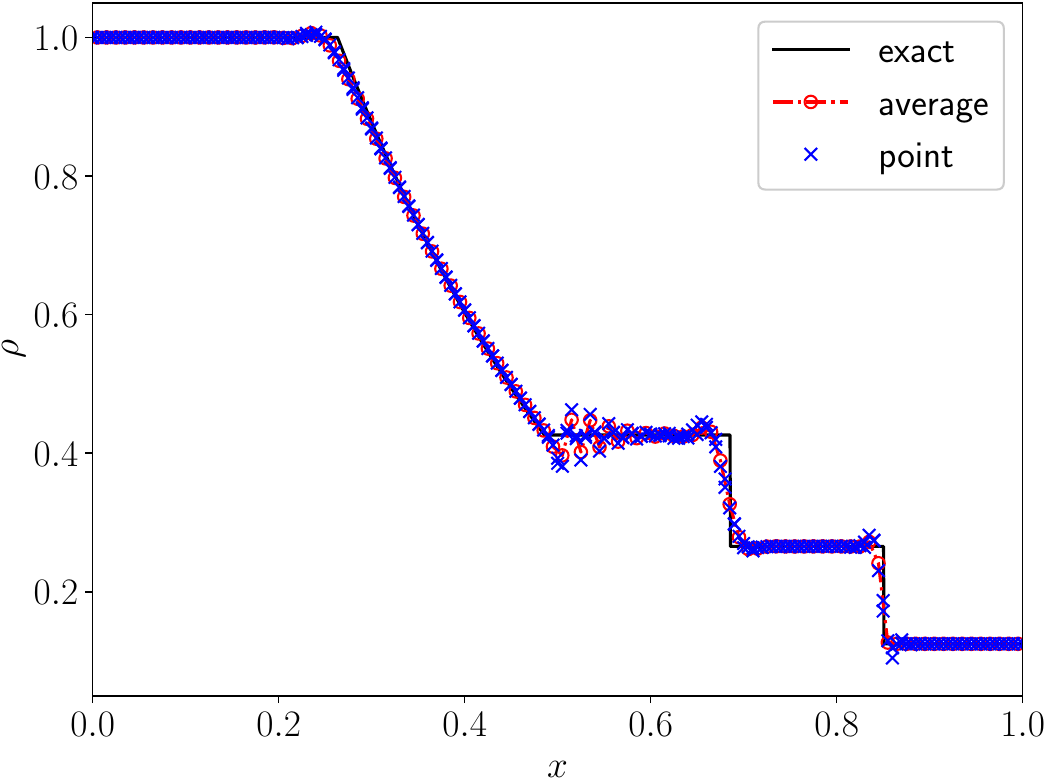}
		\end{subfigure}
		\begin{subfigure}[b]{0.24\textwidth}
			\centering
			\includegraphics[width=\linewidth]{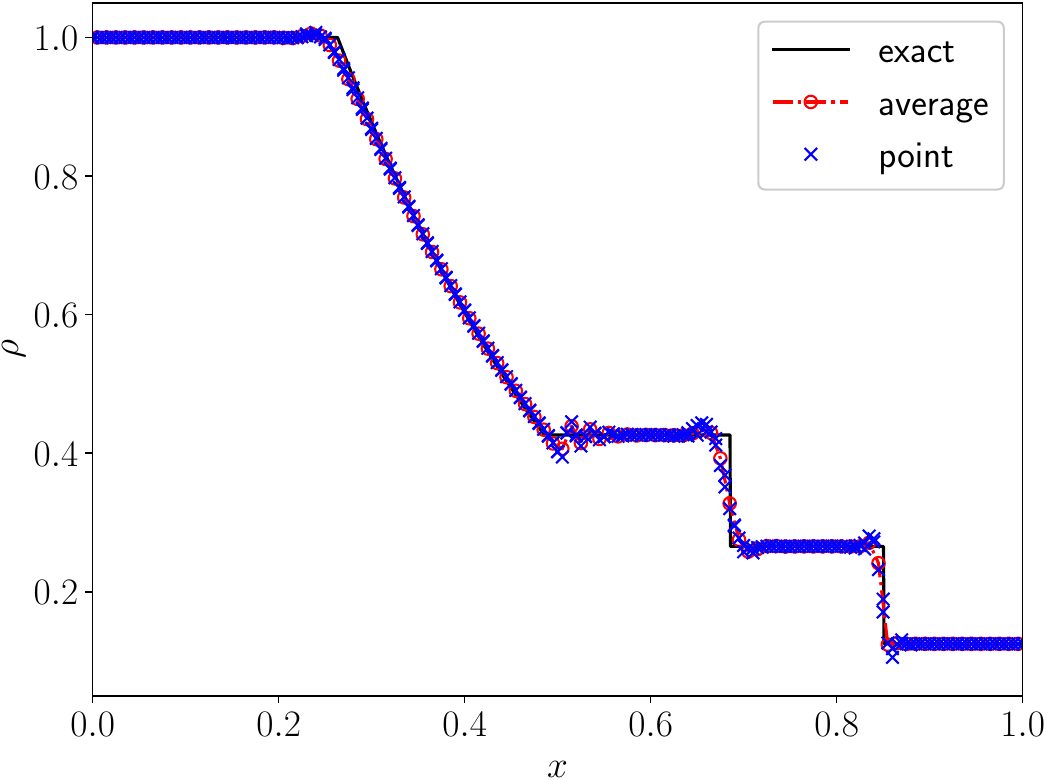}
		\end{subfigure}
		\begin{subfigure}[b]{0.24\textwidth}
			\centering
			\includegraphics[width=\linewidth]{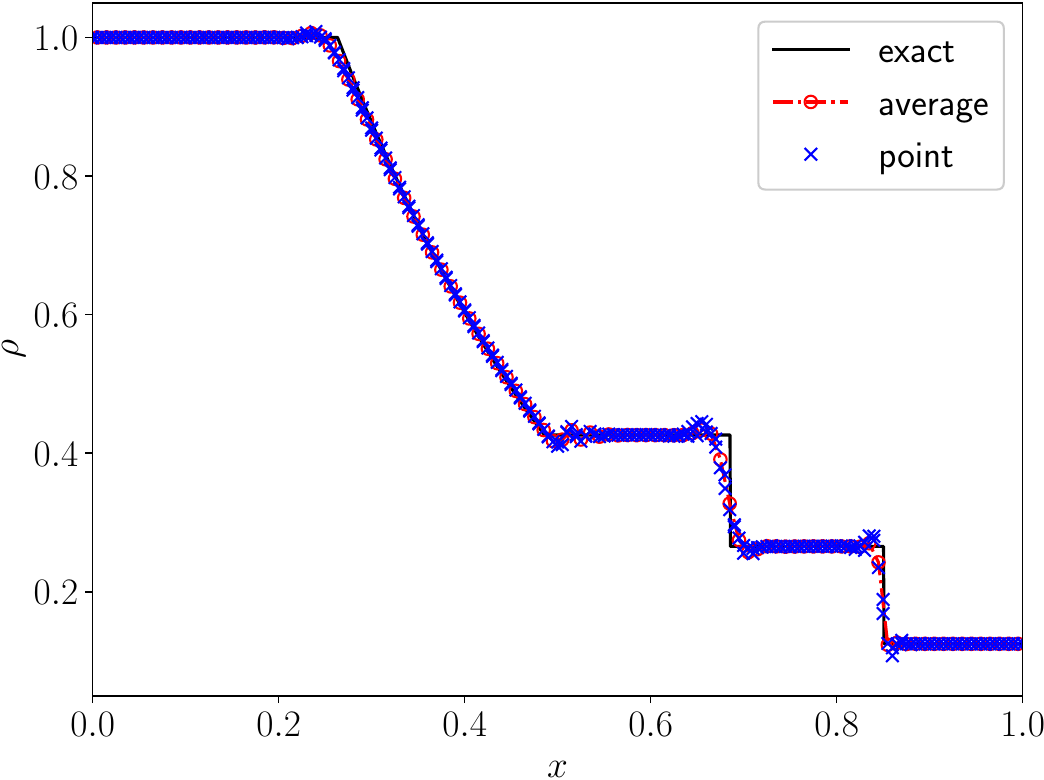}
		\end{subfigure}
		\vspace{5pt}
		
		\begin{subfigure}[b]{0.24\textwidth}
			\includegraphics[width=\linewidth]{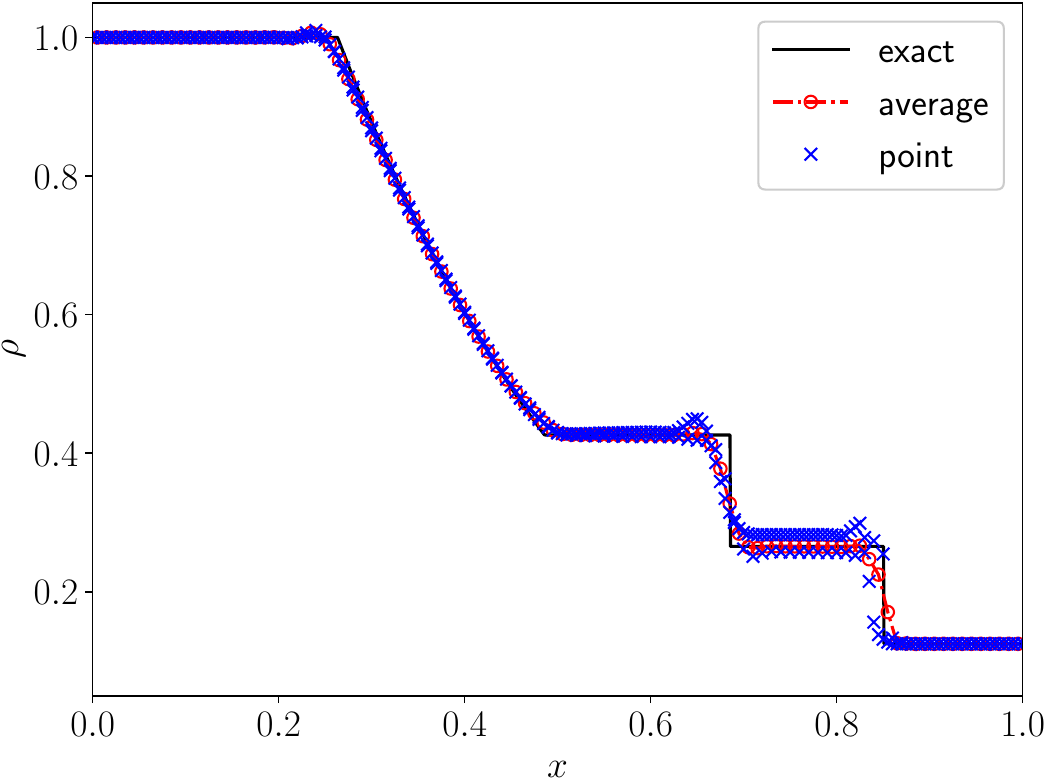}
		\end{subfigure}
		\begin{subfigure}[b]{0.24\textwidth}
			\centering
			\includegraphics[width=\linewidth]{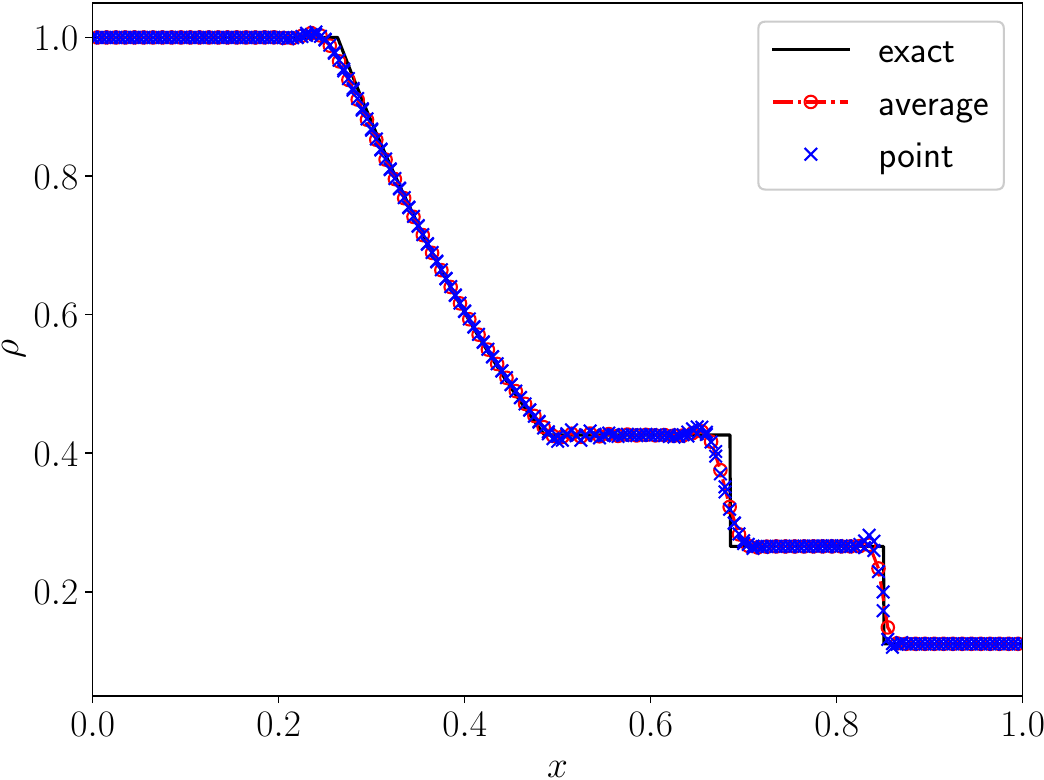}
		\end{subfigure}
		\begin{subfigure}[b]{0.24\textwidth}
			\centering
			\includegraphics[width=\linewidth]{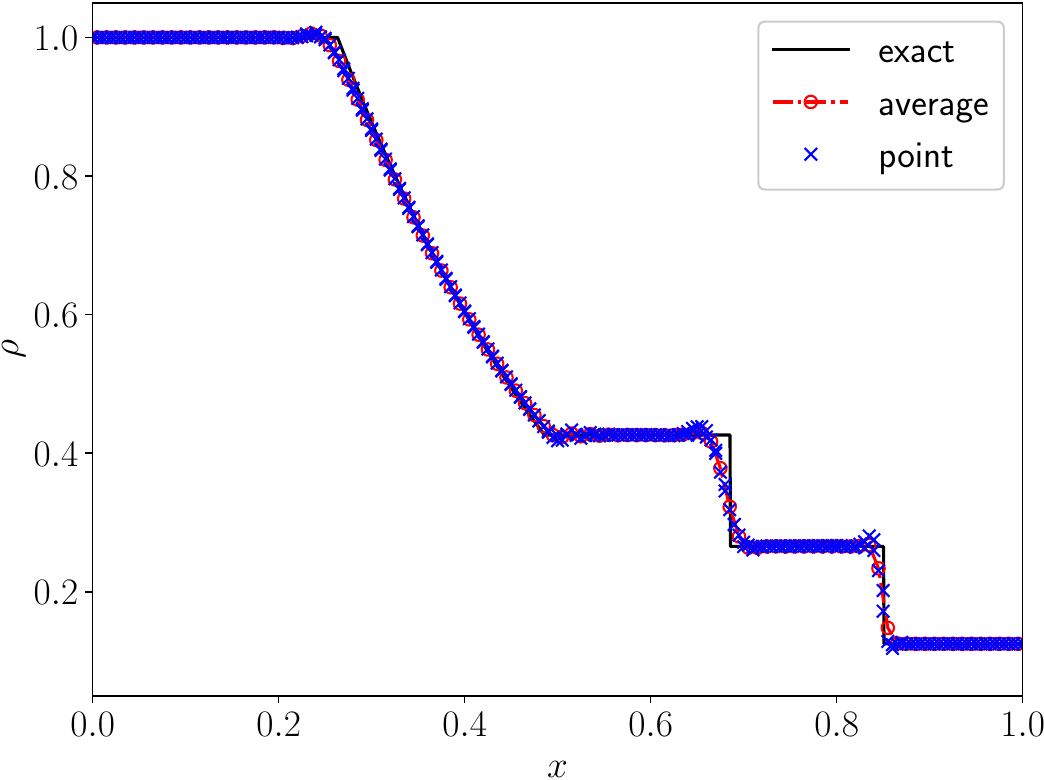}
		\end{subfigure}
		\begin{subfigure}[b]{0.24\textwidth}
			\centering
			\includegraphics[width=\linewidth]{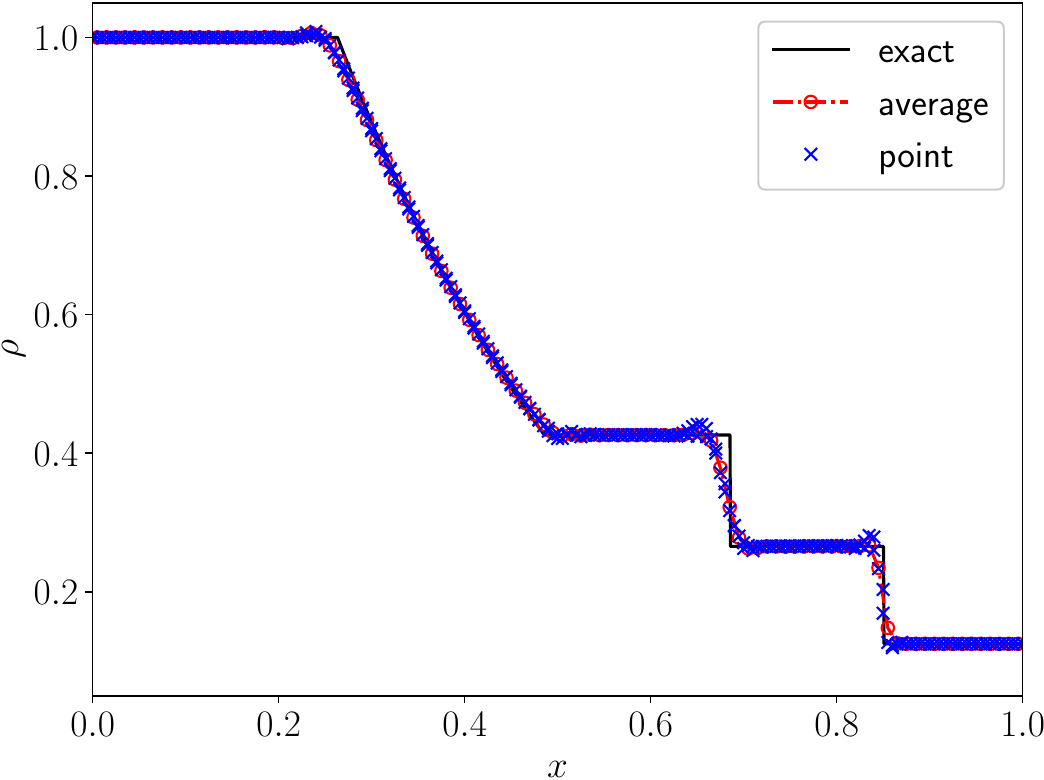}
		\end{subfigure}
		\caption{\Cref{ex:2d_sod}, quasi-2D Sod shock tube.
			The density obtained by using different point value updates without ($\kappa=0$, top row) or with the shock sensor ($\kappa=1$, bottom row).
			From left to tight: JS, LLF, SW, VH.}
		\label{fig:2d_sod_density}
	\end{figure}
\end{example}

\begin{example}[Shock reflection problem]\label{ex:2d_shock_reflection}\rm
	The computational domain is $[0,4]\times[0,1]$, which is divided into a $120\times30$ uniform mesh. The boundary conditions are outflow at the right boundary, reflective at the bottom boundary, and inflow on the other two sides with the data
	\begin{equation*}
		(\rho, v_1, v_2, p) = 
		\begin{cases}
			(1.0, ~2.9, ~0.0, ~1.0/1.4), &\text{if}~ x=0,~ 0\leqslant y\leqslant 1,\\
			(1.69997, ~2.61934, -0.50632, ~1.52819), &\text{if}~ y=1,~ 0\leqslant x\leqslant 4. \\
		\end{cases}
	\end{equation*}
	This test is solved until $T=6$ thus the numerical solution converges.
	
	The density plots obtained without any limiting ($\kappa=0$) and with the shock sensor-based limiting ($\kappa=0.5$) are shown in \Cref{fig:2d_sf_density}, and the blending coefficients based on the shock sensor are plotted in \Cref{fig:2d_sf_ss}.
	The numerical solutions converge in both cases, and the shock sensor can correctly locate the shock waves.
	It is also interesting to look at the residual between two successive time steps, presented in \Cref{fig:2d_sf_residual}, with respective to the number of iterations.
	The limiting based on the shock sensor accelerates the convergence after the reflective shock is fully formed, showing the advantage of using the shock sensor.

	\begin{figure}[htbp!]
		\centering
		\begin{subfigure}[b]{0.48\textwidth}
			\centering
			\includegraphics[width=\linewidth]{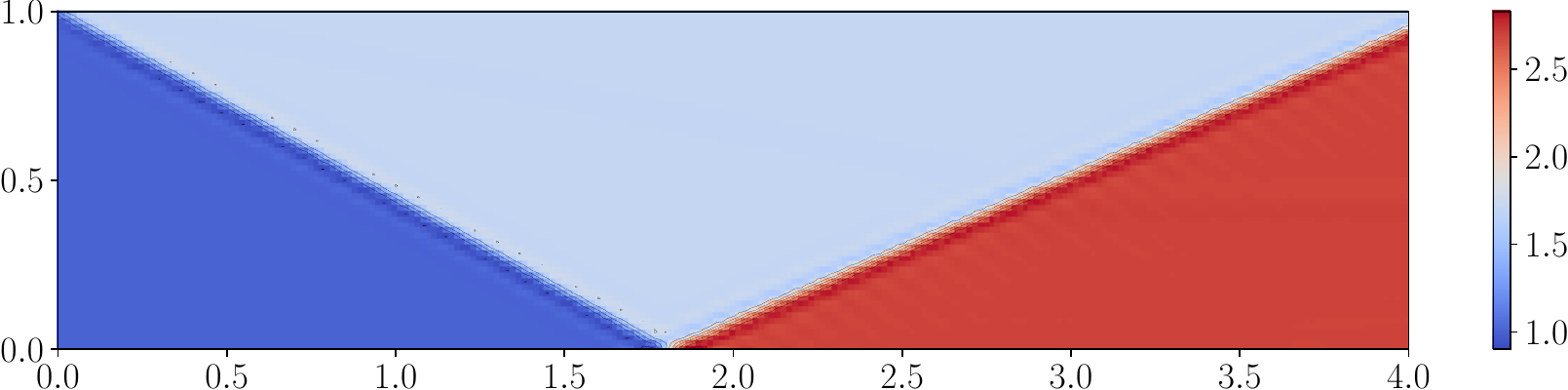}
		\end{subfigure}
		\quad
		\begin{subfigure}[b]{0.48\textwidth}
			\centering
			\includegraphics[width=\linewidth]{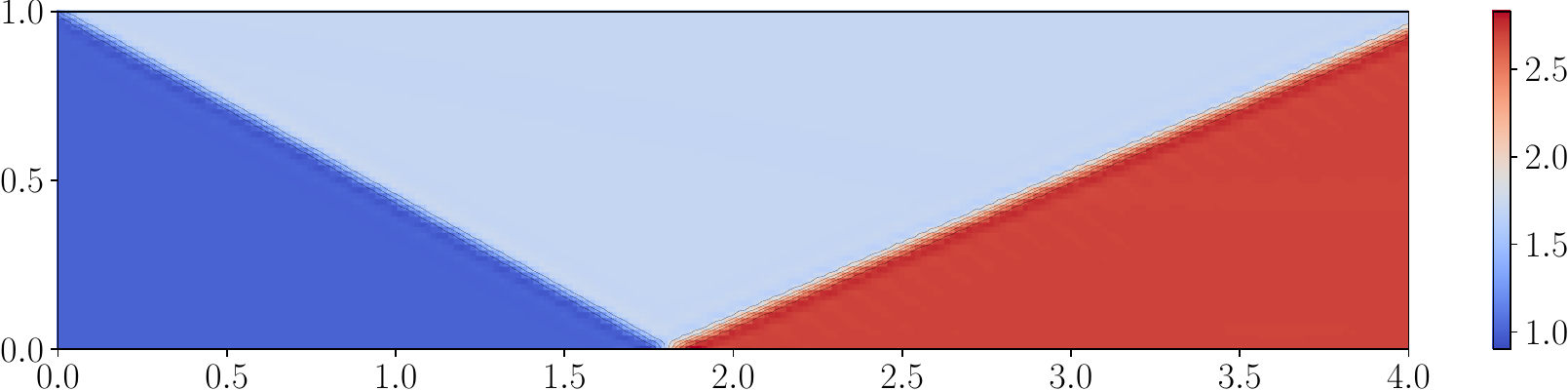}
		\end{subfigure}
		\caption{\Cref{ex:2d_shock_reflection}, shock reflection problem.
			The density obtained without ($\kappa=0$, left) or with the shock sensor ($\kappa=0.5$, right) on the $120\times30$ uniform mesh. $10$ equally spaced contour lines from $0.901$ to $2.829$ are shown.}
		\label{fig:2d_sf_density}
	\end{figure}
	
	\begin{figure}[htbp!]
		\centering
		\begin{subfigure}[b]{0.48\textwidth}
			\includegraphics[width=\linewidth]{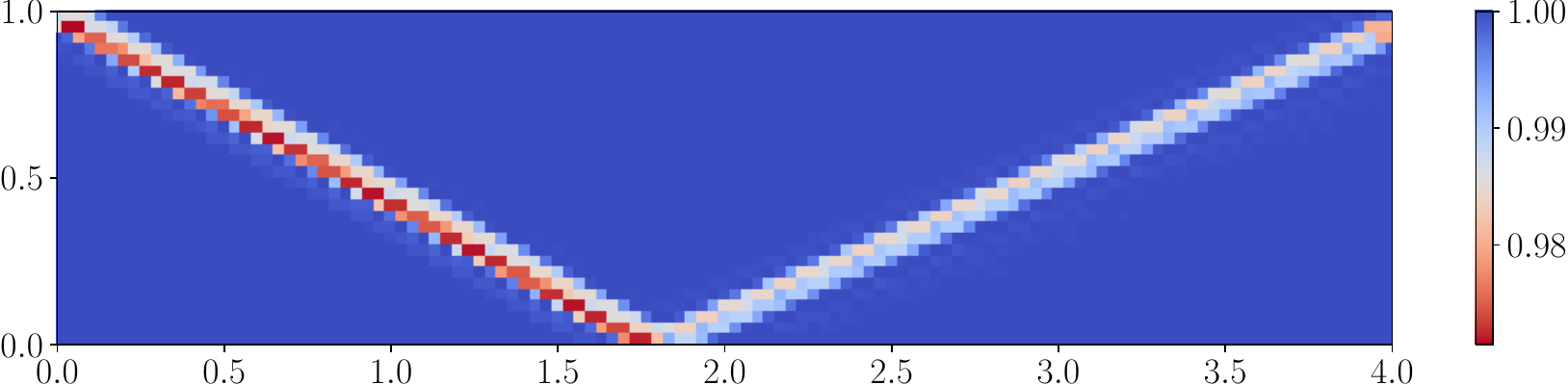}
		\end{subfigure}
		\quad
		\begin{subfigure}[b]{0.48\textwidth}
			\centering
			\includegraphics[width=\linewidth]{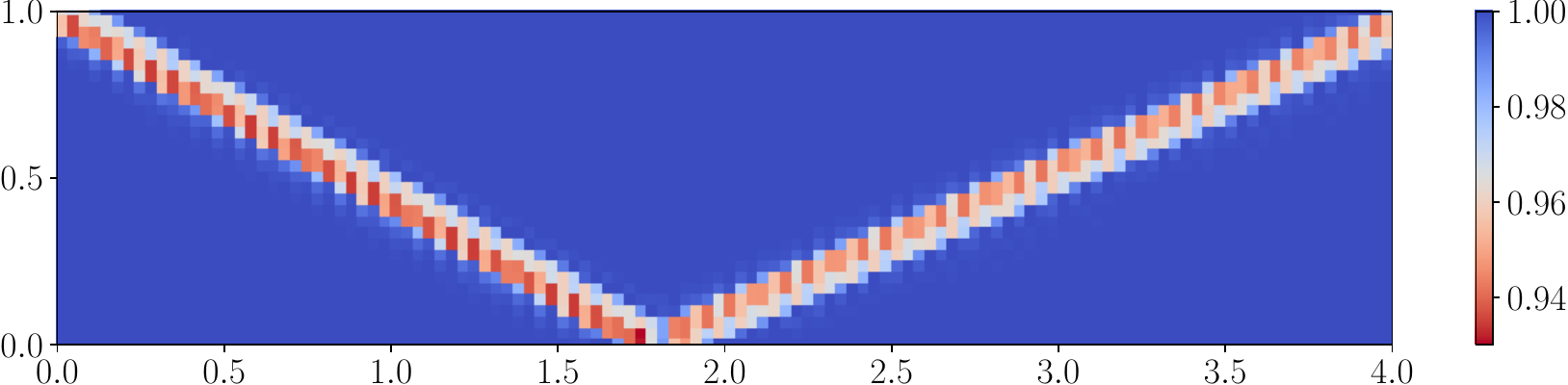}
		\end{subfigure}
		\caption{\Cref{ex:2d_shock_reflection}, shock reflection problem.
			The shock sensor-based blending coefficients $\theta_{\xr,j}^{s}$ (left) and $\theta_{i,\yr}^{s}$ (right) on the $120\times30$ uniform mesh. $\kappa=0.5$.}
		\label{fig:2d_sf_ss}
	\end{figure}
	
	\begin{figure}[htbp!]
		\centering
		\includegraphics[width=0.4\linewidth]{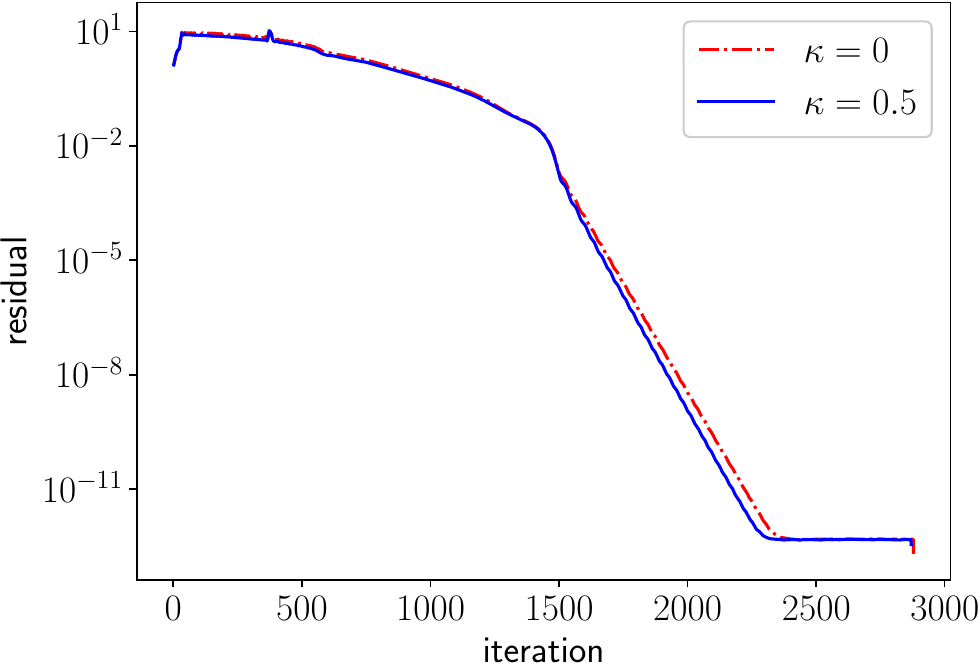}
		\caption{\Cref{ex:2d_shock_reflection}, shock reflection problem.
			The residual decay with respect to the number of iterations.}
		\label{fig:2d_sf_residual}
	\end{figure}
	
\end{example}

\begin{example}[Sedov blast wave]\label{ex:2d_sedov}\rm
	In this problem, a volume of uniform density and temperature is initialized, and a large quantity of thermal energy is injected at the center, developing into a blast wave that evolves in time in a self-similar fashion \cite{Sedov_1959_Similarity_book}.
	An exact analytical solution based on self-similarity arguments is available \cite{Kamm_2007_efficient}, which contains very low
	density with strong shocks.
	The computational domain is $[-1.1,1.1]\times[-1.1,1.1]$ with outflow boundary conditions.
	The initial density is one, velocity is zero, and total energy is $10^{-12}$ everywhere except that in the centered cell, the total energy of the cell average is $\frac{0.979264}{\Delta x\Delta y}$ with $\Delta x = 2.2/N_1, \Delta y = 2.2/N_2$, which is used to emulate a $\delta$-function at the center.
	
	This test is solved until $T=1$ with $101\times101$ and $201\times201$ uniform meshes, and the BP limitings are used, otherwise, the simulation fails due to negative pressure.
	The density plots with the shock sensor ($\kappa=0.5$) are shown in \Cref{fig:2d_sedov_density}, from which one can observe the well-captured circular shock wave.
        The cut-line along the diagonal of the domain $y=x$ in \Cref{fig:2d_sedov_density} demonstrates that the numerical solutions converge to the exact solution without spurious oscillations.
	The blending coefficients based on the shock sensor are presented in \Cref{fig:2d_sedov_ss}, verifying that this limiting is only activated at the shock wave.
	
\begin{figure}[htbp!]
	\centering
	\begin{subfigure}[b]{0.3\textwidth}
		\centering
		\includegraphics[width=\linewidth]{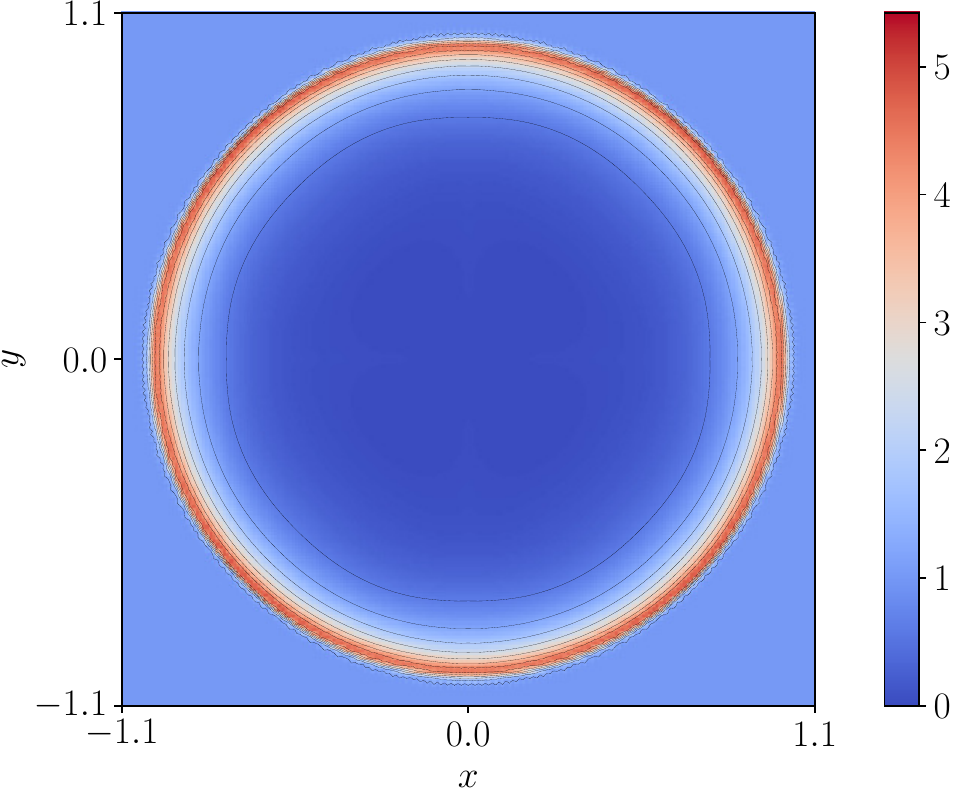}
	\end{subfigure}
	\begin{subfigure}[b]{0.3\textwidth}
		\centering
		\includegraphics[width=\linewidth]{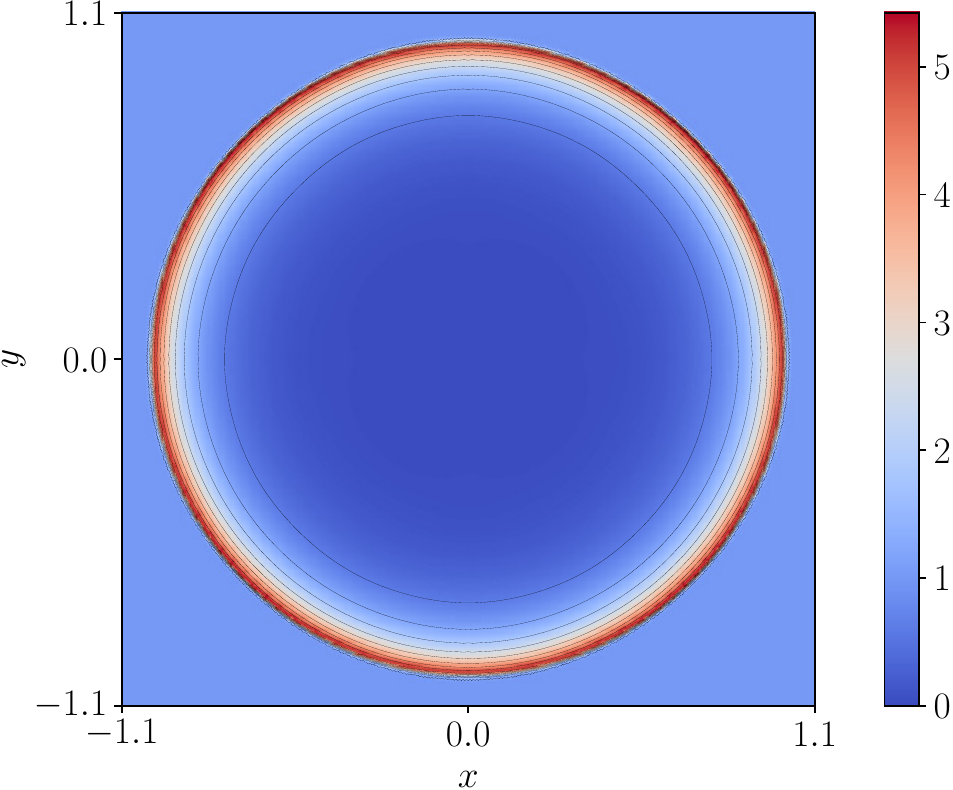}
	\end{subfigure}
        \begin{subfigure}[b]{0.3\textwidth}
		\centering
		\includegraphics[width=0.81\linewidth]{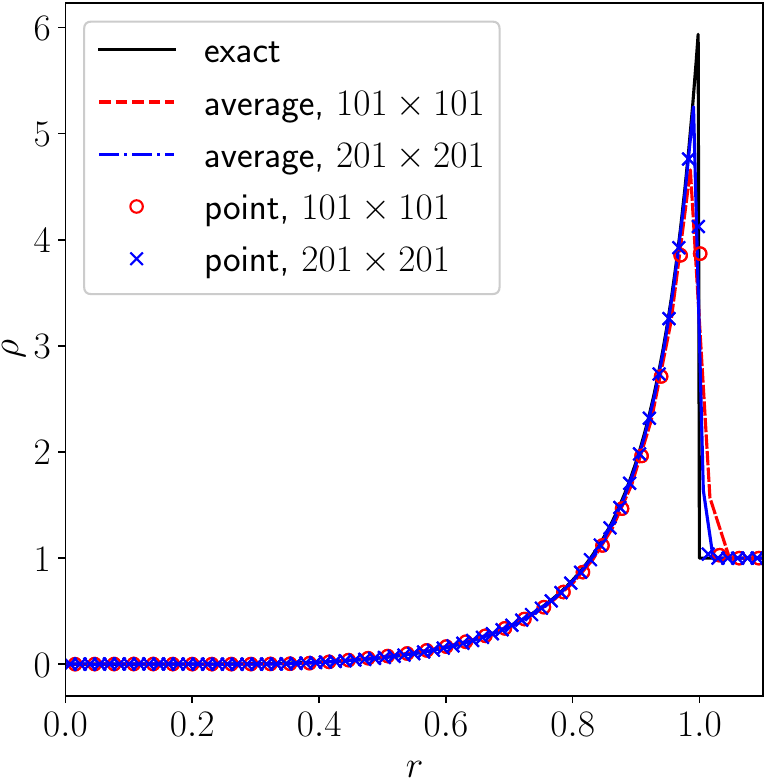}
	\end{subfigure}
	\caption{\Cref{ex:2d_sedov}, 2D Sedov blast wave.
		The density plots obtained with the BP limitings and the shock sensor ($\kappa=0.5$).
		From left to right:	$10$ equally spaced contour lines from $0$ to $5.424$ on the uniform $101\times101$ and $201\times201$ meshes, respectively, cut-line along $y=x$. }
	\label{fig:2d_sedov_density}
\end{figure}

\begin{figure}[htbp!]
	\centering
	\begin{subfigure}[b]{0.3\textwidth}
		\centering
		\includegraphics[width=\linewidth]{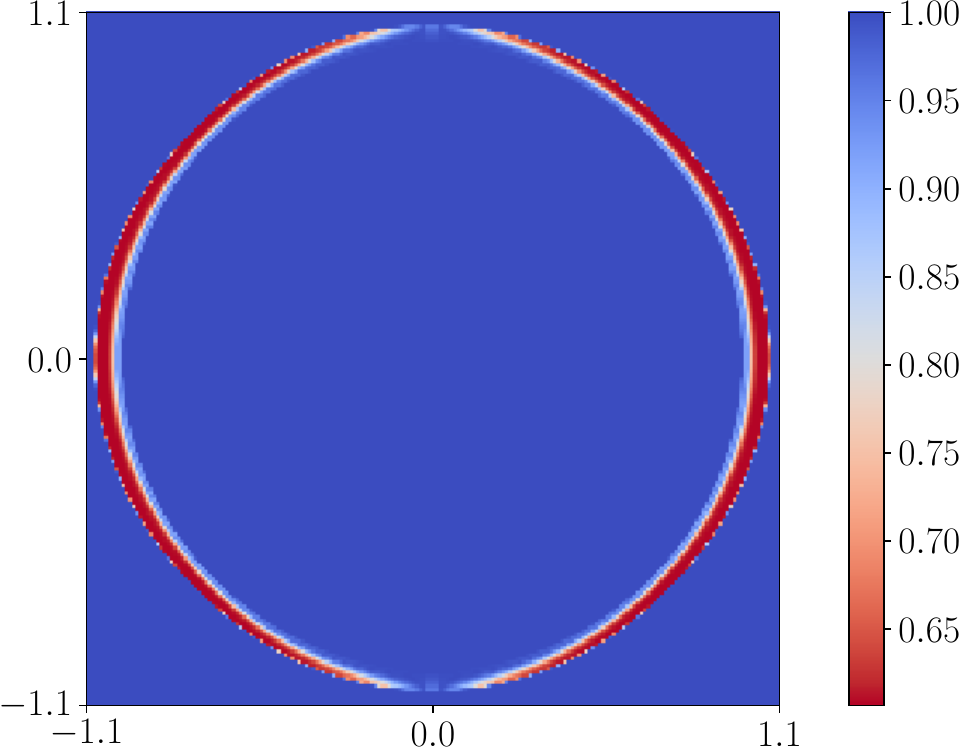}
	\end{subfigure}
	\qquad\qquad\qquad
	\begin{subfigure}[b]{0.3\textwidth}
		\centering
		\includegraphics[width=\linewidth]{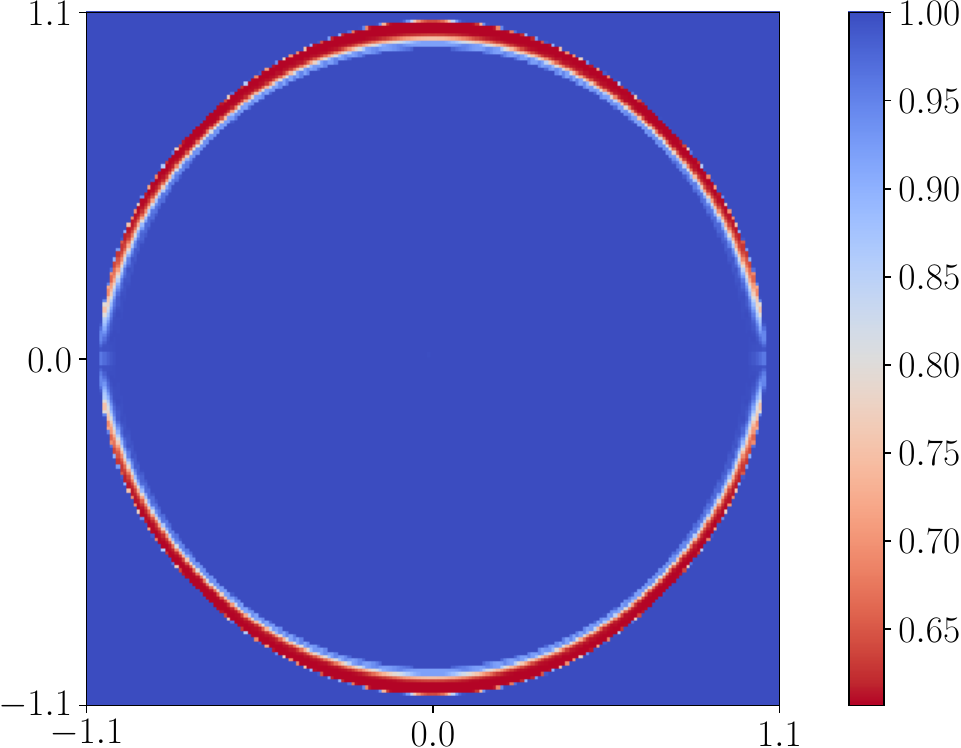}
	\end{subfigure}
	\caption{\Cref{ex:2d_sedov}, 2D Sedov blast wave.
		The shock sensor-based blending coefficients $\theta_{\xr,j}^{s}$ (left) and $\theta_{i,\yr}^{s}$ (right) on the $201\times201$ uniform mesh.}
	\label{fig:2d_sedov_ss}
\end{figure}

\end{example}

\begin{example}[2D Riemann problem]\label{ex:2d_rp}\rm
	This problem corresponds to the configuration $3$ in \cite{Lax_1998_Solution_SJoSC}, containing four initial shock waves, with the initial data
	\begin{equation*}
		(\rho, v_1, v_2, p) = 
		\begin{cases}
			(1.5, ~0, ~0, ~1.5), & x > 0.8, ~y > 0.8, \\
			(0.5323, ~1.206, ~0, ~0.3), & x < 0.8, ~y > 0.8, \\
			(0.138, ~1.206, ~1.206, ~0.029), & x < 0.8, ~y < 0.8, \\
			(0.5323, ~0, ~1.206, ~0.3), & x > 0.8, ~y < 0.8. \\
		\end{cases}
	\end{equation*}
	The test is solved on the domain $[0,1]\times[0,1]$ until $T=0.8$.
	
	Without the BP limitings, the simulation crashes due to negative pressure.
	The density plots obtained without ($\kappa=0$) and with the shock sensor ($\kappa=0.5$) are shown in \Cref{fig:2d_rp_density}.
	Without the shock sensor, the numerical solutions contain spurious oscillations, and they are reduced drastically by the shock sensor-based limiting.
	As mesh refinement, the shock waves are captured sharply, and the small-scale features are preserved well, as evidenced by the roll-ups around the mushroom-shaped jet, which are in good agreement with the results in the literature.
	The values of the shock sensor-based blending coefficients $\theta_{\xr,j}, \theta_{i,\yr}$ are also plotted in \Cref{fig:2d_rp_ss}, which indicates that the shock sensor can locate the shock waves correctly.
	
	\begin{figure}[htbp!]
		\begin{subfigure}[b]{0.3\textwidth}
			\centering
			\includegraphics[width=\linewidth]{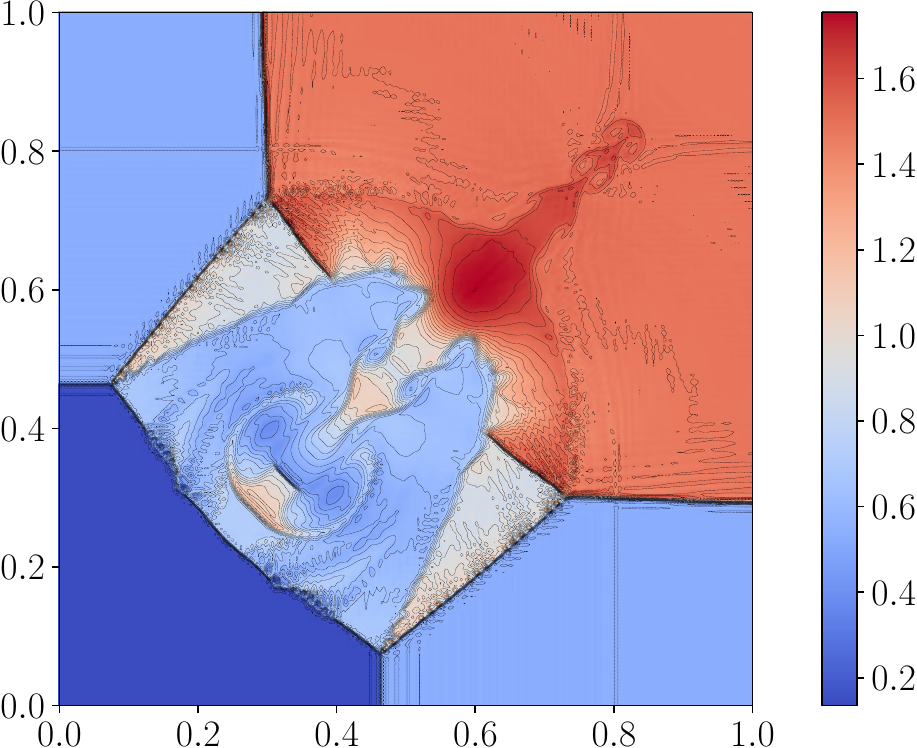}
		\end{subfigure}
		~
		\begin{subfigure}[b]{0.3\textwidth}
			\centering
			\includegraphics[width=\linewidth]{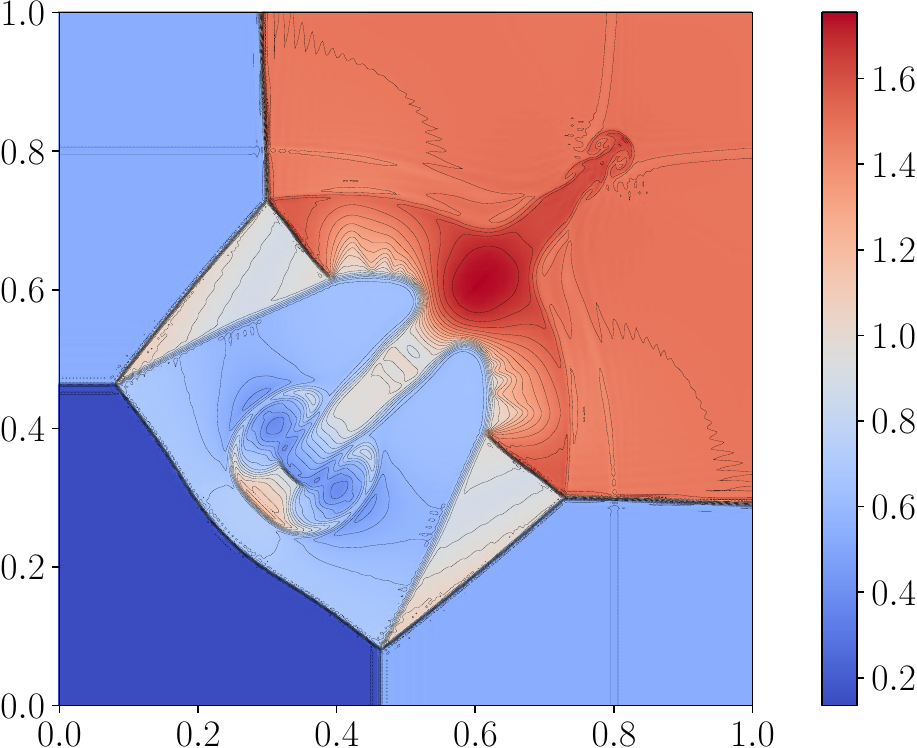}
		\end{subfigure}
		~
		\begin{subfigure}[b]{0.3\textwidth}
			\centering
			\includegraphics[width=\linewidth]{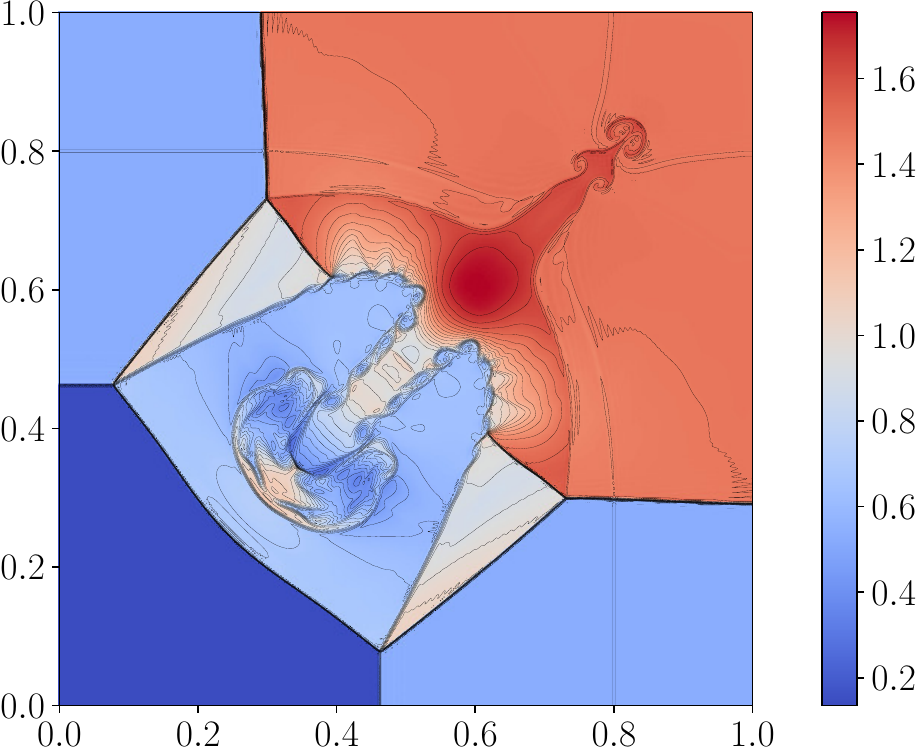}
		\end{subfigure}
		\caption{\Cref{ex:2d_rp}, 2D Riemann problem.
			The density obtained with the BP limitings and without or with the shock sensor.
			From left to right: $200\times200$ mesh with $\kappa=0$, $200\times200$ mesh with $\kappa=0.5$, $400\times400$ mesh with $\kappa=0.5$.
			$30$ equally spaced contour lines from $0.135$ to $1.754$. }
		\label{fig:2d_rp_density}
	\end{figure}
	
	\begin{figure}[htbp!]
		\centering
		\begin{subfigure}[b]{0.3\textwidth}
			\centering
			\includegraphics[width=\linewidth]{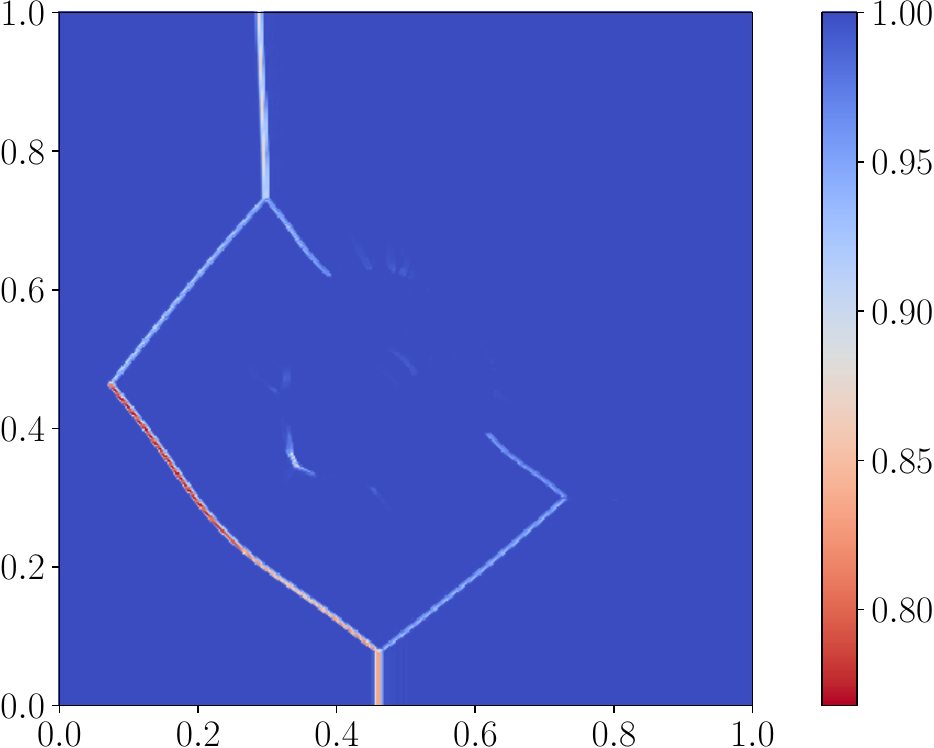}
		\end{subfigure}
		\qquad\qquad\qquad
		\begin{subfigure}[b]{0.3\textwidth}
			\centering
			\includegraphics[width=\linewidth]{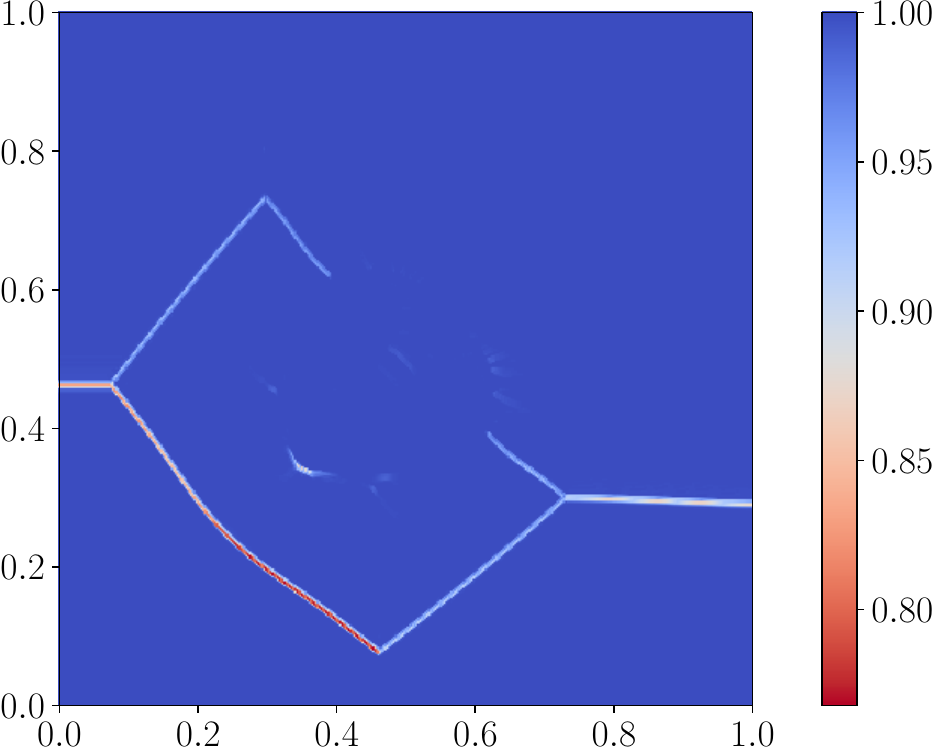}
		\end{subfigure}
		\caption{\Cref{ex:2d_rp}, 2D Riemann problem.
			The shock sensor-based blending coefficients $\theta_{\xr,j}^{s}$ (left) and $\theta_{i,\yr}^{s}$ (right) on the $400\times400$ uniform mesh.}
		\label{fig:2d_rp_ss}
	\end{figure}

\end{example}

\begin{example}[Double Mach reflection]\label{ex:2d_dmr}\rm
	The computational domain is $[0, 3]\times[0, 1]$ with a reflective wall at the bottom starting from $x = 1/6$.
	A Mach $10$ shock is moving towards the bottom wall with an angle of $\pi/6$.
	The pre- and post-shock states are
	\begin{equation*}
		(\rho, v_1, v_2, p) = 
		\begin{cases}
			(1.4, ~0, ~0, ~1), & x \geqslant 1/6 + (y+20t)/\sqrt{3}, \\
			(8, ~8.25\cos(\pi/6), -8.25\sin(\pi/6), ~116.5), & x < 1/6 + (y+20t)/\sqrt{3}. \\
		\end{cases}
	\end{equation*}
	The reflective boundary condition is applied at the wall, while the exact post-shock condition is imposed at the left boundary and for the rest of the bottom boundary (from $x = 0$ to $x = 1/6$).
	At the top boundary, the exact motion of the Mach $10$ shock is applied and the outflow boundary condition is used at the right boundary.
	The results are shown at $T = 0.2$.
	
	The AF method without the BP limitings gives negative density or pressure near the reflective location $(1/6, 0)$, so the BP limitings are necessary for this test.
	The density plots with enlarged views around the double Mach region are shown in \Cref{fig:2d_dmr_density} without ($\kappa=0$) and with the shock sensor ($\kappa=1$) on a series of uniform meshes, and the blending coefficients based on the shock sensor are shown in \Cref{fig:2d_dmr_ss}.
	When the shock sensor is not activated, the noise after the bow shock is obvious, and it is damped with the help of the shock sensor.
	As mesh refinement, the numerical solutions converge with a good resolution and are comparable to those in the literature.
    Compared to the third-order $P^2$ DG method using the TVB limiter \cite{Cockburn_2001_Runge_JoSC} with the same mesh resolution ($\Delta x=\Delta y=1/480$), the roll-ups and vortices are comparable while the AF method uses fewer DoFs ($4$ versus $6$ per cell).
	
	\begin{figure}[htbp!]
		\centering
		\begin{subfigure}[b]{0.6\textwidth}
			\centering	\includegraphics[width=\linewidth]{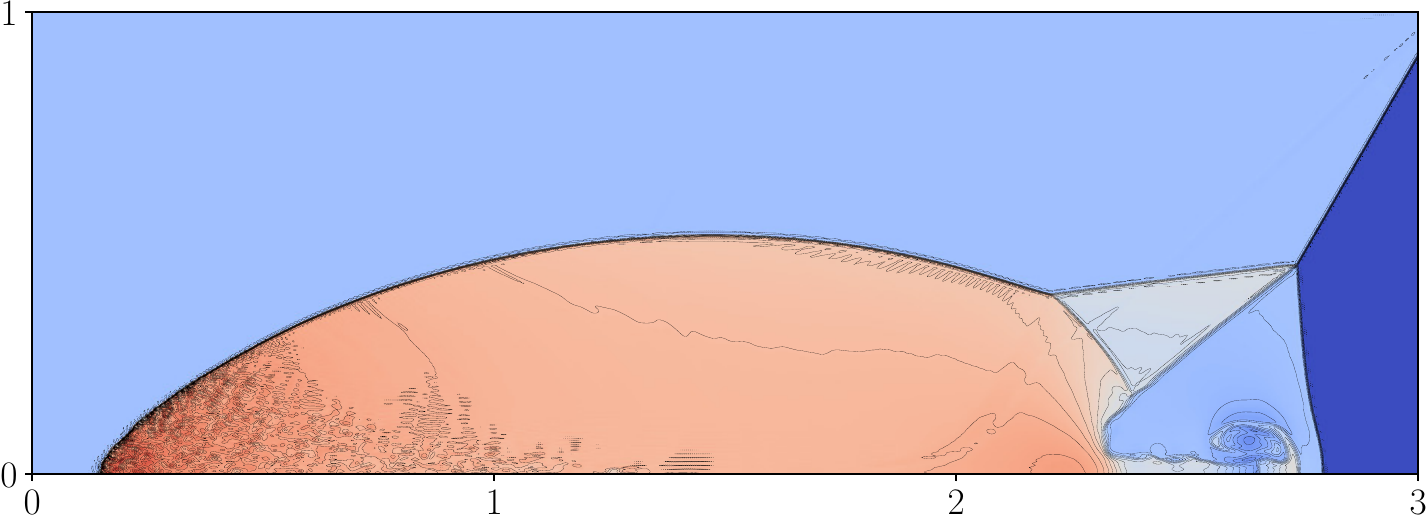}
		\end{subfigure}
  \begin{subfigure}[b]{0.384\textwidth}
			\centering
			\includegraphics[width=\linewidth]{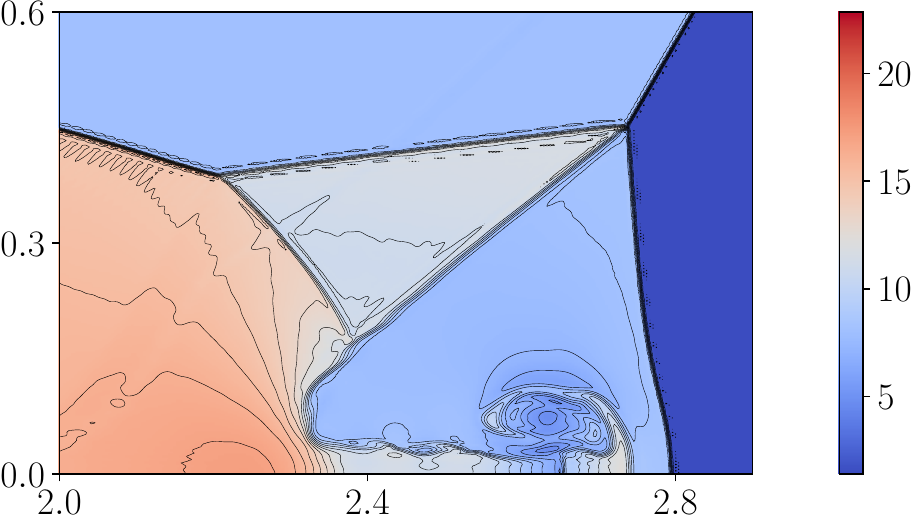}
		\end{subfigure}
		
		\begin{subfigure}[b]{0.6\textwidth}
			\centering
			\includegraphics[width=\linewidth]{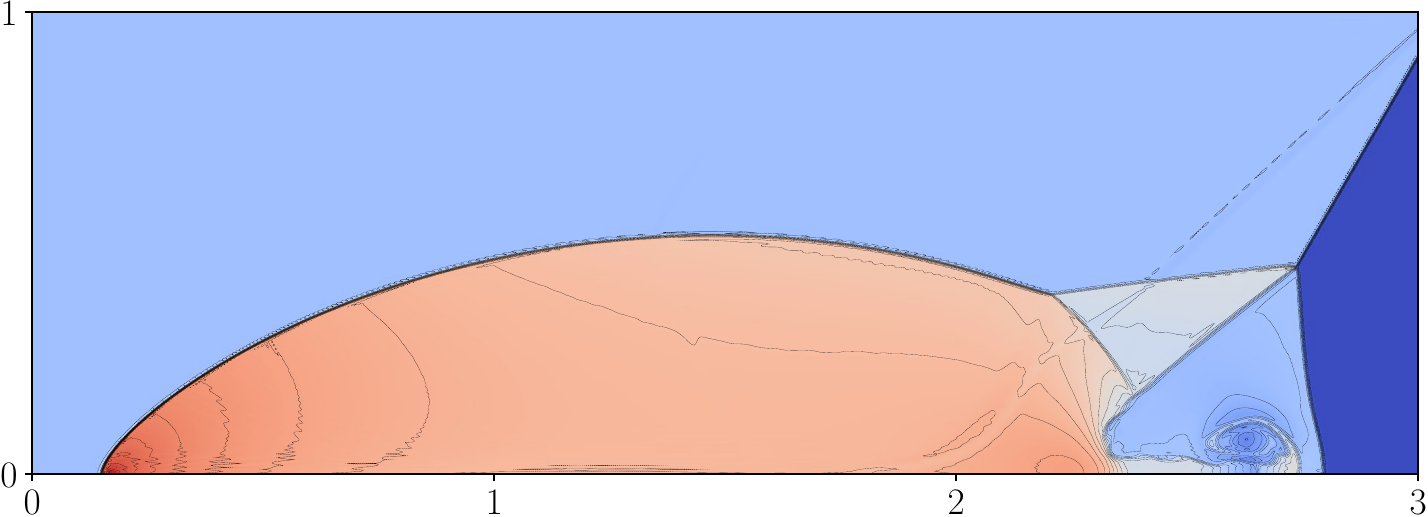}
		\end{subfigure}
  \begin{subfigure}[b]{0.384\textwidth}
			\centering
			\includegraphics[width=\linewidth]{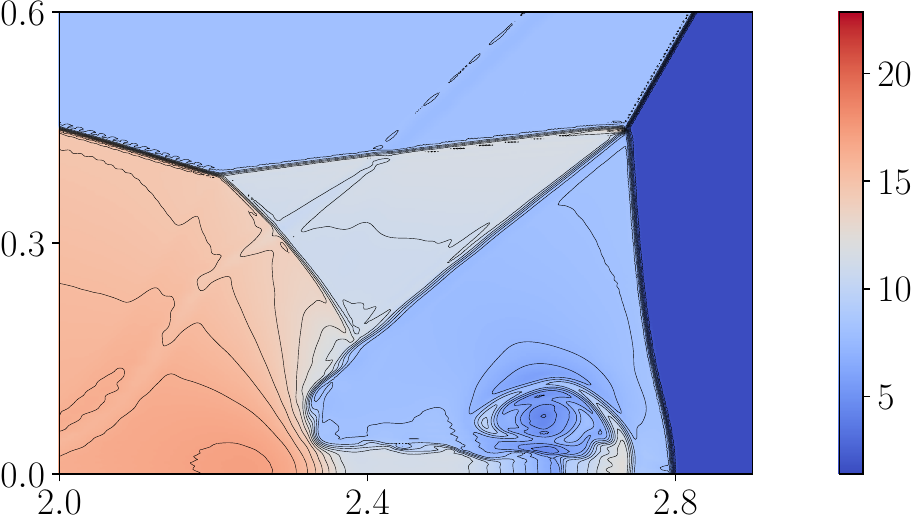}
		\end{subfigure}

      \begin{subfigure}[b]{0.6\textwidth}
			    \includegraphics[width=\linewidth]{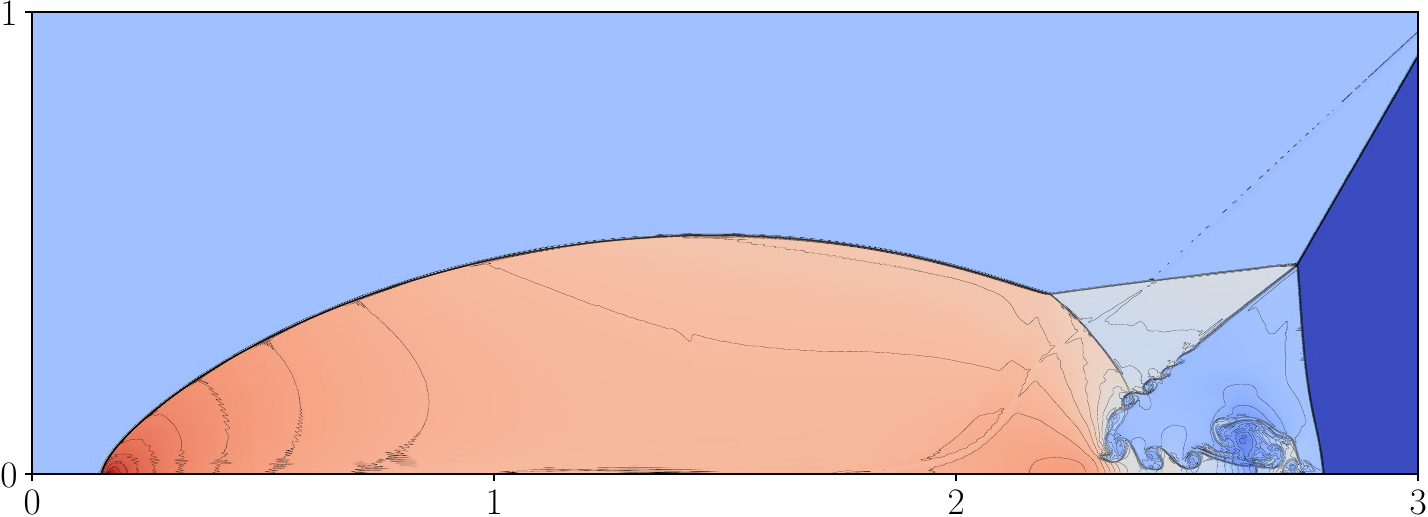}
		\end{subfigure}
  \begin{subfigure}[b]{0.384\textwidth}
			\centering
			\includegraphics[width=\linewidth]{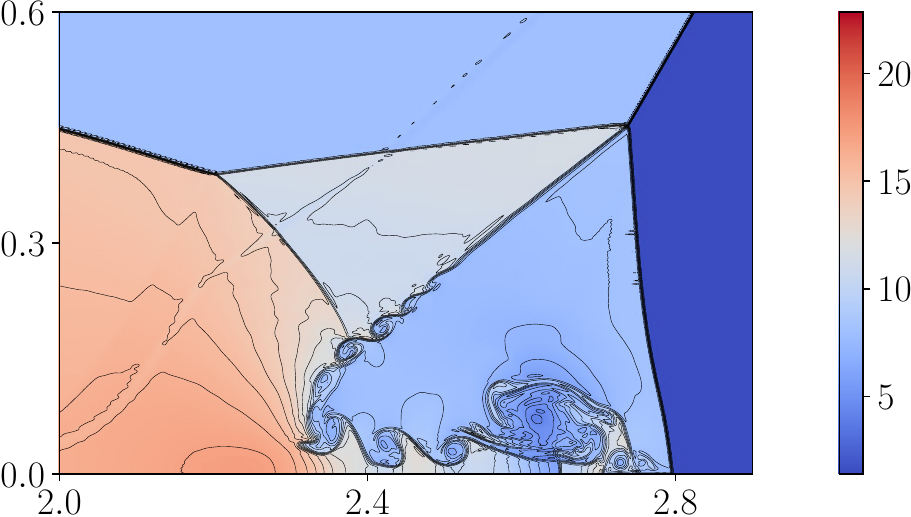}
		\end{subfigure}
      
		\caption{\Cref{ex:2d_dmr}, double Mach reflection.
			The density obtained with the BP limitings and without or with the shock sensor.
			From top to bottom: $720\times240$ mesh with $\kappa=0$, $720\times240$ mesh with $\kappa=1$, $1440\times480$ mesh with $\kappa=1$.
			$30$ equally spaced contour lines from $1.390$ to $22.861$.}
		\label{fig:2d_dmr_density}
	\end{figure}

 	\begin{figure}[htbp!]
		\centering		
		\begin{subfigure}[b]{0.48\textwidth}
			\centering
			\includegraphics[width=\linewidth]{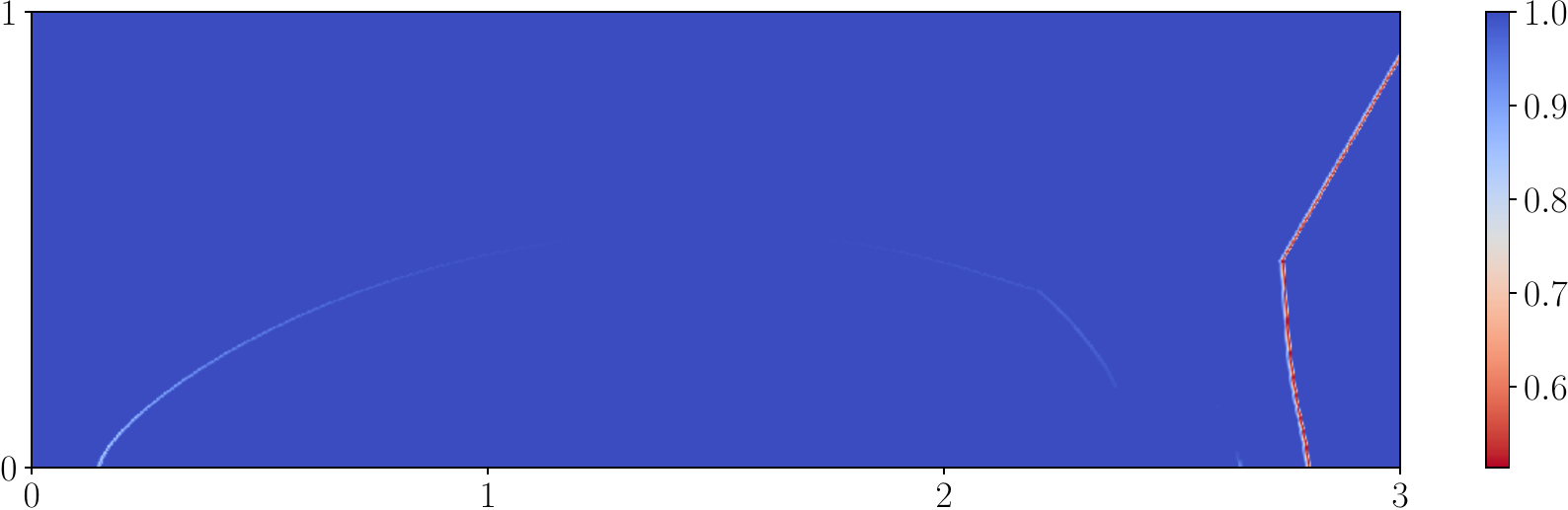}
		\end{subfigure}
		\quad
		\begin{subfigure}[b]{0.48\textwidth}
			\centering
			\includegraphics[width=\linewidth]{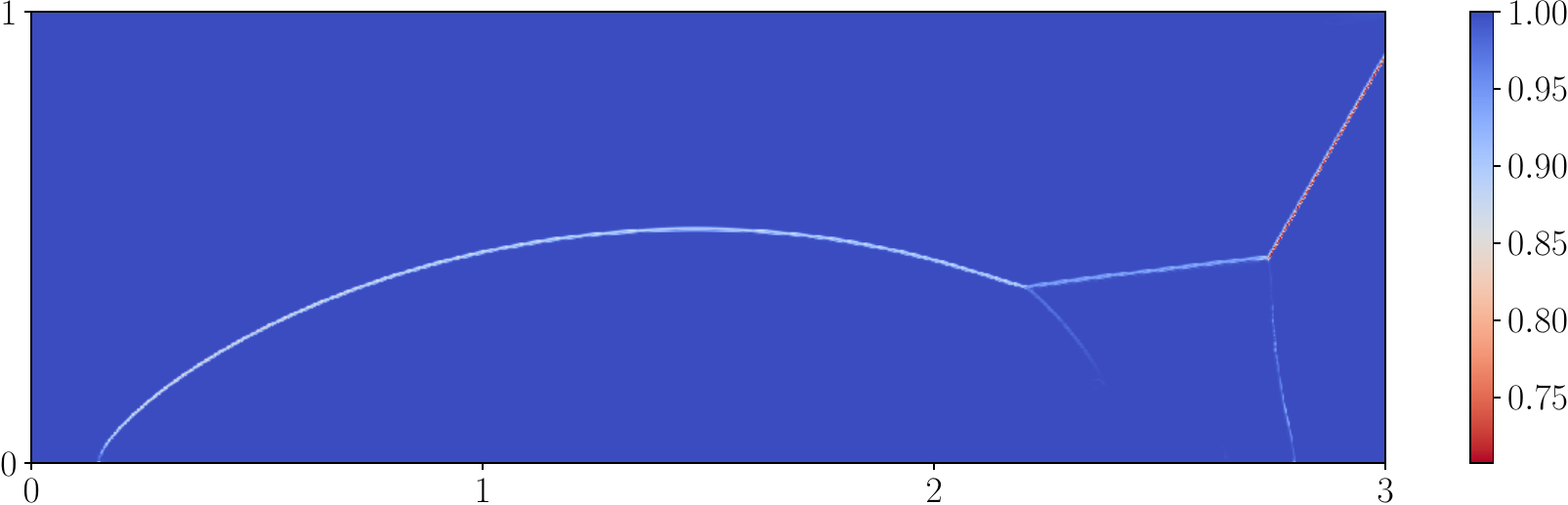}
		\end{subfigure}
		\caption{\Cref{ex:2d_dmr}, double Mach reflection.
			The blending coefficients $\theta_{\xr,j}^{s}$ (left) and $\theta_{i,\yr}^{s}$ (right) based on the shock sensor with $\kappa=1$ on the $1440\times 480$ mesh.}
		\label{fig:2d_dmr_ss}
	\end{figure}
\end{example}

\begin{example}[A Mach 3 wind tunnel with a forward-facing step]\label{ex:2d_ffs}\rm
	The initial condition is a Mach 3 flow $(\rho, v_1, v_2, p) = (1.4, 3, 0, 1)$.
	The computational domain is $[0,3]\times[0,1]$ and the step is of height $0.2$ located from $x=0.6$ to $x=3$.
	The inflow and outflow boundary conditions are applied at the left and right boundaries, respectively, and the reflective boundary conditions are imposed at other boundaries.
	Due to the reflections and interactions of the shocks, a triple point is formed, from which a trail of vortices moves to the right.
	The resolution of those vortices is usually used to examine the numerical methods.
	The output time is $T=4$.
	
	The BP limitings are necessary, otherwise the simulation crashes due to negative density or pressure near the corner.
	The density obtained without ($\kappa=0$) and with the shock sensor ($\kappa=1$) on different meshes are plotted in \Cref{fig:2d_ffs_density_dx}, and the blending coefficients based on the shock sensor are presented in \Cref{fig:2d_ffs_ss}.
	The results clearly show that our method can capture the main features and well-developed Kelvin–Helmholtz roll-ups that originate from the triple point.
	The noise after the shock waves is reduced by the shock sensor-based limiting, while the roll-ups are preserved well.
	The spurious entropy production near the corner triggers a boundary layer above the step, which is also observed in other methods in the literature.
	The entropy fix near the corner of the step introduced in \cite{Woodward_1984_numerical_JoCP} cannot be directly applied to the AF method, because the treatment of the point value on the wall is sensitive.
	In this test, the shock sensor-based limiting is not activated in the region around the step in the $x$-direction, which helps to reduce the boundary layer.
	To be specific, in the two layers of cells above the step and five layers of cells to the left of the step, $\theta_{\xr,j}^{s}$ is set as $1$.
	Compared to the results obtained by the third-order $P^2$ DG method with the TVB limiter \cite{Cockburn_2001_Runge_JoSC}, the vorticities are better captured with the same mesh size $\Delta x=\Delta y=1/160, 1/320$.
	Note that the AF method allows a larger CFL number ($0.4$ versus $0.2$ for linear stability) and uses fewer DoFs, showing the efficiency and potential to capture small-scale features in high Mach number flows.
	
	\begin{figure}[htbp!]
		\centering
		\begin{subfigure}[b]{0.6\textwidth}
			\centering
			\includegraphics[width=\linewidth]{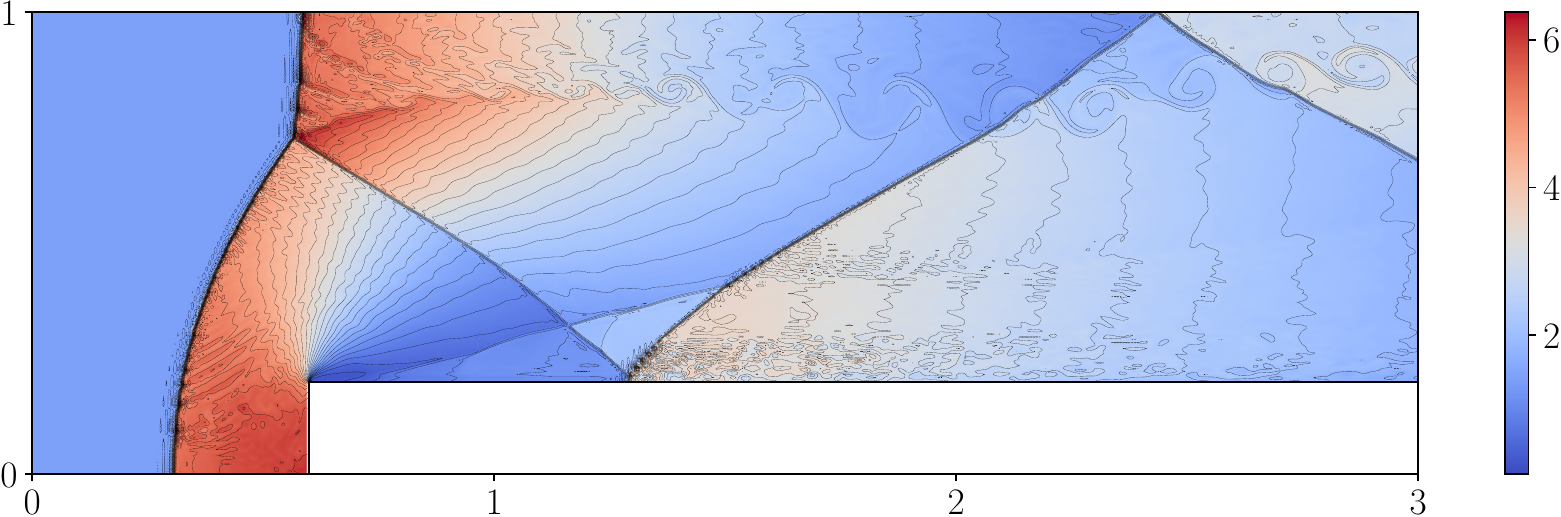}
		\end{subfigure}
		
		\begin{subfigure}[b]{0.6\textwidth}
			\centering
			\includegraphics[width=\linewidth]{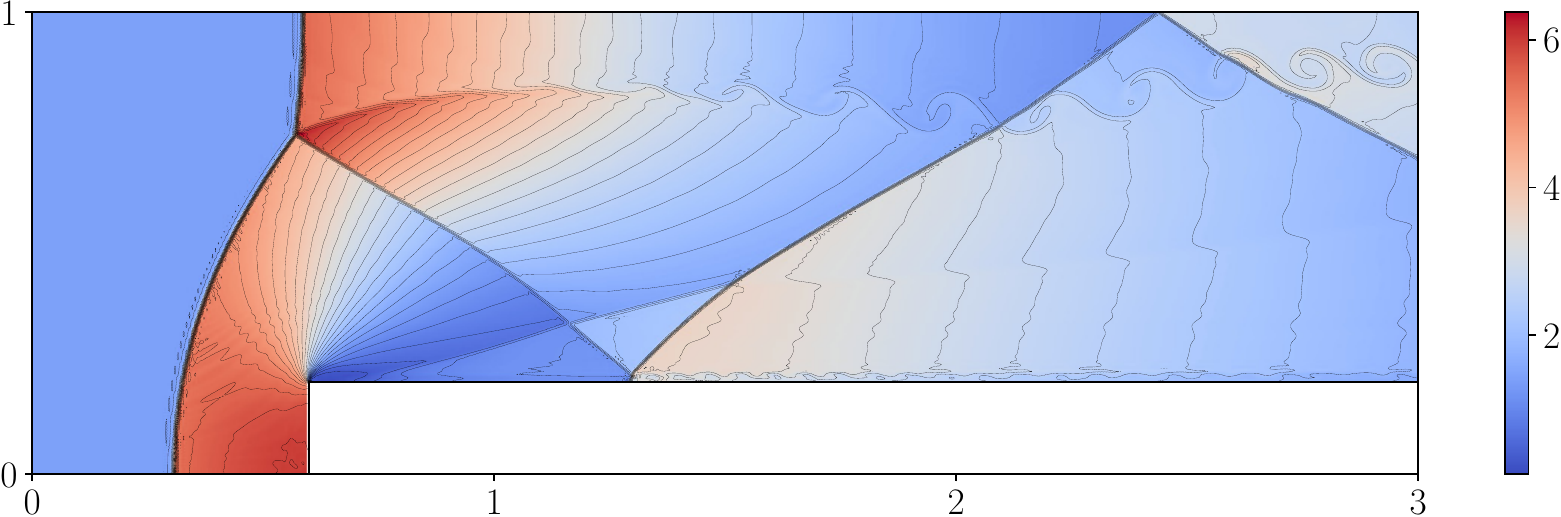}
		\end{subfigure}

		\begin{subfigure}[b]{0.6\textwidth}
			\centering
			\includegraphics[width=\linewidth]{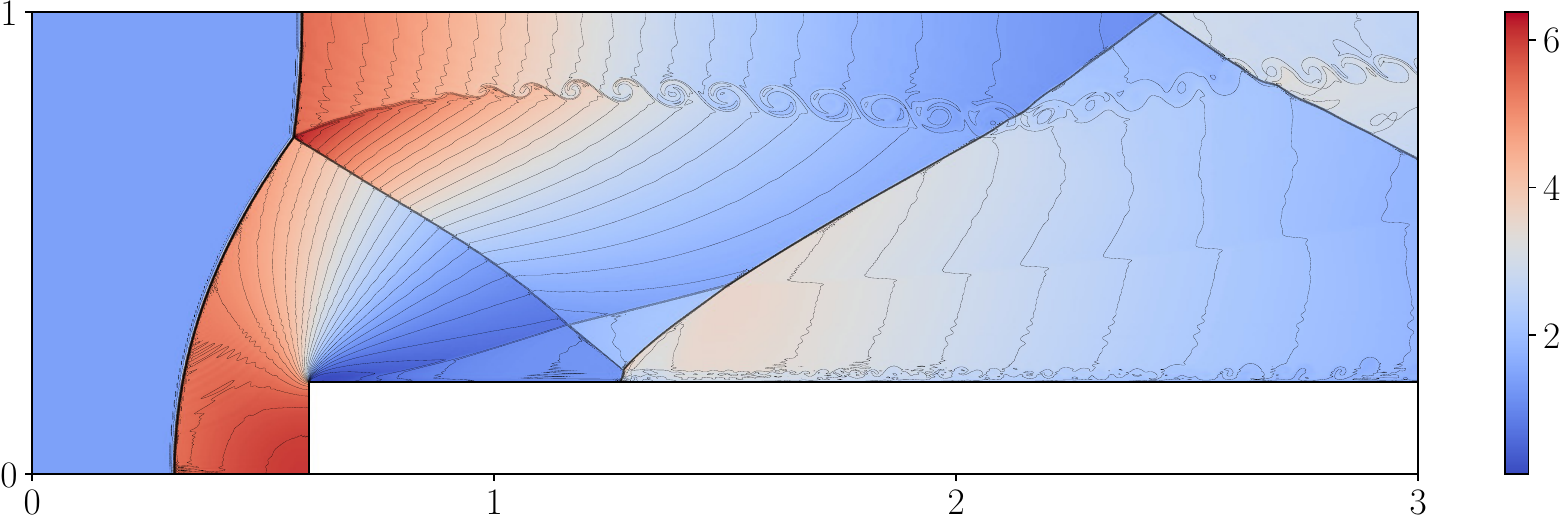}
		\end{subfigure}
		\caption{\Cref{ex:2d_ffs}, forward-facing step problem.
			The density obtained with the BP limitings and without or with the shock sensor.
			From top to bottom: $480\times160$ mesh with $\kappa=0$, $480\times160$ mesh with $\kappa=1$, 	$960\times320$ mesh with $\kappa=1$.
			$30$ equally spaced contour lines from $0.115$ to $6.378$.}
		\label{fig:2d_ffs_density_dx}
	\end{figure}
	
	\begin{figure}[htbp!]
		\centering		
		\begin{subfigure}[b]{0.48\textwidth}
			\centering
			\includegraphics[width=\linewidth]{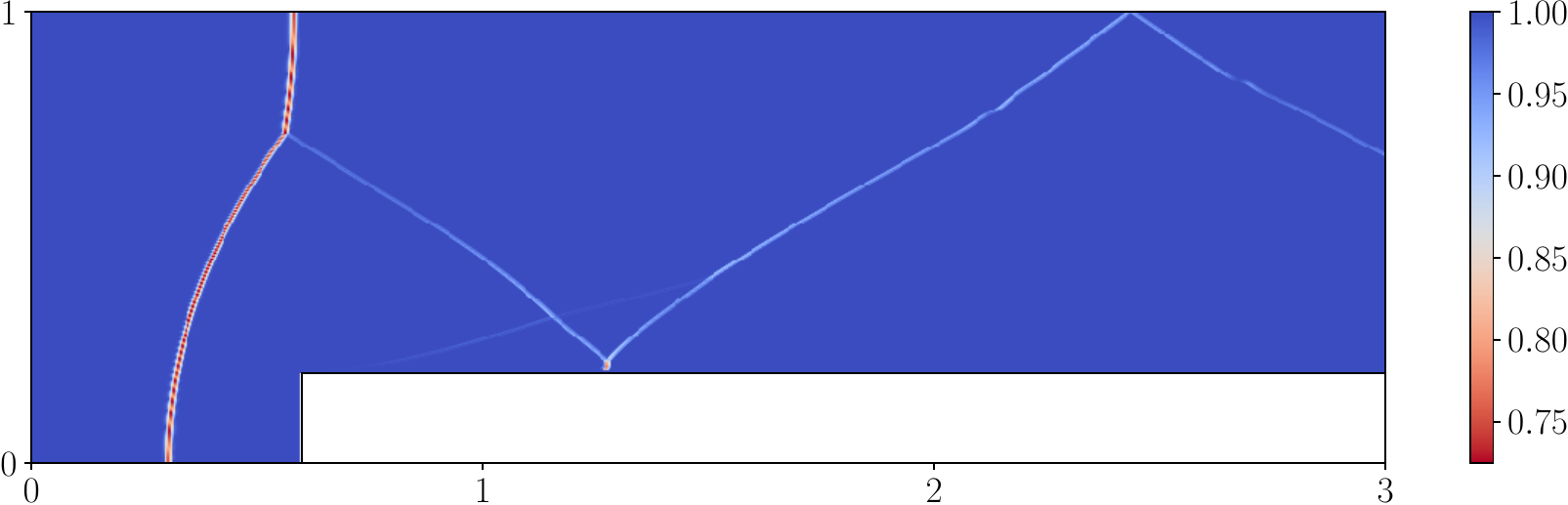}
		\end{subfigure}
		\quad
		\begin{subfigure}[b]{0.48\textwidth}
			\centering
			\includegraphics[width=\linewidth]{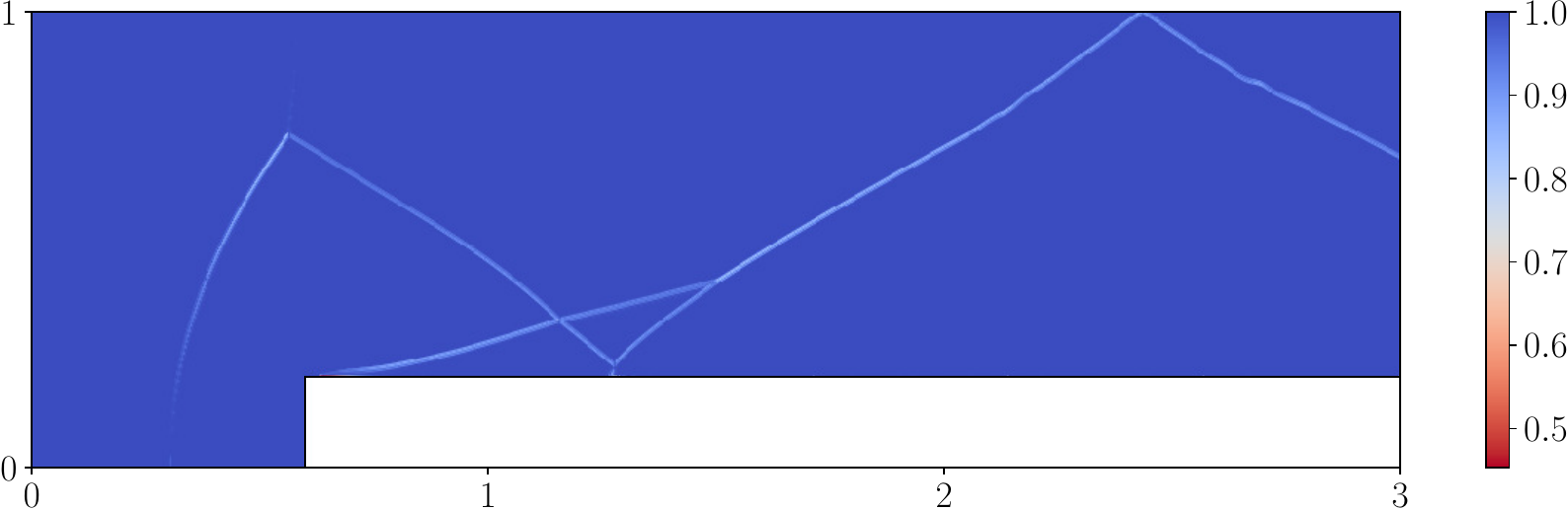}
		\end{subfigure}
		\caption{\Cref{ex:2d_ffs}, forward-facing step problem.
			The blending coefficients $\theta_{\xr,j}^{s}$ (left) and $\theta_{i,\yr}^{s}$ (right) based on the shock sensor with $\kappa=1$ on the $960\times 320$ mesh.}
		\label{fig:2d_ffs_ss}
	\end{figure}
\end{example}

\begin{example}[High Mach number astrophysical jets]\label{ex:2d_jet}\rm
	This test simulates two astrophysical jets, following the setup in \cite{Zhang_2010_positivity_JoCP}.
        The BP limitings are necessary for this kind of problem involving vacuum or near vacuum, and extremely large velocity.
        The first case considers a Mach $80$ jet on a computational domain $[0,2]\times[-0.5,0.5]$, initially filled with ambient gas with $(\rho, v_1, v_2, p) = (0.5, 0, 0, 0.4127)$.
	The outflow boundary conditions are applied at the right, top, and bottom boundaries, while the inflow boundary condition is imposed at the left boundary, $(\rho, v_1, v_2, p) = (5, 30, 0, 0.4127)$ if $\abs{y}<0.05$ and $(\rho, v_1, v_2, p) = (0.5, 0, 0, 0.4127)$ otherwise.
        The second case considers a Mach $2000$ jet on a computational domain $[0,1]\times[-0.25,0.25]$.
        The initial condition and boundary conditions are the same as the first case except that the inflow condition is $(\rho, v_1, v_2, p) = (5, 800, 0, 0.4127)$ if $\abs{y}<0.05$ and $(\rho, v_1, v_2, p) = (0.5, 0, 0, 0.4127)$ otherwise.
        The adiabatic index is $\gamma=5/3$ in this test, and 
        the output time is $0.07$ and $0.001$ for the two cases, respectively.
	
	The logarithm of the density and pressure obtained by the AF methods with the shock sensor ($\kappa=1$ for Mach $80$ and $\kappa=10$ for Mach $2000$) on the uniform $400\times200$ mesh are shown in \Cref{fig:2d_jet}.
	The main flow structures and small-scale features are captured well, comparable to those in \cite{Zhang_2010_positivity_JoCP}.
	
	\begin{figure}[htbp!]
		\centering
		\begin{subfigure}[b]{0.48\textwidth}
			\centering
			\includegraphics[width=\linewidth]{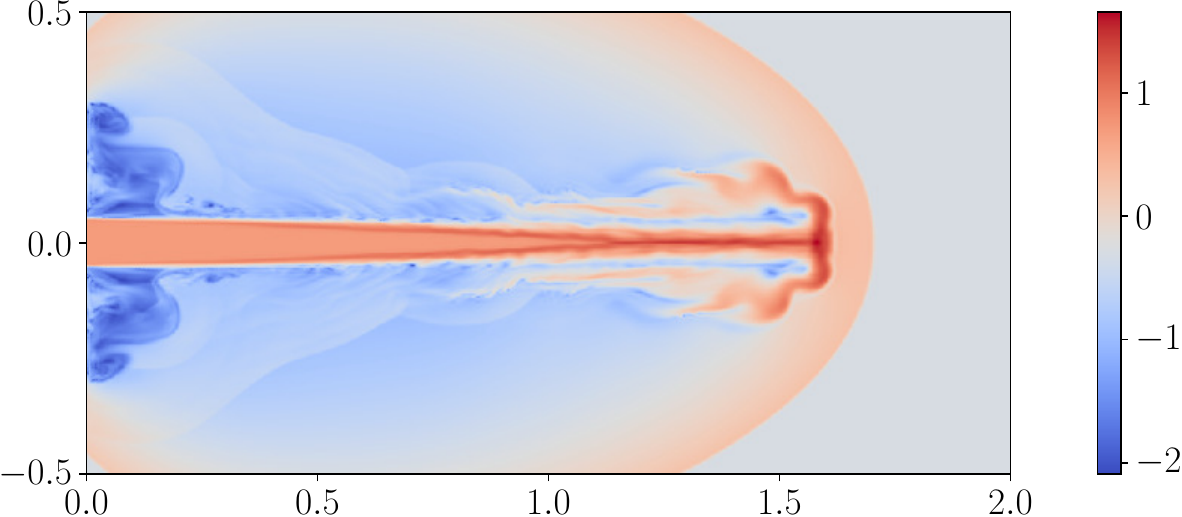}
		\end{subfigure}
		\begin{subfigure}[b]{0.48\textwidth}
			\centering
			\includegraphics[width=\linewidth]{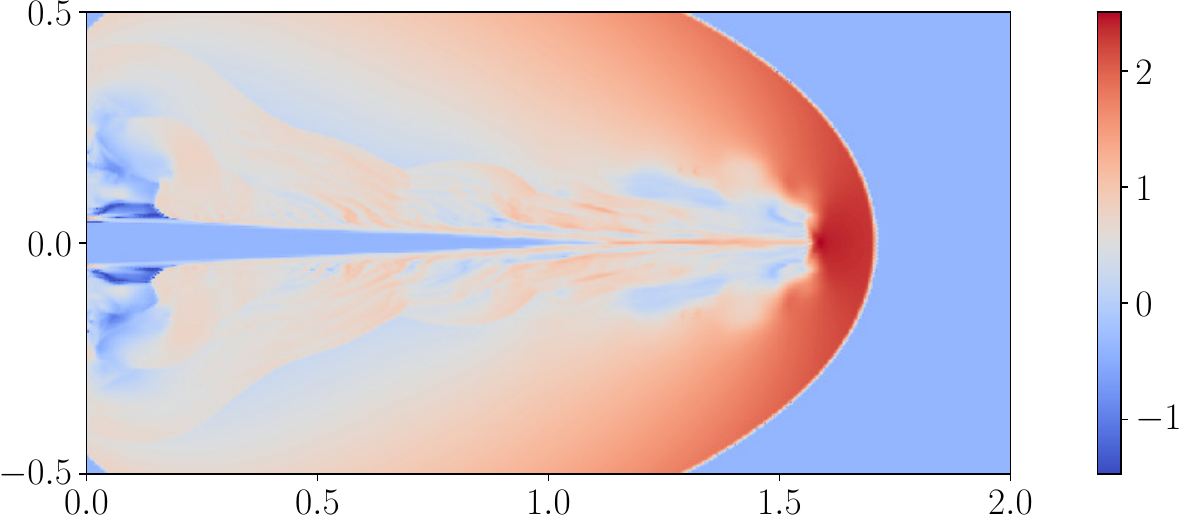}
		\end{subfigure}
		\vspace{5pt}
		
		\begin{subfigure}[b]{0.48\textwidth}
			\centering
			\includegraphics[width=\linewidth]{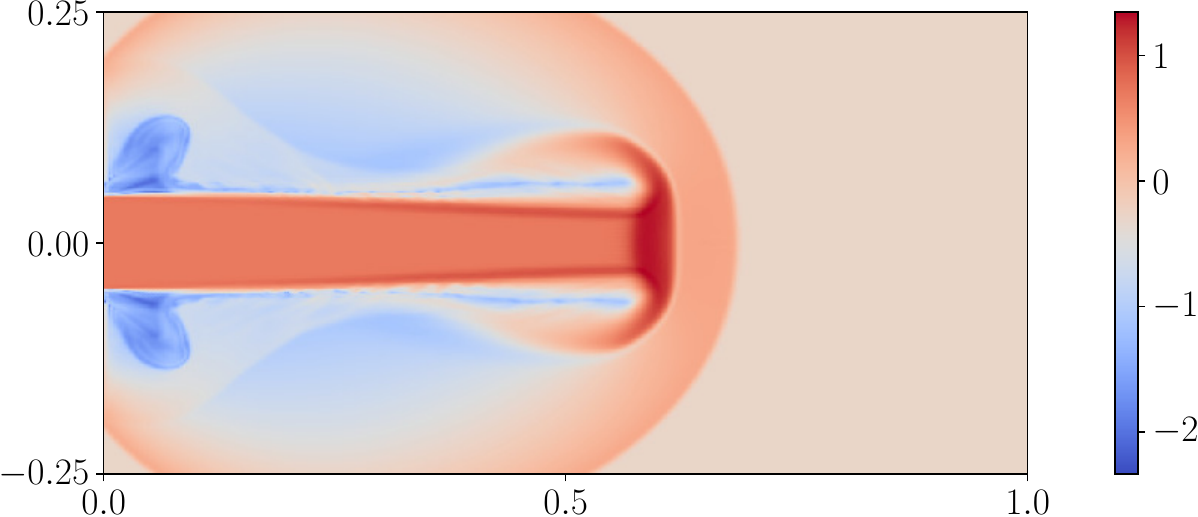}
		\end{subfigure}
		\begin{subfigure}[b]{0.48\textwidth}
			\centering
			\includegraphics[width=\linewidth]{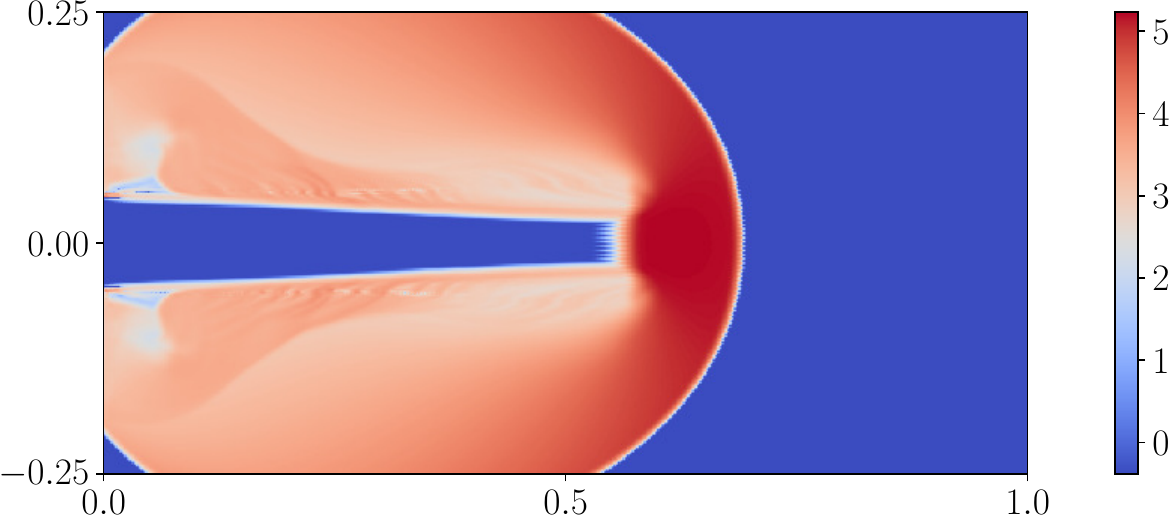}
		\end{subfigure}
		\caption{\Cref{ex:2d_jet}, the Mach $80$ jet (top row) and Mach $2000$ jet (bottom row).
		The logarithm of density (left) and pressure (right) obtained with the BP limitings and shock sensor-based limiting ($\kappa=1$ and $10$ for the Mach 80 and 2000 cases, respectively) on the uniform $400\times200$ mesh.}
		\label{fig:2d_jet}
	\end{figure}
\end{example}

\section{Conclusion}\label{sec:conclusion}
In the active flux (AF) methods, it is pivotal to design suitable point values update at cell interfaces, to achieve stability and high-order accuracy.
The Jacobian splitting (JS) based scheme for the point value update may suffer from the transonic issue and mesh alignment issue.
The flux vector splitting (FVS) for the point value update proposed in our previous work \cite{Duan_2024_Active_SJoSC} was extended to the 2D case in this work, which maintains the continuous reconstruction as the original AF methods, and offers a natural and uniform remedy to both the transonic issue and mesh alignment issue, thus a good alternative to the JS.
To construct suitable AF methods for problems involving strong discontinuities, this paper has also extended the bound-preserving (BP) AF methods \cite{Duan_2024_Active_SJoSC} to the 2D case.
The convex limiting and scaling limiter were employed for the cell average and point value, respectively, by blending the high-order AF methods with the first-order local Lax-Friedrichs (LLF) or Rusanov methods.
For scalar conservation laws, the blending coefficients were determined based on the global or local maximum principle,
while for the compressible Euler equations, they were obtained by enforcing the positivity of density and pressure.
The shock sensor-based limiting was proposed to further improve the shock-capturing ability.
The numerical results verified the accuracy, BP property, and shock-capturing ability of our BP AF methods.
Moreover, for the double Mach reflection and forward-facing step problems, the present AF method was able to capture comparable or better small-scale features compared to the third-order discontinuous Galerkin method with the TVB limiter on the same mesh resolution \cite{Cockburn_2001_Runge_JoSC}, while using fewer degrees of freedom, demonstrating the efficiency and potential of our FVS-based BP AF method for high Mach number flows.


\section*{Acknowledgement}

JD was supported by an Alexander von Humboldt Foundation Research fellowship CHN-1234352-HFST-P. CK and WB acknowledge funding by the Deutsche Forschungsgemeinschaft (DFG, German Research Foundation) within \textit{SPP 2410 Hyperbolic Balance Laws in Fluid Mechanics: Complexity, Scales, Randomness (CoScaRa)}, project number 525941602.
We acknowledge helpful discussions with Praveen Chandrashekar at TIFR-CAM Bangalore on the Ducros' shock sensor.

\newcommand{\etalchar}[1]{$^{#1}$}


%




\bibliographystyle{siamplain}
\bibliography{/Users/Junming/Research/references.bib}

\end{document}